\documentclass[11pt]{amsart}
\usepackage[margin=1in]{geometry}
\usepackage{amscd,amsthm,amsfonts,amssymb,amsmath,esint}
\usepackage{tikz-cd}
\usepackage[all]{xy}
\usepackage[hidelinks]{hyperref}
\usepackage{float}
\newcommand{\BC}{{\mathbb {C}}} 
 
\newcommand{\BG}{{\mathbb {G}}} 
 
\newcommand{\BK}{{\mathbb {K}}} 
 
 \newcommand{\BP}{{\mathbb {P}}}
\newcommand{\BQ}{{\mathbb {Q}}} 
\newcommand{\BS}{{\mathbb {S}}}

 \newcommand{\BZ}{{\mathbb {Z}}}
\newcommand{\CA}{{\mathcal {A}}} 
\newcommand{\CC}{{\mathcal {C}}} \renewcommand{\CD}{{\mathcal {D}}}
 \newcommand{\CF}{{\mathcal {F}}}
 \newcommand{\CH}{{\mathcal {H}}}
\newcommand{\CI}{{\mathcal {I}}} 
\newcommand{\CK}{{\mathcal {K}}} 
 \newcommand{\CN}{{\mathcal {N}}}
\newcommand{\CO}{{\mathcal {O}}} \newcommand{\CP}{{\mathcal {P}}}

\newcommand{\CU}{{\mathcal {U}}}

\newcommand{\fa}{{\mathfrak{a}}} 
 
\newcommand{\fe}{{\mathfrak{e}}} 
\newcommand{\fg}{{\mathfrak{g}}} 
 
 \newcommand{\fl}{{\mathfrak{l}}}
 \newcommand{\fn}{{\mathfrak{n}}}
\newcommand{\fo}{{\mathfrak{o}}} \newcommand{\fp}{{\mathfrak{p}}}
 
\newcommand{\fs}{{\mathfrak{s}}} \newcommand{\ft}{{\mathfrak{t}}}
\newcommand{\fu}{{\mathfrak{u}}}

\newcommand{\aut}{\mathfrak{aut}}

\newcommand{\ad}{{\mathrm{ad}}}

\newcommand{\Aut}{{\mathrm{Aut}}}

\newcommand{\Coker}{{\mathrm{Coker}}}

\newcommand{\End}{{\mathrm{End}}}

\newcommand{\Hom}{{\mathrm{Hom}}}

\newcommand{\id}{{\mathrm{id}}}
\renewcommand{\Im}{{\mathrm{Im}}}

\newcommand{\Ker}{{\mathrm{Ker}}}

\newcommand{\Lie}{{\mathrm{Lie}}}

\renewcommand{\mod}{\ \mathrm{mod}\ }

\newcommand{\Sym}{{\mathrm{Sym}}}

  \newcommand{\Supp}{{\mathrm{Supp}}}

\theoremstyle{plain}
\newtheorem{thm}{Theorem}[section]
\newtheorem{cor}[thm]{Corollary}
\newtheorem{lem}[thm]{Lemma}

\newtheorem{ques}[thm]{Question}
\newtheorem*{ques*}{Question}
\newtheorem{prop}[thm]{Proposition}

\newtheorem{thmA}{Theorem}

\newtheorem{corD}{Corollary}

\newtheorem{thmB}{Theorem}

\newtheorem{thmC}{Theorem}

\newtheorem{defn}[thm]{Definition}

\newtheorem*{HLcriterion}{Spencer condition}
\newtheorem{lem-defn}[thm]{Lemma-Definition}
\newtheorem*{Notation}{Notation}
\theoremstyle{remark}

\newtheorem{example}[thm]{Example}

\newcommand{\hfg}{\hat{\mathfrak{g}}}
\numberwithin{equation}{section}

\everymath{\displaystyle}
\title{On isotrivial cone structures of highest weight type}
\author[Y. Liu]{Yingqi Liu}
\address{Center for Complex Geometry, Institute for Basic Science (IBS), Daejeon 34126, Republic of Korea}
\email{yqliu@ibs.re.kr}
\begin{document}
	\date{July 2026}
    \begin{abstract}
	We study the projective-geometric criterion of Hwang–Li for isotrivial cone structures admitting characteristic conic connections. We give a complete classification of flag varieties satisfying this criterion, and in the process characterize flag varieties with an injective Gaussian map. As an application, we classify germs of minimal rational curves whose VMRT at a general point is isomorphic to that of a non-homogeneous smooth projective symmetric variety of Picard number one, apart from two exceptional cases.
	\end{abstract}
	\maketitle	
	%\tableofcontents
\section{Introduction}
\subsection{Main results}
\subsubsection*{Gaussian maps for flag varieties}\vspace{-0.5em}
Let $Z\subset \BP U$ be a smooth complex projective variety. Wahl introduced the Gaussian map
$$
\psi:\wedge^2U^\vee \longrightarrow H^0(Z,\Omega_Z(2)), \qquad l\wedge l' \longmapsto dl\otimes l'-l\otimes dl',
$$
which has been extensively studied since his seminal works~\cite{Wah90,Wah97}. In particular, the surjectivity of the Gaussian map for flag varieties (Definition~\ref{highwvar}) was conjectured in~\cite{Wah91} and later proved by Kumar~\cite{Kum92}.
 On the other hand, as observed in~\cite[Proposition~1.8]{Wah97}, the Gaussian map has a nontrivial kernel whenever the embedding line bundle is sufficiently positive. An effective version was recently obtained in~\cite[Theorem~5.5]{Cho25}. It is therefore natural to ask which projective varieties have an \emph{injective} Gaussian map \cite[Question~5.8.1]{Wah92}, a condition that we call \emph{tangential non-degeneracy} (Definition~\ref{tang_nond_defn}); see \cite[Theorem~1.3]{Gav05} for results on curves. Our first result gives a complete classification for flag varieties.

\begin{thmA}(Theorem~\ref{class_tang})\label{thm_A}
	Let
	$
	Z\subset\BP U
	$
	be a flag variety. Then the Gaussian map is injective if and only if
	$
	Z
	$
	is one of the following:
\begin{itemize}
	\item[(1)]
	the VMRT of an irreducible Hermitian symmetric space, namely the Segre variety
	$
	\BP^k\times\BP^l,
	$
	the Veronese variety
	$
	v_2(\BP^n)\subset \BP(\Sym^{2}\BC^{n+1}),
	$
	or one of the following irreducible Hermitian symmetric spaces of rank at most
	two, each in its minimal embedding:
	\[
	\BP^n\ (n\ge 1),\qquad
	Gr(2,n)\ (n\ge 5),\qquad
	\BQ^n\ (n\ge 3),\qquad
	\BS_5,\qquad
	E_6/P_1,
	\]
	where \(\BS_5=OG_+(5,10)\) is the \(10\)-dimensional spinor variety;
	
	\item[(2)]
	the adjoint variety
	$
	X_{\mathrm{ad}}(\fg)\subset \BP\fg
	$
	of a simple Lie algebra \(\fg\) (Definition~\ref{ad_coad});
	
	\item[(3)]
	the symplectic Grassmannian \(Gr_\omega(2,2n)\, (n \ge 3)\)  or  \(F_4/P_4\), each in its minimal embedding.
\end{itemize}
\end{thmA}
 The proof splits according to whether the highest weight $\lambda$ is fundamental. In the non-fundamental case, we adapt the method of tensors of linear forms developed in \cite{Cho25}. In the fundamental case, we use the description of the image of the Gaussian map in \cite{Kum92} to reduce the problem to a root-theoretic criterion. We note that these varieties are known to have vanishing third fundamental forms, which is also sufficient for tangential non-degeneracy by \cite[Proposition~1.3.1]{HM99}.

\subsubsection*{Isotrivial cone structures}
Our interest in Gaussian maps is motivated by the study of isotrivial cone structures with characteristic conic connections. Let $M$ be a complex manifold and $Z\subset\BP U$ a smooth projective variety with $\dim U=\dim M$. A cone structure on $M$ is a closed submanifold $\CC\subset\BP TM$
submersive over $M$, and it is called $Z$-isotrivial if every fiber
$\CC_x\subset\BP T_xM$ is projectively isomorphic to $Z\subset\BP U$.

An important source is provided by VMRT-structures associated with minimal rational curves on uniruled projective manifolds; see Definition~\ref{vmrt_defn}. Under mild assumptions, their local equivalence classes strongly constrain the global geometry of the underlying variety~\cite{HM01}, and every VMRT-structure naturally carries a characteristic conic connection~\cite{HM04}. Isotrivial VMRT-structures occur, for instance, on quasi-homogeneous
projective manifolds covered by lines and on certain
linear sections of flag varieties;
see~\cite{Mok08,HH08,LH21,HN22} for studies in the homogeneous
setting, and~\cite{Hwa26} for a comprehensive survey. The following problem combines~\cite[Question~2.8]{Hwa26} and~\cite[Question~1.4]{FH18a}; see Definition~\ref{def_symflat} for the notions of local flatness and local symmetry.

\begin{ques}\label{gene_ques}
Fix a positive-dimensional smooth projective subvariety $Z\subset\BP U.$
\begin{enumerate}
\item Let $\CC\subset\BP TM$
be a $Z$-isotrivial cone structure equipped with a characteristic conic connection. Is $\CC$ locally flat \textup{(}resp.\ locally symmetric\textup{)}? Can one classify its local equivalence class?

\item Suppose moreover that $\CC$ is the restriction to an open subset
$M\subset X$ of a $Z$-isotrivial VMRT-structure on a uniruled projective manifold $X$. Can one classify such VMRT-structures and the germ-equivalence classes of the corresponding smooth minimal rational curves?
\end{enumerate}
\end{ques}
Following the approach of~\cite{FH18a,LH24}, we study this problem via the associated $G$-structure, namely the principal bundle with structure group
$\Aut(\hat Z).$ A key result of~\cite[Theorem~1.3]{LH24} provides a projective-geometric criterion ensuring that, for any $Z$-isotrivial cone structure equipped with a characteristic conic connection, the associated $G$-structure admits a \emph{torsion-free} principal connection (see Theorem~\ref{tors_prin}). The criterion consists of three conditions. While the first two can often be verified directly, the main difficulty lies in the third, which is formulated in terms of the Spencer homomorphism recalled below.

\begin{HLcriterion}\label{Spencer_condtition}
	Let $Z \subset \BP U$ be a smooth projective subvariety, $\hat{Z} \subset U$ its affine cone, and $\aut(\hat{Z})$ the Lie algebra of infinitesimal automorphisms. Denote by
	\[
	\Xi_Z
	=
	\left\{
	\sigma\in \Hom(\wedge^2U,U)
	\ \middle|\
	\sigma(\hat z,\widehat T_zZ)\subset \widehat T_zZ
	\text{ for all }z\in Z
	\right\}.
	\]
	Let the Spencer homomorphism be defined by
	\begin{equation*}
	\partial:
	\Hom(U,\fa \fu \ft(\widehat Z))
	\longrightarrow
	\Hom(\wedge^2U,U),
	\qquad
	\partial h(u_1,u_2)
	=
	h(u_1)(u_2)-h(u_2)(u_1).
	\end{equation*}
	The condition (iii) in \cite[Theorem 1.3]{LH24} asks whether
	$
	\Xi_Z\subset \Im(\partial).
	$
\end{HLcriterion}
Hwang--Li verified the criterion when $Z$ is an adjoint variety of a simple Lie algebra~\cite[Theorem~1.5]{LH24}. The main goal of this paper is to classify all flag varieties satisfying the criterion. The first two conditions hold for arbitrary flag varieties, see Proposition \ref{surj_Psi_}. For the Spencer condition, we reduce it to a problem on the Gaussian map.
\begin{thmB}(Theorem \ref{reduction})\label{thm_B}
		Let \(Z\subset \BP U\) be a flag variety and let $\psi$ be the Gaussian map. If the Spencer condition holds, then one of the following occurs:
		\begin{itemize}
			\item[(i)] \(Z\subset \BP U\) is tangentially non-degenerate; equivalently the Gaussian map $\psi$ is injective;
			\item[(ii)] \(Z\) is isomorphic to a Hermitian symmetric space and $Ker(\psi) \cong \BC$ is one-dimensional.
		\end{itemize}
\end{thmB}
We refer to Corollary~\ref{cor_classify} for a more general reduction beyond
the setting of flag varieties. Combined with Theorem~\ref{thm_A}, this yields
the main result of the paper. The subadjoint case was recently classified in~\cite{CN26}.
\begin{thmC}\label{main_thm}
	Let
	$
	Z\subset \BP U
	$
	be a flag variety. The Spencer condition $\Xi_{Z} \subset \operatorname{Im}(\partial)$ holds if and only if $Z$ is one of the following:
	\begin{enumerate}
		\item the VMRT of an irreducible Hermitian symmetric space under its minimal embedding;
			\item
			a subadjoint variety except for
			$
			\BP^1\times v_2(\BP^1)\subset
			\BP(\BC^2\otimes \Sym^{2}\BC^2),
			$
			namely
			\[
			\BP^1\times\BQ^m\ (m\ge2),\quad
			v_3(\BP^1),\quad
			E_7/P_7,\quad
			\BS_6,\quad
			Gr(3,6),\quad
			Lag(3,6),
			\]
			under their minimal embeddings, where $\BS_{6}=OG_{+}(6,12)$ is the 15-dimensional spinor variety;
		\item
an adjoint variety
$
X_{\mathrm{ad}}(\fg)\subset \BP\fg,
$
where \(\fg\) is not of type \(A_n\) for \(n \neq 2\);
		
		\item
		$Gr_{\omega}(2,6)$ or
		$F_4/P_4$
		under their minimal embeddings.
	\end{enumerate}
\end{thmC}
\subsection{Application to Question \ref{gene_ques}}
The rigidity of torsion-free $G$-structures is particularly strong when
$Z\subset\BP U$ is a flag variety. Building on the classification of~\cite{ML99}, Hwang--Li proved in~\cite[Theorem~4.6]{LH24} that every flag variety satisfying the criterion of~\cite[Theorem~1.3]{LH24} gives rise to \emph{locally symmetric} $Z$-isotrivial cone structures, apart from two exceptional families; see also Theorem~\ref{appp_homoge}. These two exceptional families coincide with families~\textup{(1)} and~\textup{(2)} of Theorem~\ref{main_thm}, except for the type-$B_3$ subadjoint variety, which is excluded by our classification.
In the more restrictive setting of isotrivial VMRT-structures, local flatness was proved by Mok~\cite{Mok08} for family~\textup{(1)}, and by Hwang–Li~\cite{HL22} for family~\textup{(2)}, with the exception of the twisted cubic $v_3(\BP^1)$.

Family~\textup{(3)} was completely treated in~\cite[Theorem~1.7]{LH24}, which classifies germs of minimal rational curves whose VMRT is an adjoint variety for $\fg$ not of type $A_{n\neq2}$. The remaining adjoint varieties of type $A_{n\ge3}$ are excluded by ~\cite[Example~8.8]{LH24}. In family~\textup{(4)}, although
$
\operatorname{Gr}_{\omega}(2,2n)
$
with $n\ge4$ occur in Theorem~\ref{thm_A}\textup{(3)}, they are excluded by Example~\ref{nonhomospin}. Both constructions arise from the general hyperplane sections of flag varieties studied in~\cite{BFM20}, whose isotrivial VMRT-structures are not locally homogeneous. 

The only genuinely new cases are the two remaining varieties in family~\textup{(4)}, both arising as VMRTs of smooth projective symmetric varieties of Picard number one.

\begin{defn}[{\cite[Section 26]{Ti11}}]
	Let $\widetilde G$ be a simply connected reductive group with an involution $\nu$, and let
	\[
	\widetilde G^\nu\subset G\subset
	N_{\widetilde G}(\widetilde G^\nu).
	\]
	A normal $\widetilde G$-variety containing an open orbit isomorphic to $\widetilde G/G$ is called a \emph{symmetric variety}.
\end{defn}

By the classification of~\cite{Ruz11}, there are six non-homogeneous smooth projective symmetric varieties of Picard number one. In each case, the lines in the minimal embedding form a distinguished family of minimal rational curves whose VMRT at a general point is a flag variety; see Table~\ref{Table3}. Together with the two adjoint cases in family~(3), this gives the following corollary.

\begin{corD}(Corollary~\ref{cor_vmrt})\label{cor_D}
	Let $X$ be a non-homogeneous smooth projective symmetric variety of Picard number one, not isomorphic to $X_4$ or $X_5$ in Table~\ref{Table3}, and let $Z\subset\BP U$ be the flag variety isomorphic to the VMRT of $X$ at a general point. Then every $Z$-isotrivial cone structure admitting a characteristic conic connection is locally symmetric, and hence locally homogeneous.
	
	Moreover, let $X'$ be a smooth projective variety of dimension $\dim(X)$ equipped with a family $\CK$ of minimal rational curves such that, over some open subset $M\subset X'$, the restricted VMRT-structure $\CC|_M$ is $Z$-isotrivial. Then:
	\begin{itemize}
		\item[(1)]
		A neighborhood of a general member of $\CK$ is biholomorphic either to a neighborhood of a general minimal rational curve on $X$, or to the locally flat model in Example~\ref{flatmodel}.
		
		\item[(2)]
		If $b_2(X')=1$, then $X'$ is quasi-homogeneous, i.e.  $\Aut(X')$ has an open orbit.
	\end{itemize}
\end{corD}
Recently, several general studies on VMRT-structures of smooth projective
symmetric varieties have appeared; see, in particular,
\cite{BKP26,HL26}. It would be interesting to extend the computations of the
present paper toward a more complete answer to Question~\ref{gene_ques} for
isotrivial cone structures arising from symmetric varieties.

\subsection{Idea of the proof}
We briefly outline the proof of Theorem~\ref{main_thm}. The Spencer condition $\Xi_Z\subset \Im(\partial)$
is first reformulated as the surjectivity of the restricted map $\partial:\partial^{-1}(\Xi_Z)\longrightarrow \Xi_Z.$
We then use the natural filtration
\[
\Xi_Z\supset \Xi'_Z\supset \Xi^0_Z
\]
introduced in~\cite{FH18a}, and pull it back via the Spencer homomorphism. This reduces the problem to three successive surjectivity questions, corresponding to the maps
$\partial^0,\partial',$ and $\partial$ in \eqref{diag_Spencer}.

The key observation is that for flag varieties the lowest piece of the filtration is controlled by tangential non-degeneracy. Consequently the surjectivity of \(\partial^0\) is shown
to be equivalent to one of the cases in Theorem~\ref{thm_B};
see Theorem~\ref{refor_reduction}.  The first case is classified in Theorem~\ref{thm_A}, while the second gives precisely the subadjoint varieties; see Proposition~\ref{partial0_surj_classify}.
The main ingredient is the more general Corollary~\ref{cor_classify}, which gives a similar reduction whenever the isotropy Lie algebra acts faithfully on the affine tangent space at a general point. For flag varieties, this faithfulness is verified in Lemmas~2.11 and~2.12.

It remains to pass from $\partial^0$
to the full Spencer condition. This is done by evaluating $\Xi_Z$
and $\Xi'_Z$ along lines on $Z,$ following~\cite{FH18a,LH24} and applying the results in \cite{LM03}; see
Lemmas \ref{esti_parital'} and \ref{esti_partial}.
 We first prove that, once $\partial^0$
is surjective, the next map $\partial'$
is also surjective. The final step is to analyze the remaining quotient
$\Xi_Z/\Xi'_Z.$ This is reduced to a comparison between the image of the isotropy representation and the infinitesimal automorphism algebra of the VMRT of
$Z$ at a general point; see Proposition \ref{main_432} and Corollary \ref{suff_condition_surj}. This comparison singles out exactly the additional cases
$Gr_{\omega}(2,6)$
and
$
F_4/P_4,
$
and excludes
$
Gr_{\omega}(2,2n)
$
for
$
n\geq 4
$
as well as the type-$B_3$ subadjoint variety. This completes the proof of Theorem~\ref{main_thm}.

\subsection{Organization of the paper}
Section~2 collects the necessary preliminaries. In Section~3, we classify tangentially non-degenerate flag varieties, and in Section~4, we study the filtration associated with the Spencer homomorphism and prove Theorem~B. In Section~5, we verify the Hwang--Li criterion for flag varieties and prove Theorem~C. Finally, Section~6 develops the applications to isotrivial VMRT-structures and proves Corollary~D.

\begin{Notation}
	We use the following notation throughout the paper.
	
	\begin{enumerate}
		\item
		Let \(G\) be a semisimple algebraic group with Lie algebra
		\(\fg=\Lie(G)\). Fix a Borel subgroup \(B\subset G\) and a maximal torus
		\(T\subset B\), with \(\ft=\Lie(T)\), and write
		$
		\fg=\fn\oplus\ft\oplus\fn^-
		$
		for the corresponding triangular decomposition. Let \(\Phi\) be the root
		system, with positive roots \(\Phi^+\) and simple roots
		\(\Delta=\{\alpha_1,\ldots,\alpha_n\}\). We denote by \(\CD\) the Dynkin
		diagram of \(G\), by \(W\) its Weyl group, and by \(w_0\in W\) the longest
		element. Let \(\Lambda^+\) be the set of dominant weights, let
		\(\lambda_1,\ldots,\lambda_n\) be the fundamental weights, and set
		$
		\rho=\sum_{i=1}^n\lambda_i.
		$
		For each \(1\le i\le n\), fix root vectors
		\(e_i\in\fg_{\alpha_i}\) and \(f_i\in\fg_{-\alpha_i}\) such that
		\[
		\langle e_i,f_i\rangle
		=
		\frac{2}{\langle\alpha_i,\alpha_i\rangle},
		\]
		where \(\langle\, ,\,\rangle\) denotes the Killing form on \(\fg\).
		
		For a root \(\alpha\in\Phi\), let \(m_{\alpha_i}(\alpha)\) denote the
		coefficient of \(\alpha_i\) in \(\alpha\), and set
		\[
		\Supp(\alpha)
		=
		\{\alpha_i\in\Delta\mid m_{\alpha_i}(\alpha)\neq 0\}.
		\]
		We denote by \(\CD_\alpha\subset\CD\) the subdiagram generated by
		\(\Supp(\alpha)\).
		
		\item
		Let \(Z\subset\BP U\) be a smooth projective subvariety and
		\(\hat Z\subset U\) its affine cone. For \(z\in Z\), let
		\(\BC\hat z\subset U\) be the corresponding line and
		\(\widehat T_zZ\subset U\) the affine tangent space. Then
		\[
		T_zZ
		\simeq
		\Hom(\BC\hat z,\widehat T_zZ/\BC\hat z).
		\]
		We denote by \(\fa\fu\ft(\hat Z)\) the Lie algebra of infinitesimal
		automorphisms of \(\hat Z\):
		\begin{equation}\label{fautdef}
			\fa\fu\ft(\hat Z)
			=
			\left\{
			A\in\End(U)
			\ \middle|\
			A(\hat z)\in\widehat T_zZ
			\text{ for all }z\in Z
			\right\}.
		\end{equation}
		
		Let \(\widehat T\subset U\otimes\CO_Z\) be the affine tangent bundle of
		\(Z\). It fits into the exact sequence
		\begin{equation}\label{eulerseq}
			0
			\longrightarrow
			\CO_Z(-1)
			\longrightarrow
			\widehat T
			\longrightarrow
			T_Z(-1)
			\longrightarrow
			0.
		\end{equation}
		After twisting by \(\CO_Z(1)\) and taking global sections, we obtain
		\begin{equation}\label{eulerseq_02}
			0
			\longrightarrow
			\BC\,\id_U
			\longrightarrow
			\fa\fu\ft(\widehat Z)
			\longrightarrow
			H^0(Z,T_Z),
		\end{equation}
		whose rightmost map is surjective whenever \(H^1(Z,\CO_Z)=0\).
	\end{enumerate}
\end{Notation}
\section*{Acknowledgements}
The author is grateful to Jun-Muk Hwang for suggesting the problem and for many valuable discussions throughout the development of this work. He also thanks Qifeng Li for fruitful discussions and for explaining the results of~\cite{HL22,LH24}, Katharina Neusser for helpful conversations that led to the correction of an earlier calculation, Junho Choe and Baohua Fu for their comments. This work is supported by the Institute for Basic Science (IBS-R032-D1).
\section{Preliminaries}

\subsection{Isotrivial cone structures and the Spencer criterion}

We briefly recall the ingredients from \cite[Section~4]{LH24} that will be used in this paper.

Let $Z \subset \BP U$ be a smooth projective subvariety, and let $\CC \subset \BP TM$ be a $Z$-isotrivial cone structure. Set the structure group $\hat{G} \subset GL(U)$ and its Lie algebra $\hfg \subset \fg \fl(U)$ by
\[
\hat{G}=\operatorname{Aut}(\widehat Z)
=\{g\in \operatorname{GL}(U)\mid g(\widehat Z)=\hat Z\},
\qquad
\hfg=\mathfrak{aut}(\hat Z)=\operatorname{Lie}(\hat{G}).
\]
The associated \(G\)-structure is the principal \(G\)-subbundle
$
\mathcal P\subset \mathbf F M
$
of the frame bundle defined by
\[
\mathcal P_x
=
\{\epsilon\in \operatorname{Isom}(U,T_xM)
\mid
\epsilon(\widehat Z)=\widehat{\mathcal C}_x\}.
\]
On the principal bundle $\CP$ one can define a principal connection $\CH \subset T_{\CP}$ which induces a conic connection $\CF \subset T_{\CC}$ on \(\mathcal C\). Conversely, \cite{LH24} studies when a conic connection on \(\mathcal C\), especially a \emph{characteristic conic connection} on $\CC$, induces a \emph{torsion-free principal connection} on \(\mathcal P\).

The following is the criterion of \cite[Section~4]{LH24} in the form needed in this paper. 
\begin{thm}\label{tors_prin}
	Let \(\mathcal C\subset \mathbb P TM\) be a \(Z\)-isotrivial cone structure, and let \(\mathcal P\subset \mathbf F M\) be the associated \(G\)-structure. Assume that:
	\begin{itemize}
		\item[(i)] \(H^0(Z,\mathcal O(1))=U^\vee\);
		\item[(ii)] the natural map
		\[
		\Psi:
	 \hfg \otimes U^\vee
		\longrightarrow
		H^0(Z,TZ\otimes\mathcal O(1))
		\]
		induced by the infinitesimal action of $\mathfrak{aut}(\widehat Z)$ on \(Z\) is surjective.
	\end{itemize}
	Then every conic connection on \(\mathcal C\) is locally induced by a principal connection on \(\mathcal P\).
	
	If, in addition,
	\[
	\tag{iii}
	\Xi_Z\subset \operatorname{Im}(\partial),
	\]
	then every characteristic conic connection on \(\mathcal C\) induces a torsion-free principal connection on \(\mathcal P\).
\end{thm}
The existence of a torsion-free principal connection on the associated \(G\)-structure imposes strong restrictions on the corresponding cone structure. 
We recall the following terminology.

\begin{defn}\label{def_symflat}
	Let \(\mathcal C\subset \mathbb P TM\) be a \(Z\)-isotrivial cone structure.
	
	\begin{itemize}
		\item[(1)] We say the cone structure is \emph{locally symmetric} if the associated $G$-structure $\CP$ is locally symmetric (cf.  \cite[Definition 2.14]{LH24}); 
		\item[(2)]
		We say that \(\mathcal C\) is \emph{locally flat} if the 
		associated $G$-structure $\CP$ is locally flat. Equivalently, for every \(x\in M\), there exists a neighborhood \(x\in O\subset M\) with suitable local holomorphic coordinates such that the induced trivialization
		\[
		\mathbb P TO\simeq \mathbb P U\times O
		\]
		identifies \(\mathcal C|_O\) with the trivial cone structure \(Z\times O\);
		\item[(3)]
		We say that \(\mathcal C\) is \emph{locally homogeneous} if, for any two points \(x,x'\in M\), there exist neighborhoods \(x\in O\subset M\), \(x'\in O'\subset M\), and a biholomorphism
		\(f:O\to O'\)
		such that the induced map
		\[
		df:\mathbb P TO\longrightarrow \mathbb P TO'
		\]
		sends \(\mathcal C|_O\) to \(\mathcal C|_{O'}\).
	\end{itemize}
\end{defn}
 Here for (2) we also use the definition as in \cite[Definition 1.1]{FH18a}. Next we recall the locally symmetric consequence from the preceding criterion. Denote the space of formal curvature tensors of \(\hat{\mathfrak g}\) by
\begin{equation}\label{form_curv}
\BK(\hat{\mathfrak g})
=
\left\{
R\in \operatorname{Hom}(\wedge^2U,\hat{\mathfrak g})
\ \middle|\
R(u,v)w+R(v,w)u+R(w,u)v=0
\text{ for all }u,v,w\in U
\right\}.
\end{equation}

\begin{thm}[{\cite[Theorem~4.6 and Corollary~4.11]{LH24}}]\label{appp_homoge}
	Let \(Z\subset \mathbb P U\) be a flag variety (cf. Definition \ref{highwvar}).
	\begin{itemize}
		\item[(1)]
		Assume that conditions \textup{(i)--(iii)} of Theorem~\ref{tors_prin} hold, and that \(Z\) is not one of the varieties excluded in \cite[Theorem~4.6]{LH24}. Then every \(Z\)-isotrivial cone structure admitting a characteristic conic connection is locally symmetric, and hence locally homogeneous.
		
		\item[(2)]
		Assume, in addition, that \(\dim \BK(\hat{\mathfrak g})=1\). Then any two non-locally-flat \(Z\)-isotrivial cone structures equipped with characteristic conic connections are locally equivalent. More precisely, if \(\mathcal C\subset \mathbb P TM\) and \(\widetilde{\mathcal C}\subset \mathbb P T\widetilde M\) are such cone structures, then for any \(x\in M\) and \(\widetilde x\in\widetilde M\), there exist neighborhoods \(x\in O\subset M\), \(\widetilde x\in\widetilde O\subset\widetilde M\), and a biholomorphism \(\varphi:O\to \widetilde O\) such that the induced map
		\[
		d\varphi:\mathbb P TO\longrightarrow \mathbb P T\widetilde O
		\]
		sends \(\mathcal C|_O\) to \(\widetilde{\mathcal C}|_{\widetilde O}\).
	\end{itemize}
\end{thm}
\subsection{VMRT-structures}
In this section we briefly recall the relevant terminology and results needed to explain the implication of Theorem~\ref{appp_homoge}. These notions will be also used in Section~5.3 to study the Spencer condition.

\begin{defn}\label{vmrt_defn}
	Let $RatCurves(X)$ be the space of rational curves on a smooth projective variety $X.$ An irreducible component
	$\CK \subset RatCurves(X)$ is called a \emph{family of minimal rational curves} on $X$ if, for a general point
	$x\in X,$ the subscheme $\CK_{x}\subset \CK$ consisting of members of $\CK$ passing through $x$ is nonempty and projective. Fix such a family $\CK.$ Let $\CC^{+}\subset \BP TX$ be the closure of the union of tangent directions to members of $\CK.$
	Then $\CC^{+}$ has a unique irreducible component
	$\CC\subset \BP TX$ dominating $X.$ We call $\CC$
	the VMRT-structure associated with
	$\CK.$ For a point $x\in X,$ the fiber $\CC_{x}\subset \BP T_{x}X$ is called the VMRT at $x$ associated with the family $\CK.$
\end{defn}
The following propositions collect important properties of VMRT-structures. We refer to \cite{HM99} for more details.
\begin{prop}
	Keep the notation as above. Let $x\in X$
	be a general point, and let $l\in \CK_{x}$
	be a general minimal rational curve through $x,$
	given by $f:\BP^{1}\to X,$ with $f(o)=x$.
	Then
	\begin{equation}\label{unbendable}
		f^{*}T_{X}
		\cong
		\CO_{\BP^{1}}(2)
		\oplus
		\CO_{\BP^{1}}(1)^{\oplus s}
		\oplus
		\CO_{\BP^{1}}^{\oplus n-s-1},
	\end{equation}
	where
	$s=\dim(\CC_{x})$ and $n=\dim(X).$
	Moreover, if $s\ge 1,$ then $l$ is smooth at $x,$ where the tangent direction of $l$ is identified with $T_{x}l=\CO(2)_{o},$ and the affine tangent space of the VMRT $\CC_{x}$ at $[T_{x}l]$ is
	\begin{equation}\label{unbendable_}
		\widehat{T}_{[T_{x}l]}\CC_{x}X
		=
		(\CO(2)\oplus \CO(1)^{\oplus s})_{o}.
	\end{equation}
\end{prop}
An important source of minimal rational curves is given by covering families of lines.
\begin{prop}
	Let $X\subset \BP^{N}$
	be a nonsingular projective variety covered by lines. Then any irreducible component of the family of lines covering
	$X$ is a minimal rational component. Moreover, the VMRT
	$\CC_{x}\subset \BP T_{x}X$ at a general point $x\in X$
	is nonsingular.
\end{prop}

For VMRT-structures that are smooth at general points, one obtains a natural characteristic conic connection as follows.

\begin{prop}[{\cite[Proposition~8]{HM04}}]\label{vmrtischar}
	Keep the notation of Definition~\ref{vmrt_defn}. Assume that the VMRT $\CC_{x}X\subset \BP T_{x}X$
	is smooth for a general point $x\in X.$
	Equivalently, there exists a Zariski-open subset $M\subset X$ such that $\CC_{M}:=\CC|_{M}\subset \BP TM$ is a cone structure. Then the tangent directions to members of $\CK$
	determine a characteristic conic connection on $\CC_{M}.$
\end{prop}
This provides a powerful approach to the study of uniruled projective manifolds via characteristic conic connections, as developed in~\cite{HN22,LH24}. One of the key ingredients is the following Cartan--Fubini type extension theorem.

\begin{thm}[{\cite[Theorem~1.2]{HM01}}]\label{HM01}
	Let \(\CK\) and \(\widetilde{\CK}\) be families of minimal rational curves on smooth projective varieties \(X\) and \(\widetilde X\), respectively. Assume that the VMRT
	$
	\CC_x\subset \BP T_xX
	$
	is smooth, irreducible, and linearly non-degenerate at a general point \(x\in X\). Suppose that there exists a biholomorphic map
	$
	f:U\longrightarrow \widetilde U
	$
	between connected nonempty open subsets \(U\subset X\) and \(\widetilde U\subset \widetilde X\), such that the induced map
	$
	df:\BP TU\longrightarrow \BP T\widetilde U
	$
	sends \(\CC|_U\) biholomorphically onto \(\widetilde{\CC}|_{\widetilde U}\).
	
	Then there exist a general member \(C\subset X\) of \(\CK\), a general member \(\widetilde C\subset \widetilde X\) of \(\widetilde{\CK}\), neighborhoods \(C\subset O\subset X\) and \(\widetilde C\subset \widetilde O\subset \widetilde X\), and a biholomorphic map
	$
	f:O\longrightarrow \widetilde O
	$
	such that \(f(C)=\widetilde C\).
	
	If furthermore \(X\) and \(\widetilde X\) have Picard number one, then f extends to a biregular isomorphism
	$
	F:X\longrightarrow \widetilde X
	$
	with \(F|_U=f\).
\end{thm}

\subsection{Flag varieties}
In this paper we study isotrivial cone structures associated with irreducible $G$-structures, namely those for which
$
Z \subset \BP U
$
is a flag variety, also called a highest weight variety in~\cite{LH24}. We recall the basic definitions and facts concerning flag varieties.

\begin{defn}\label{highwvar}
	A projective subvariety
	$Z \subset \BP U$
	is called a \emph{flag variety} if
	$
	\Aut(\widehat{Z})
	$
	is a reductive group acting irreducibly on
	$
	U,
	$
	and
	$
	Z
	$
	is the unique closed orbit of the induced action of
	$
	\Aut(\widehat{Z})
	$
	on
	$
	\BP U.
	$
\end{defn}

In the sequel, we shall always describe a flag variety  $Z \subset \BP U$ as follows. The identity component of the automorphism group $G=\Aut^{0}(Z)$ is a semisimple group with Lie algebra
$\fg=\Lie(G),$ and $U=U_{\lambda}$ is the irreducible
$\fg$-module of highest weight $\lambda\in \Lambda^{+}.$
Then $Z$ can be written as
$Z=G/P,$ where $P\subset G$
is the stabilizer of $z=[u_{\lambda}]$, and $
u_{\lambda}\in \widehat{Z}
$ is a highest weight vector. 

\subsubsection{Parabolic grading and the isotropy action}

Let \(Z=G/P\subset\BP U_\lambda\) be a flag variety. Let \(L\subset P\) be the standard Levi subgroup and
\(R_u(P)\) the unipotent radical of \(P\). We denote by $\fp=\Lie(P)\subset \fg$ the parabolic subalgebra.
The parabolic grading of $\fg$ determined by $\fp$ is defined as follows. 
\begin{defn}
Write $\lambda=\sum_{i=1}^{n} a_{i}\lambda_{i}$ and set
$
\Delta_Z = \{\alpha_i \in \Delta \mid a_i \neq 0\},
$ and for each $k \in \BZ$, define
\[
\Phi_Z^{k} = \Big\{ \alpha = \sum_i m_i(\alpha)\alpha_i \in \Phi \ \Big|\
\sum_{\alpha_i \in \Delta_Z} m_i(\alpha) = k \Big\}.
\]
Then the grading is defined by
\[\label{para_grad}
\fg = \bigoplus_{k=-d}^{d} \fg_k, \qquad
\fg_{k} = \bigoplus_{\alpha \in \Phi_Z^{k}} \fg_{\alpha} \quad (k\not =0), \qquad
\fg_0 = \ft \oplus \bigoplus_{\alpha \in \Phi_Z^{0}} \fg_{\alpha} ,
\]
where we call $d$ the length of the grading.
\end{defn}
In particular $\fg_{0}$ is a reductive subalgebra and acts on each $\fg_{k}$. Moreover we have the identifications of sub-algebras:
$$
 \fp=\bigoplus_{k \ge 0} \fg_k, \quad   \Lie(L)=\fg_{0},  \quad 
\fn_{\fp}=\bigoplus_{k > 0} \fg_{k}, \quad  \fn_{\fp}^{-} = \bigoplus_{k < 0} \fg_{k},
$$
where $\fn_{\fp}=\Lie(R_{u}(P))$ is the nilpotent radical of $\fp$, and $\fn_{\fp}^{-}$ is a complement of $\fp$ in $\fg$. 

That $P$ fixes the point $[u_{\lambda}]$ gives its isotropy action on the (affine) tangent space.  The (affine) tangent space of $Z$ at the point $[u_{\lambda}]$ is identified as:

	\begin{equation}\label{iden}
		\hat{T}_{[u_\lambda]} Z=\BC u_\lambda\oplus \fn_{\fp}^{-}. u_\lambda,\qquad
		T_{[u_\lambda]}Z=Hom(\BC u_{\lambda}, \hat{T}_{[u_\lambda]}Z/\BC u_{\lambda})\cong \fg/\fp,
	\end{equation}
	where the last equality follows by evaluating the action of $\fg$ on $\BC u_{\lambda}$.  By taking the differential, the induced isotropy action of $\fp$ on the tangent space is known to be given by:
	\begin{equation}\label{phip_defn}
	\phi_{\fp}:\fp \longrightarrow End(\fg/\fp),
	\qquad
	\phi_{\fp}(x)(\overline{y})=\overline{[x,y]}.
	\end{equation}
The following results describe its kernel, which is very useful to our reduction step in Section 4.
\begin{lem}\label{ker_isotropy}
If $\fg$ is simple, then the kernel of $\phi_{\fp}$ is equal to $\fg_{d}$, which is an irreducible $P$-module.
\end{lem}
	\begin{proof}
That $\fg_{d} \subset \ker(\phi_{\fp})$ follows from
$
[\fg_{d}, \fg] \subset \fg_{\ge 0} = \fp,
$
since $d$ is the length of the grading. Moreover, $\fg_d$ is an irreducible $P$-module. Indeed, the unipotent radical of $P$ acts trivially on $\fg_d$, so the $P$-module structure factors through the Levi subgroup. A highest weight vector of $\fg_d$ is also a highest root vector of $\fg$, and hence is unique up to scalar. 

The map $\phi_{\fp}$ is $P$-equivariant. Hence its kernel decomposes into irreducible $\fg_{0}$-modules. Let $X$ be a $\fg_{0}$-highest weight vector in $\ker(\phi_{\fp})$. Then $X$ is either a root vector or a semisimple element in the Cartan subalgebra $\ft$. It suffices to show that $X$ must be a highest root vector which lies in $\fg_{d}$.

If $X$ is an element in $\ft$, then it implies that $X$ lies in the center of $\fg_{0}$, that is $\alpha_j(X)=0$ for all $\alpha_j \notin \Delta_Z$. Now take any $Y_{\alpha_i} \in \fg_{-\alpha_i}$ with $\alpha_i \in \Delta_Z$. Then
\[
\phi_{\fp}(X)(\overline{Y_{\alpha_i}}) = \alpha_i(X)\,\overline{Y_{\alpha_i}} = 0,
\]
which implies that $\alpha_i(X)=0$ and hence $X=0$.

Now suppose that \(X=X_{\alpha}\) is a root vector. Since \(X\) is a highest weight vector of \(\fp\) with respect to the \(\fg_0\)-action, the root \(\alpha\) is \(\fg_0\)-dominant. In particular, \(\alpha\) is a positive root belonging to \(\Phi_Z^k\) for some \(k\geq 0\). Assume that $\alpha$ is not the highest positive root. Since $\fg$ is simple, there exists a simple root $\alpha_i$ such that
$
\alpha' = \alpha + \alpha_i \in \Phi.
$
Moreover, since $X$ is $\fg_{0}$-highest, we may assume that $\alpha_i \in \Delta_Z$. In particular, $-\alpha' \in \Phi_{Z}^{-k-1}$ and
\[
-\alpha' + \alpha = -\alpha_i.
\]
Choose any nonzero $Y_{\alpha'} \in \fg_{-\alpha'}$. Then $[X, Y_{\alpha'}] \in \fg_{-\alpha_i}$ and is nonzero by~\cite[Proposition~8.4(d)]{Hum72}. Hence
\[
\phi_{\fp}(X)(\overline{Y_{\alpha'}}) = \overline{[X, Y_{\alpha'}]} \neq 0,
\]
a contradiction. Therefore $\alpha$ must be the highest root.
	\end{proof}
\begin{lem}\label{faithful_isotropy}
If $\fg$ is simple, then the action of $\fg_{d}$ on the affine tangent space $\hat{T}_{[u_{\lambda}]}Z$ is faithful.	
\end{lem}
\begin{proof}
Assume otherwise. Since the action is $\fp$-equivariant and
$\fg_{d}$
is an irreducible
$
\fp
$
-module, the action of
$
\fg_{d}
$
on
$
\hat{T}_{[u_{\lambda}]}Z
$
must be trivial. Let
$
X_{\alpha}\in \fg_{\alpha}
$
be a highest root vector, and choose
$
0\neq X_{-\alpha}\in \fg_{-\alpha}.
$
Then
$
H_{\alpha}:=[X_{\alpha},X_{-\alpha}]\in \ft
$
satisfies
$H_{\alpha}(u_{\lambda})=\lambda(H_{\alpha})u_{\lambda}
\neq 0,$
since
$\lambda(H_{\alpha})\neq 0$
for $\alpha\in \Phi^{d}_{Z}$ with
$d>0.$
Hence, for the nonzero vector
$
X_{-\alpha}(u_{\lambda})\in \hat{T}_{[u_{\lambda}]}Z,
$
we obtain
\[
X_{\alpha}(X_{-\alpha}(u_{\lambda}))
=
[X_{\alpha},X_{-\alpha}](u_{\lambda})
=
H_{\alpha}(u_{\lambda})
\neq 0,
\]
a contradiction.
\end{proof}
\subsubsection{Borel--Weil--Bott theorem}
Most vector bundles and bundle morphisms considered on $Z=G/P$ are homogeneous, so their study reduces to the corresponding $P$-modules at the base point $[u_{\lambda}]$. We recall the Borel--Weil--Bott theorem and the PRV formula, which provide the main tools for this reduction.
Let $M$ be a $P$-module. Denote by 
\[
E(M) = G \times^{P} M \longrightarrow G/P = Z
\]
the associated $G$-homogeneous vector bundle on $Z$. Recall $\rho$ is the sum of fundamental weights.

\begin{thm}\label{PRV0}
\textup{(Borel--Weil--Bott)}
		Let 
		$
		M
		$
		be an irreducible $P$-module with lowest weight
		$
		\mu_{\mathrm{low}}.
		$
		Then:
		\begin{itemize}
			\item[(a)]
			If
			$
			\rho-\mu_{\mathrm{low}}
			$
			is singular, then
			$
			H^{p}(Z,E(M))=0,
			\quad
			\text{for all }
			p\geqslant 0.
			$
			
			\item[(b)]
			If
			$
			\rho-\mu_{\mathrm{low}}
			$
			is regular, then there exists a unique nonvanishing cohomology group
			$
			H^{i}(Z,E(M)),
			$
			which is an irreducible $\fg$-module, and
			$
			i
			$
			is the length of the unique Weyl group element
			$
			w
			$
			such that
			$
			w(\rho-\mu_{\mathrm{low}})
			$
			is dominant.
		\end{itemize}

\end{thm}
The following formula gives an explicit description of $H^0(Z,E(M))$; see \cite{PRV67,PY08} for further details.
Let $U_{\mu}$ be any irreducible $\fg$-module of highest weight $\mu$. Consider the evaluation map at a lowest weight vector 
$u_{-\mu} \in U_{\mu}^{\vee}$:
\[
\pi : \Hom(U_{\mu}^{\vee}, M) \longrightarrow M, 
\qquad h \longmapsto h(u_{-\mu}).
\]
For a weight $\eta$, denote by $M_{\eta}$ the $\eta$-weight space of $M$ and set
\[
M^{\eta} = \{\, m \in M \mid e_i^{\langle \eta,\alpha_i^{\vee}\rangle+1}(m) = 0 \ \text{for all } i \,\}.
\]
\begin{thm}\label{PRV}
	Let $M$ be a $P$-module. Then the following hold:
	\begin{itemize}
		\item[(1)]
		There is a canonical isomorphism of $\fg$-modules:
		\begin{equation}\label{peter_wey}
			H^{0}(Z,E(M))
			\cong
			\bigoplus_{\mu\in \Lambda^{+}}
			U_{\mu}^{\vee}
			\otimes
			\Hom_{P}(U_{\mu}^{\vee},M),
		\end{equation}
		where $\fg$ acts trivially on $\Hom_{P}(U_{\mu}^{\vee},M).$	
		\item[(2)]
		(\emph{PRV formula})
		 The projection
		$
		\pi
		$
		induces an isomorphism of vector spaces
		\[
		\pi:
		\Hom_{\fn}(U_{\mu}^{\vee},M)
		\xrightarrow{\ \sim\ }
		M^{\mu},
		\]
        where $\fn$ is the nilpotent radical of the Borel subalgebra. In particular, it restricts to an isomorphism
		\[
		\pi:
		\Hom_{P}(U_{\mu}^{\vee},M)
		\xrightarrow{\ \sim\ }
		\left\{
		m\in M_{-\mu}\cap M^{\mu}
		\ \middle|\
		f_{i}(m)=0
		\text{ for all }
		\alpha_{i}\notin \Delta_{Z}
		\right\}.
		\]
	\end{itemize}
\end{thm}
\subsubsection{(Co)minuscule, quasi-minuscule and adjoint varieties}
We recall the definition of several special classes of flag varieties. 
The first class consists of Hermitian symmetric spaces.

\begin{defn}\label{cominuscule}
	Let
	$
	Z=G/P\subset \BP U
	$
	be a flag variety. 
	It is called an \emph{(irreducible) Hermitian symmetric space} if $\fg/\fp$
	is a completely reducible (resp.\ irreducible) $P$-module.
\end{defn}

A Hermitian symmetric space is a product of irreducible Hermitian symmetric spaces (IHSS for short). The following well-known fact characterizes IHSSs in terms of highest weights.

\begin{lem}
	Let
	$
	Z\subset \BP U_{m \lambda_{i}}
	$
	be a flag variety for some fundamental weight $\lambda_{i}$. Then $Z$ is isomorphic to an IHSS  if and only if 
	$
	\lambda_{i}
	$
	is minuscule or cominuscule.
\end{lem}

We refer to~\cite{Ses78} for the definitions and classification of (co)minuscule highest weights.

Next we introduce the class of flag varieties associated with highest roots.

\begin{defn}\label{ad_coad}
Let $\fg$ be a simple Lie algebra, and let $\beta_l$ and $\beta_s$ denote the highest long and short roots in $\Phi^+$, respectively. A highest weight module $U_\lambda$ is called \emph{adjoint} if $\lambda=\beta_l$, and \emph{quasi-minuscule} if $\lambda=\beta_s$. The corresponding flag varieties are called adjoint and quasi-minuscule varieties, respectively.
\end{defn}

Our convention for quasi-minuscule weights is equivalent to the original definition in terms of Weyl group actions on weights; this follows from the classification of quasi-minuscule highest weights in~\cite{Ses78}.
For convenience, we list the three classes of flag varieties in Table~1. Here we follow the notation as in \cite{Bou02}.
\begin{table}[H]
	\centering
	\caption{}
	\label{Table1}
	\begin{tabular}{|c|c|c|c|c|c|c|c|c|c|}
		\hline
		Type of $\fg$ & $A_{n}$ & $B_{n}$ & $C_{n}$ & $D_{n}$ & $E_{6}$ & $E_{7}$ & $E_{8}$ & $F_{4}$ & $G_{2}$ \\
		\hline
		(co)minuscule
		& $\lambda_{k}$
		& $\lambda_{1},\lambda_{n}$
		& $\lambda_{1},\lambda_{n}$
		& $\lambda_{1},\lambda_{n-1},\lambda_{n}$
		& $\lambda_{1},\lambda_{6}$
		& $\lambda_{7}$
		& $\emptyset$
		& $\emptyset$
		& $\emptyset$
		\\
		\hline
		quasi-minuscule
		& $\lambda_{1}+\lambda_{n}$
		& $\lambda_{1}$
		& $\lambda_{2}$
		& $\lambda_{2}$
		& $\lambda_{2}$
		& $\lambda_{1}$
		& $\lambda_{8}$
		& $\lambda_{4}$
		& $\lambda_{1}$
		\\
		\hline
		adjoint
		& $\lambda_{1}+\lambda_{n}$
		& $\lambda_{2}$
		& $2\lambda_{1}$
		& $\lambda_{2}$
		& $\lambda_{2}$
		& $\lambda_{1}$
		& $\lambda_{8}$
		& $\lambda_{1}$
		& $\lambda_{2}$
		\\
		\hline
	\end{tabular}
\end{table}
\section{Tangentially non-degenerate varieties}
In this section, we give the classification of tangentially non-degenerate flag varieties. 
\begin{defn}\label{tang_nond_defn}
	Let \( Z \subset \mathbb{P}U \) be a non-degenerate projective subvariety. We say that \( Z \) is tangentially non-degenerate if
	\[
	\wedge^{2}_{Z}
	:=
	\{ w \in \wedge^{2}U^{\vee} \mid w(\hat{z}, \widehat{T}_{z}Z)=0,\ \forall z \in Z_{\mathrm{sm}} \}
	= 0.
	\]
\end{defn}
\begin{thm}\label{class_tang}
	Let
	$
	Z\subset \BP U
	$
	be a flag variety. Then
	$
	Z
	$
	is tangentially non-degenerate if and only if it is isomorphic to one of the following:
	\begin{itemize}
		\item[(1)] the VMRT of an irreducible Hermitian symmetric space in its minimal embedding, namely:
		the Segre product
		$
		\BP^{k}\times \BP^{l},
		$
		the Veronese variety
		$
		v_{2}(\BP^{n})\subset \BP(S^{2}\BC^{n+1}),
		$
		or an irreducible Hermitian symmetric space of rank at most two,
		\[
		\BP^{n},
		\quad
		Gr(2,n),
		\quad
		\BQ^{n},
		\quad
		\BS_{5},
		\quad
		E_{6}/P_{1},
		\]
		under their minimal embeddings;
		
		\item[(2)]
		the adjoint variety
		$
		X_{\mathrm{ad}}(\fg)\subset \BP\fg
		$
		(Definition~\ref{ad_coad}) for a simple Lie algebra
		$
		\fg;
		$
		
		\item[(3)]
		a quasi-minuscule variety (Definition~\ref{ad_coad}) not appearing in {\rm(1)} or {\rm(2)}, namely
		$
		Gr_{\omega}(2,2n) \, (n \ge 3)
		$ or $
		F_{4}/P_{4},
		$
		under their minimal embeddings.
	\end{itemize}
\end{thm}

We refer to Table~\ref{Table1} for the complete list of varieties in {\rm(2)}.

Our approach is to study skew-symmetric forms \( \wedge^2 U^{\vee} \) via (higher) Gauss maps, namely through the fundamental forms of \( Z \). It starts with the Gaussian map, which evaluates skew forms on tangent space:
\begin{defn}[{\cite{Wah90}}]\label{tan_nond_denf2}
	Let	$Z \subset \BP U$ be a smooth projective variety. Define the Gaussian map
	\begin{align*}
		\psi: \wedge^{2}U^{\vee} &\rightarrow H^{0}(Z, \Omega_Z(2)),\\
		l \wedge l' &\longmapsto dl \otimes l' - l \otimes dl'. 
	\end{align*}
	Then \( \ker(\psi)=\wedge^{2}_{Z} \), and hence \( Z \) is tangentially non-degenerate if and only if \( \psi \) is injective. 
\end{defn}
For flag varieties, the surjectivity of $\psi$ was proved in \cite{Kum92}, see Corollary \ref{im_Gaussian}. Therefore we have an exact sequence:
\begin{equation}\label{Gaussian_sequence}
	0 \rightarrow \wedge_{Z}^2	\rightarrow   \wedge^{2}U^{\vee} \xrightarrow{\psi} H^{0}(Z,\Omega_Z(2)) \rightarrow 0
\end{equation}

Our argument splits into two cases according to whether the highest weight is fundamental. In the non-fundamental case (Proposition~\ref{tensor}), we follow the approach of~\cite[Section~5.2]{Cho25} to construct nonzero elements in $\wedge_{Z}^{2}$ using tensors of linear forms (Proposition~\ref{non_vanish}). In the fundamental case (Proposition~\ref{prim_class}), we apply Kumar’s results~\cite{Kum92} on the image of the Gaussian map for flag varieties. This reduces the classification to the isolated case
$E_{6}/P_{3},$
which is excluded by a direct computation.
\subsection{Gauss maps and evaluation of bilinear forms}
In this section, we recall basic facts needed for the classification. Most of them can be found in \cite{Wah90,Wah91,EN23}. We first recall the filtration on spaces of bilinear forms introduced in \cite[(1.4)]{Wah90}, and then specialize them to flag varieties. Then we record two basic families of tangentially non-degenerate varieties and verify the varieties listed in Theorem~\ref{class_tang}.
\subsubsection{Gauss maps and fundamental forms}
Let \(Z\subset\BP U\) be a non-degenerate projective subvariety, and let
\(z\in Z\) be a smooth point. For each \(k\ge0\), consider the
\(k\)-jet restriction map
\[
j_{k,z}:
U^\vee
\longrightarrow
\CO_Z(1)_z\otimes\CO_z/\mathfrak m_z^{k+1}.
\]
It induces an injective homomorphism
\[
\ker(j_{k-1,z})/\ker(j_{k,z})
\hookrightarrow
\CO_Z(1)_z\otimes\Sym^kT_z^*Z,
\]
for $k \ge 1$,
whose image
\[
\Pi_{k,z}
\subset
\CO_Z(1)_z\otimes\Sym^kT_z^*Z
\]
is called the \(k\)-th projective fundamental form of \(Z\) at \(z\).

The \(k\)-th osculating space of \(Z\) at \(z\) is defined by
\[
T_z^{(k)}Z
:=
\{\,v\in U\mid
\ell(v)=0
\text{ for all }
\ell\in\ker(j_{k})\,\}.
\]
Setting
$
T_z^{(-1)}Z=0,
$
we have
$
T_z^{(0)}Z=\langle\hat z\rangle
$
and
$
T_z^{(1)}Z=\widehat T_zZ,
$
giving the osculating flag
\[
\langle\hat z\rangle
\subset
\widehat T_zZ
\subset
T_z^{(2)}Z
\subset
\cdots
\subset
T_z^{(r)}Z
=
U.
\]
Equivalently, the $k$-th fundamental form is the surjective homomorphism
\[
\Pi_{k,z}:
\Sym^kT_zZ
\longrightarrow
\CO_Z(1)_z
\otimes
\bigl(T_z^{(k)}Z/T_z^{(k-1)}Z\bigr).
\]

For each \(k\ge0\), there is a maximal nonempty open subset
\(V_k\subset Z\) on which
\(\dim T_z^{(k)}Z\) is constant, say \(t_k\)
(cf. \cite[Definition~2.2]{EN23}). We obtain the \(k\)-th Gauss map
\[
g_k:
V_k
\longrightarrow
Gr(t_k,U),
\qquad
z
\longmapsto
T_z^{(k)}Z.
\]

We recall the following description of fundamental forms in terms of Gauss maps.

\begin{prop}[{\cite[Proposition 3.8 and Remark 3.9]{EN23}}]\label{gau_fund}
	Let $z \in V_{k-1}$, and denote by $f_{k,z} = dg_{k-1,z}$ the differential of $g_{k-1}$ at $z$. Then $f_{k,z}$ can be written as
	\[
	f_{k,z} : T_z Z \times T_z^{(k-1)}Z \longrightarrow U / T_z^{(k-1)}Z,
	\]
	where we identify 
	\[
	T_{g_{k-1}(z)} Gr(t_{k-1},U) = \Hom(T_z^{(k-1)}Z, U/T_z^{(k-1)}Z).
	\]
	
	\begin{itemize}
		\item[(1)] If $z \in V_k$, then the image of $f_{k,z}$ is $T_z^{(k)}Z / T_z^{(k-1)}Z$.
		
		\item[(2)] If $v \in T_z^{(k-2)}Z$, then $f_{k,z}(u,v)=0$ for all $u \in T_z Z$. Hence $f_{k,z}$ factors as
		\[
		\overline{f}_{k,z} : T_z Z \times \big(T_z^{(k-1)}Z / T_z^{(k-2)}Z\big)
		\longrightarrow U / T_z^{(k-1)}Z.
		\]
		
		\item[(3)] Iterating the maps $\overline{f}_{k,z}, \overline{f}_{k-1,z}, \dots, \overline{f}_{1,z}$, one obtains a multilinear map
		\[
		\gamma'_{k,z} : T_z Z^{\otimes k} \longrightarrow \Hom(\hat{z}, T_z^{(k)}Z / T_z^{(k-1)}Z).
		\]
		This map is symmetric and coincides with the $k$-th fundamental form, i.e.
		$
		\gamma_{k,z} ^{'}= \Pi_{k,z},
		$
		under the natural identification
		\[
		\Hom(\hat{z}, T_z^{(k)}Z / T_z^{(k-1)}Z)
		\simeq
		\mathcal{O}_Z(1)_z \otimes \big( T_z^{(k)}Z / T_z^{(k-1)}Z \big).
		\]
	\end{itemize}
\end{prop}
Next we evaluate the bilinear forms in terms of fundamental forms. 
\begin{defn}\label{ev}
	Let \( Z \subset \mathbb{P}U \) be a smooth non-degenerate projective variety. For any \( k \geq 0 \), define
	\[
	(U^{\vee} \otimes U^{\vee})_k
	:=
	\{ \sigma \in U^{\vee} \otimes U^{\vee} \mid
	\sigma(\hat{z}, T_z^{(k-1)}Z) = 0,\ \forall z \in Z \}.
	\]
	This defines a filtration of \( U^{\vee} \otimes U^{\vee} \), and there is a natural evaluation map
	\[
	0 \longrightarrow (U^{\vee} \otimes U^{\vee})_{k+1}
	\longrightarrow (U^{\vee} \otimes U^{\vee})_k
	\xrightarrow{ev_k}
	H^0\big(Z, \mathrm{Sym}^k(\Omega_Z)(2)\big),
	\]
	where for each \( z \in Z \), we identify
	\[
	\mathrm{Sym}^k(\Omega_Z)(2)_z
	\simeq
	(\langle \hat{z} \rangle^\vee)^{\otimes 2}
	\otimes \mathrm{Sym}^k T_z^*Z,
	\]
	and for \( \sigma \in (U^{\vee} \otimes U^{\vee})_k \), \( \alpha \in \mathrm{Sym}^k T_z Z \), define
	\[
	ev_k(\sigma)_z(\hat{z}, \hat{z}, \alpha)
	=
	\sigma\big(\hat{z}, \Pi_{k,z}(\alpha)(\hat{z})\big).
	\]
	In particular,
	$
	ev_k(\sigma)_z(\hat{z}) \in \Pi_{k,z}
	\subset \mathcal{O}_Z(1)_z \otimes \mathrm{Sym}^k T_z^*Z.
	$
\end{defn}
When \(Z\subset \mathbb P U\) is embedded by the complete linear system of a
very ample line bundle, this filtration agrees with Wahl's construction
in~\cite[(1.4)]{Wah90}. Indeed, on \(Z\times Z\) there is an exact sequence
\begin{equation}\label{filt_03}
	0\to
	p_1^*\CO_Z(1)\otimes p_2^*\CO_Z(1)\otimes\mathcal I_{\Delta_Z}^{k+1}
	\to
	p_1^*\CO_Z(1)\otimes p_2^*\CO_Z(1)\otimes\mathcal I_{\Delta_Z}^{k}
	\to
	\Sym^{k}\Omega_Z(2)
	\to 0,
\end{equation}
where \(\Delta_Z\subset Z\times Z\) is the diagonal and $p_{1},p_{2}$ are the two projections to each factor.

\begin{lem}
	For every \(k\ge 0\), we have
	\[
	H^0\!\left(
	Z\times Z,\,
	p_1^*\CO_Z(1)\otimes p_2^*\CO_Z(1)\otimes\mathcal I_{\Delta_Z}^{k}
	\right)
	=
	(U^\vee\otimes U^\vee)_k
	\subset U^\vee\otimes U^\vee.
	\]
	Moreover, \(ev_k\) is the map on global sections induced by the rightmost
	morphism in~\eqref{filt_03}.
\end{lem}
\begin{proof}
	Taking global sections of~\eqref{filt_03}, set
	$
	\mathcal R_k
	=
	p_1^*\CO_Z(1)\otimes p_2^*\CO_Z(1)\otimes
	\mathcal I_{\Delta_Z}^k,
	$
	and denote the induced rightmost map by
	\[
	ev_k':
	H^0(Z\times Z,\mathcal R_k)
	\longrightarrow
	H^0\!\left(Z,\Sym^{k}\Omega_Z(2)\right).
	\]
	Since the two filtrations are defined inductively as the kernels of
	\(ev_k\) and \(ev_k'\), respectively, and coincide for \(k=0\), it
	suffices to prove that, for every \(k\ge1\),
	\[
	H^0(Z\times Z,\mathcal R_k)
	\subset
	(U^\vee\otimes U^\vee)_k,
	\]
	and that the restriction of \(ev_k\) to
	\(H^0(Z\times Z,\mathcal R_k)\) coincides with \(ev_k'\).
	By the projection formula,
	\[
	p_{1,*}\mathcal R_k
	=
	\CO_Z(1)\otimes
	p_{1,*}\!\left(
	p_2^*\CO_Z(1)\otimes\mathcal I_{\Delta_Z}^k
	\right).
	\]
	For each \(z\in Z\), restriction to the fiber of \(p_1\) gives
	\begin{equation}\label{rest_filt}
		\left.
		p_{1,*}\!\left(
		p_2^*\CO_Z(1)\otimes\mathcal I_{\Delta_Z}^k
		\right)
		\right|_z
		\longrightarrow
		H^0\!\left(
		Z,\CO_Z(1)\otimes\mathfrak m_z^k
		\right)
		=
		\ker(j_{k-1,z}).
	\end{equation}
	Thus, if
	$
	\sigma\in
	H^0(Z\times Z,\mathcal R_k)
	\subset
	H^0(Z\times Z,\mathcal R_0)
	=
	U^\vee\otimes U^\vee,
	$
	then, for every \(z\in Z\), its restriction to
	\(\{z\}\times Z\) vanishes to order at least \(k\) at \(z\).
	By the definition of the osculating spaces, this is equivalent to $\sigma
\in 
	(U^\vee\otimes U^\vee)_k.
	$
	Moreover, after restricting to the fiber of \(p_1\) over \(z\), the
	rightmost morphism in~\eqref{filt_03} is, after tensoring by
	\(\CO_Z(1)_z\), induced by the quotient map
	\[
	\ker(j_{k-1,z})
	\longrightarrow
	\ker(j_{k-1,z})/\ker(j_{k,z})
	=
	\Pi_{k,z}.
	\]
	It follows from Definition~\ref{ev} that the restriction of \(ev_k\) to
	\(H^0(Z\times Z,\mathcal R_k)\) coincides with \(ev_k'\).
\end{proof}
The involution exchanging the two factors of $Z\times Z$ induces the respective filtrations on the symmetric and alternating tensors.

\begin{prop}[{\cite[Proposition~1.5]{Wah90}}]\label{12_ev}
	Let
	$
	Z\subset\BP U
	$
	be a smooth projective variety embedded by a very ample line bundle. Then 
	\[
	\Sym^{2}U^\vee\cap (U^\vee\otimes U^\vee)_{2k-1}
	\subset
	(U^\vee\otimes U^\vee)_{2k},
	\qquad
	\wedge^2U^\vee\cap (U^\vee\otimes U^\vee)_{2k}
	\subset
	(U^\vee\otimes U^\vee)_{2k+1},
	\]
	for any $k \ge 1$.
\end{prop}
For $k=1,2$, set $\psi=ev_1|{\wedge^2U^\vee}$ and $\varphi=ev_2|{I_2(z)}$, respectively. We then obtain the following two sequences.

\begin{align}
	0 \rightarrow \wedge_{Z}^{2} 
	&\rightarrow \wedge^{2}U^{\vee} 
	\xrightarrow{\psi} H^{0}(Z,\Omega_Z(2)), \label{Gaussian} \\
	0 \rightarrow I_{2}(Z) \cap (U^{\vee} \otimes U^{\vee})_{3} 
	&\rightarrow I_{2}(Z) 
	\xrightarrow{\varphi} H^{0}(Z, S^{2}\Omega_Z(2)), \notag
\end{align}
where $I_{2}(Z)=H^{0}(\mathbb{P}U,\mathcal{I}_{Z}(2)) \subset \mathrm{Sym}^{2}U^{\vee}$.
Here $\psi$ recovers the Gaussian map in \eqref{Gaussian_sequence};
for $\varphi$, it follows from Definition \ref{ev} that it factors through the following sequence, obtained by taking global sections:
\begin{align}
	\mathcal{I}_{Z}(2) \longrightarrow \mathcal{I}_{Z}/\mathcal{I}_{Z}^{2}(2) \longrightarrow (\CN_{Z}^{2})^{\vee}(2) \xrightarrow{\Pi_{2}^{\vee}(1)} \Sym^{2}\Omega_Z(2),
\end{align}
where $\CN_{Z}^{2}$ denotes the subsheaf of the normal bundle whose fiber at each point $z \in Z$ is
\[
\Hom(\hat{z},\, T_z^{(2)}Z / T_z^{(1)}Z),
\]
and $\Pi_{2}^{\vee}(1)$ is the dual of second fundamental form, twisted by $\CO(1)$.
\begin{cor}\label{quadrics}

Let $Z\subset\BP U$
be a smooth projective variety of dimension $n\ge1$, embedded by a very ample line bundle. Then the Gaussian map $\psi$ is nonzero. If moreover the ideal of $Z$ is generated by quadrics and $V_2=Z$, then $\varphi$ in \eqref{Gaussian} is nonzero.
\end{cor}

\begin{proof}

Fix $z=[u]\in Z$. Since $\dim Z\ge1$, there exists
$
v\in\widehat T_zZ\setminus\BC u.
$
Choose $w\in\wedge^2U^\vee$ such that $w(u,v)\neq0$. Then $\psi(w)_z\neq0$ by definition, and hence $\psi\neq0$.

For the second claim, since the ideal of $Z$ is generated by quadrics, $\CI_Z(2)$ is globally generated. Moreover, since $V_2=Z$ and $Z\neq\BP U$, the bundle $\CN_Z^2$ in~(3.5) is nonzero otherwise the second fundamental form would vanish identically, forcing $Z$ to be linear. In~(3.5), the first two morphisms are surjective by construction, while the last one is injective, being dual to the second fundamental form $\Pi_2$. Hence their composition
is nonzero. Since $\CI_Z(2)$ is globally generated, the induced map on global sections is nonzero, and coincides with $\varphi$.
\end{proof}

\subsubsection{Flag varieties}
We illustrate the above discussion when
$Z=G/P\subset \BP U$ is a flag variety, using the notation of Section~2. The fundamental forms and osculating spaces of $Z$ are $G$-invariant. Hence, by homogeneity, $V_{k}=Z$ for every $k\geqslant 0,$
and it suffices to describe them at the base point $z=[u_{\lambda}].$ There they are given by the \emph{PBW filtration}. More precisely, let $U(\fg)$ be the universal enveloping algebra of $\fg,$ and let $U_{k}(\fg)$
denote its $k$-th filtered piece.
\begin{prop}[{\cite[Proposition 2.3]{LM03}}]
	The $k$-th osculating space of $Z=G/P \subset  \BP U_{\lambda}$ at $z=[u_{\lambda}]$ is $$T_{z}^{(k)}Z=U_{k}(\fg).u_{\lambda}.$$ 
	Moreover its $k$-th fundamental form is given by the following diagram:
	\begin{center}
		\begin{tikzcd}
			\Sym^{k}\fg \otimes  \langle u_{\lambda} \rangle \arrow[r, "\simeq"] \arrow[d]&
			U_{k}(\fg)/U_{k-1}(\fg) \otimes  \langle u_{\lambda} \rangle \arrow[d] \\ 
			\Sym^{k}(\fg/\fp) \otimes \langle u_{\lambda} \rangle  \arrow[r,"\Pi_{k,z}"]& U_{k}(\fg).u_{\lambda}/U_{k-1}(\fg).u_{\lambda}
		\end{tikzcd}
	\end{center}
	where we identify $T_{[u_{\lambda}]}Z=\fg/\fp$ as before.
\end{prop}
As a corollary we can describe the bilinear forms of lowest weights:
\begin{lem}\label{inv_form}
	Let $\sigma\in U^{\vee}\otimes U^{\vee}$ be an $\fn^{-}$-invariant form. Then
	$\sigma\in (U^{\vee}\otimes U^{\vee})_{k}$ if and only if
	$\sigma\bigl(u_{\lambda},U_{k-1}(\fg).u_{\lambda}\bigr)=0$.
	Moreover, $ev_k(\sigma)\neq 0$ if and only if
	$ev_k(\sigma)([u_{\lambda}])\neq 0$, namely the polynomial
	\[
	ev_k(\sigma)([u_{\lambda}]):\fg\longrightarrow\BC,
	\qquad
	X\longmapsto \sigma(u_{\lambda},X^k u_{\lambda})
	\]
	is nonzero.
\end{lem}

\begin{proof}
	Let $R=\exp(\fn^{-})$ be the unipotent subgroup generated by $\fn^{-}$. Since $\sigma$ is $\fn^{-}$-invariant, it is $R$-invariant. If
	$\sigma(u_{\lambda},U_{k-1}(\fg).u_{\lambda})=0,$
	then for any $u'\in R.u_{\lambda}$, $\sigma(u',U_{k-1}(\fg).u')=0$ by $R$-invariance. Since $R.[u_{\lambda}]$ is open in $Z$, it follows from the above Proposition that the section $ev_{k-1}(\sigma)$ vanishes identically, hence
	$\sigma\in (U^{\vee}\otimes U^{\vee})_{k}.$
	The converse is immediate from the definition. Since $ev_{k}$ is $G$-equivariant, the second statement follows similarly. The last claim follows directly from Definition~\ref{ev} and the above Proposition.
\end{proof}
\subsubsection{Examples}
In this subsection we verify that all varieties appearing in Theorem 3.2 are tangentially non-degenerate. The converse classification will be proved in Sections 3.2 and 3.3. 

\begin{lem}[{\cite[Proposition 14]{HM98}}]
	Let
	$Z\subset \BP U$
	be a Hermitian symmetric space which is one of the varieties listed in Theorem~\ref{class_tang}(1). Then $Z$ is tangentially non-degenerate.
\end{lem}

\begin{lem}
	Let $Z \subset \BP U_{\lambda}$ be the adjoint or the quasi-minuscule variety of a simple Lie algebra $\fg$. Then $Z$ is tangentially non-degenerate.
\end{lem}

\begin{proof}
	In this case $\lambda\in\Phi^+$ is the highest long root
	\textup{(}resp.\ highest short root\textup{)}, hence the lowest weight of $U$
	is $w_0(\lambda)=-\lambda$. Let $u_\lambda$ and $u_{-\lambda}$ be
	highest and lowest weight vectors, and let
	$SL_{2,\lambda}\subset G$ be the subgroup associated with the root
	$\lambda$. Then $SL_{2,\lambda}$ acts on
	$
	\left\langle SL_{2,\lambda}\cdot u_{-\lambda}\right\rangle
	$
	via its adjoint representation. Hence
	$
	SL_{2,\lambda}\cdot [u_\lambda]\cong\BP^1
	$
	is a conic in $Z$ passing through $[u_\lambda]$ and $[u_{-\lambda}]$.
	
	By the same argument as in~\cite[Lemma~2.12]{FL24}, the diagonal
	$G$-orbit
	$
	G\cdot\bigl([u_\lambda],[u_{-\lambda}]\bigr)
	\subset Z\times Z
	$
	is open. Therefore a general pair of points can be joined by a conic contained in \( Z \), and hence $Z$ is tangentially
	non-degenerate as in~loc.~cit.
\end{proof}

\subsection{Non-fundamental case}
In this subsection, we treat the case where the highest weight is not fundamental. Let $Z\subset \BP U$ be a flag variety under the action of $G$, with $U=U_\lambda$  and $\lambda=\lambda'+\lambda'',$  where $\lambda'$ and $\lambda''$ are nonzero dominant weights. Denote the corresponding flag varieties by $Z'=G/P_{\lambda'}\subset \BP U_{\lambda'},$ and $Z''=G/P_{\lambda''}\subset \BP U_{\lambda''}.$
The following result  motivates our approach.
\begin{thm}[{\cite[Theorem~5.5]{Cho25}}]\label{Cho25}
	Let \( Z \subset \mathbb{P}U \) be a non-degenerate projective variety embedded by the complete linear system of a very ample line bundle \( L \). Assume that there is a decomposition \( L = L_{1}+L_{2}+L_{3} \) such that \( h^{0}(L_{i}) \geqslant 2 \) for each \( i \). Then \( \wedge_{Z}^{2} \neq 0 \).
\end{thm}
 We will prove a refined version:
\begin{prop}\label{tensor}
	Let \( Z \subset \mathbb{P}U \) be a flag variety under the action of a group \( G \). Keep the notation as above. 
	If $Z' \not= \BP U_{\lambda'}$ or $ Z'' \not=\BP U_{\lambda''} $, then \( \wedge_{Z}^2 \neq 0 \).
\end{prop}
Together with examples given in Section 3.1.3, this immediately implies the following:
\begin{cor}\label{nprim}
	Let \( Z \subset \mathbb{P}U_{\lambda} \) be a flag variety such that $\lambda$ is not a fundamental weight. Then \( Z \) is tangentially non-degenerate if and only if it is isomorphic to one of the following:
	$$
	v_{2}(\BP^{n}),\mathbb{P}^{k} \times \mathbb{P}^{l}, Fl(1,n;n+1)\, (n \ge 2).
	$$
\end{cor}
  \begin{proof}
Assume that $Z$ is tangentially non-degenerate. By Proposition~3.14, for any decomposition
$\lambda=\lambda'+\lambda''$
into nonzero dominant weights, we have
$
Z'=\BP U_{\lambda'}
$
and
$
Z''=\BP U_{\lambda''}.
$
Hence $\lambda'$ and $\lambda''$ are fundamental weights corresponding to the standard representation or its dual of a simple factor of type $A$, which gives exactly the three varieties listed above. Conversely, these varieties are tangentially non-degenerate by Lemmas~3.11 and~3.12.
\end{proof}
To prove Proposition \ref{tensor} we present the general approach using tensors of linear forms to construct nonzero elements in $\wedge_{Z}^{2}$.
\begin{defn}\label{tensor_map}
Keep the notation above. Choose highest weight vectors
$u_{\lambda'}\in U_{\lambda'}$ and
$u_{\lambda''}\in U_{\lambda''}$, and identify the embedding
$$
U_{\lambda'+\lambda''}
=
\left\langle
\fg\cdot(u_{\lambda'}\otimes u_{\lambda''})
\right\rangle
\subset
U_{\lambda'}\otimes U_{\lambda''}.
$$
This induces a $G$-equivariant inclusion
$$
Z\longrightarrow Z'\times Z'',
\qquad
[g\cdot(u_{\lambda'}\otimes u_{\lambda''})]
\longmapsto
\bigl([g\cdot u_{\lambda'}],[g\cdot u_{\lambda''}]\bigr),
$$
and, by duality, a surjection
$$
\rho:
U_{\lambda'}^\vee\otimes U_{\lambda''}^\vee
\longrightarrow
U_{\lambda'+\lambda''}^\vee.
$$
Via the Segre embedding
$Z'\times Z''\subset\BP(U_{\lambda'}\otimes U_{\lambda''})$,
the restriction
$$
\wedge^2\rho:
\wedge^2_{Z'\times Z''}
\longrightarrow
\wedge^2_Z
$$
is called the \emph{tensor map} associated with
$\lambda=\lambda'+\lambda''$.
\end{defn}

\begin{lem}\label{lem_prod_wedge_Z}
Keep the notation above. Under the decomposition
$$\wedge^2\bigl(
U_{\lambda'}^\vee\otimes U_{\lambda''}^\vee
\bigr)
\cong
\wedge^2U_{\lambda'}^\vee\otimes S^2U_{\lambda''}^\vee
\oplus
S^2U_{\lambda'}^\vee\otimes\wedge^2U_{\lambda''}^\vee,
$$
we have the following equality of subspaces:
\begin{equation}\label{prod_wedge_Z}
\wedge^2_{Z'\times Z''}
=
\left(
\wedge^2_{Z'}\otimes S^2U_{\lambda''}^\vee
+
\wedge^2U_{\lambda'}^\vee\otimes I_2(Z'')
\right)
\oplus
\left(
I_2(Z')\otimes\wedge^2U_{\lambda''}^\vee
+
S^2U_{\lambda'}^\vee\otimes\wedge^2_{Z''}
\right).
\end{equation}
In particular, if $\wedge^2\rho$ is nonzero on one of the two summands in~\eqref{prod_wedge_Z}, then $\wedge^2_Z\neq0$.
\end{lem}

\begin{proof}
For $[u']\in Z'$ and $[u'']\in Z''$, the affine tangent space of the Segre embedding is
$$
\widehat T_{[u'\otimes u'']}(Z'\times Z'')
=
u'\otimes\widehat T_{[u'']}Z''+
\widehat T_{[u']}Z'\otimes u''.
$$
For
$\sigma'\in\wedge^2U_{\lambda'}^\vee$,
$q''\in S^2U_{\lambda''}^\vee$,
$v'\in\widehat T_{[u']}Z'$, and
$v''\in\widehat T_{[u'']}Z''$, one computes
$$
(\sigma'\otimes q'')
\bigl(
u'\otimes u'',
v'\otimes u''+u'\otimes v''
\bigr)
=
\sigma'(u',v')q''(u'',u'').
$$
The analogous formula for
$q'\otimes\sigma''\in
S^2U_{\lambda'}^\vee\otimes\wedge^2U_{\lambda''}^\vee$
is obtained by interchanging the two factors. The decomposition then follows directly from the definitions of
$\wedge^2_{Z'}$, $\wedge^2_{Z''}$, $I_2(Z')$, and $I_2(Z'')$.
The last assertion is immediate from Definition~\ref{tensor_map}.
\end{proof}

Now Proposition \ref{tensor} follows from the following more precise statement.

\begin{prop}\label{non_vanish}
	If \( Z' \not=\mathbb{P}U_{\lambda'} \) then the restriction of \( \wedge^{2}\rho \) to \( I_{2}(Z') \otimes \wedge^{2}U_{\lambda''}^{\vee} \) induces a nonzero map
	\[
	\wedge^{2}\rho: I_{2}(Z') \otimes \wedge^{2}U_{\lambda''}^{\vee} \longrightarrow \wedge_{Z}^{2}.
	\]
	More precisely $ev_{3}\circ \wedge^{2}\rho \not=0$ (cf. Definition \ref{ev}) and hence $\wedge_{Z}^2 \not=0$.
\end{prop}
\begin{proof}
	Since $Z'$ is a flag variety not isomorphic to a projective space, its ideal in $\BP U_{\lambda'}$ is generated by quadrics by a result of Kostant (cf. \cite{Gar82}). By Corollary~\ref{quadrics}, the map $\varphi = ev_{2}|_{I_{2}(Z')}$ is nonzero. Hence, by Lemma~\ref{inv_form}, there exists $\sigma' \in I_{2}(Z')$ such that the function
	\[
	f_{\sigma'}:
	\fg \longrightarrow \BC, \quad
	X \longmapsto \sigma'(u_{\lambda'}, X^{2}u_{\lambda'})
	\]
	is nonzero.
	On the other hand, since $\dim(Z'') \geqslant 1$, the Gaussian  map for $Z''$ is nonzero by Corollary \ref{quadrics}. Hence, by Lemma \ref{inv_form}, there exists $\sigma'' \in \wedge^{2}U_{\lambda''}^{\vee}$ such that
	the function
	\[
	f_{\sigma''}:
	\fg \longrightarrow \BC, \quad
	X \longmapsto \sigma''(u_{\lambda''}, Xu_{\lambda''})
	\]
	is nonzero.
		Set $\sigma = \sigma' \otimes \sigma'' \in I_{2}(Z') \otimes \wedge^{2}U_{\lambda''}^{\vee}$. We claim that $(\wedge^{2}\rho)(\sigma)$ is nonzero in $\wedge_{Z}^{2}$. Indeed, for any $k \geqslant 1$ and $X \in \fg$, we have
	\[
	(\wedge^{2}\rho)(\sigma)(u_{\lambda'} \otimes u_{\lambda''},\, X^{k}(u_{\lambda'} \otimes u_{\lambda''}))
	= \sum_{i=1}^{k} \binom{k}{i} \, \sigma'(u_{\lambda'}, X^{i}u_{\lambda'}) \cdot \sigma''(u_{\lambda''}, X^{k-i}u_{\lambda''}).
	\]
	By Proposition~\ref{12_ev}, this expression vanishes for $k \leqslant 2$. For $k=3$, the only possibly nonzero term corresponds to $i=2$, giving
	\[
	3\,\sigma'(u_{\lambda'}, X^{2}u_{\lambda'}) \cdot \sigma''(u_{\lambda''}, Xu_{\lambda''})
	= 3\, f_{\sigma'}(X) f_{\sigma''}(X),
	\]
	which is a product of two nonzero polynomial functions on $\fg$. Therefore, it is nonzero for some $X \in \fg$, and hence
	$
	ev_{3}\big((\wedge^{2}\rho)(\sigma)\big) \neq 0
	$
	by Lemma~\ref{inv_form}.
\end{proof}

\subsection{Fundamental case}
In this subsection we give the classification in the case where $\lambda=\lambda_{i}$ is a fundamental weight. 
\begin{prop}\label{prim_class}
	Let $Z \subset \BP U_{\lambda}$ be a flag variety, where $\lambda = \lambda_{i}$ is a fundamental weight. Then $Z$ is tangentially non-degenerate if and only if it is projectively isomorphic to one of the following varieties:
	\begin{itemize}
		\item[(1)] an irreducible Hermitian symmetric space of rank $\leqslant 2$ under its minimal embedding, namely
		\[
		\mathbb{P}^{n} \,(n \ge 1),\quad \mathrm{Gr}(2,n) \,(n \ge 5),\quad \mathbb{Q}^{n}\, (n \ge 3),\quad \mathbb{S}_{5},\quad E_{6}/P_{1};
		\]
		\item[(2)] an adjoint variety which is not of type $A$ or $C$, or a quasi-minuscule variety (Definition \ref{ad_coad}).
	\end{itemize}
\end{prop}
This case follows from Kumar's description of the image of the Gaussian map below. Define
$
S_{\lambda} = \{1 \leqslant i \leqslant n: \langle \lambda,\alpha_{i}^{\vee} \rangle  = 0\},
$
and for any $\beta \in \Phi^{+}$,
$
F_{\beta} = \{1 \leqslant i \leqslant n: \beta - \alpha_{i} \notin \Phi^{+}\}.
$
\begin{thm}[\cite{Kum92}]
Let $Z \subset \BP U_{\lambda}$ be a flag variety. Then the map 
$$
ev_{1}:H^{0}(Z\times Z,\CO_Z(1)\boxtimes\CO_Z(1) \otimes \mathcal I_{\Delta_Z}) \rightarrow  H^{0}(Z,\Omega_{Z}(2))  
$$ is surjective. Moreover there is an irreducible decomposition
$$
H^{0}\left(
Z\times Z,
\CO_{Z\times Z}/\CI_{\Delta_Z}^2
\otimes\CO_Z(1)\boxtimes\CO_Z(1)
\right)=\bigoplus_{\substack{
\beta\in\Phi^+\cup \{0\},\
2\lambda-\beta\in\Lambda^+,\
S_\lambda\subset F_\beta
}}
U_{2\lambda-\beta}^{\vee}.
$$
\end{thm}

\begin{cor}\label{im_Gaussian}
	The Gaussian map
	$
	\psi: \wedge^{2}U_{\lambda}^{\vee} \longrightarrow H^{0}(Z,\Omega_{Z}(2))
	$
	is surjective. Moreover, there is an irreducible decomposition
\begin{equation}\label{irr_decom}
H^{0}(Z,\Omega_{Z}(2))=\bigoplus_{\substack{
\beta\in\Phi^+,\
2\lambda-\beta\in\Lambda^+,\
S_\lambda\subset F_\beta
}}
U_{2\lambda-\beta}^{\vee}.
	\end{equation}
\end{cor}
\begin{proof}
By the preceding theorem, $ev_1$ is surjective. Since
$$
H^{0}\left(
Z\times Z,
\CO_Z(1)\boxtimes\CO_Z(1)\otimes\CI_{\Delta_Z}
\right)=I_2(Z)\oplus\wedge^2U_\lambda^\vee,
$$
and $ev_1|_{{I_2(Z)}}=0$ by Proposition~\ref{12_ev}, the restriction of $ev_{1}$ to
$\wedge^2U_{\lambda}^\vee$ is the Gaussian map $\psi$, which is therefore surjective. For the decomposition, consider the exact sequence on $Z\times Z$
$$
0\rightarrow
\CI_{\Delta_Z}/\CI_{\Delta_Z}^2
\otimes\CO_Z(1)\boxtimes\CO_Z(1)
\rightarrow
\CO_{Z\times Z}/\CI_{\Delta_Z}^2
\otimes\CO_Z(1)\boxtimes\CO_Z(1)
\rightarrow
\CO_{Z\times Z}/\CI_{\Delta_Z}
\otimes\CO_Z(1)\boxtimes\CO_Z(1)
\rightarrow0.
$$
Taking global sections gives
$$
0\rightarrow H^{0}(Z,\Omega_Z(2))
\rightarrow
H^{0}\left(
Z\times Z,
\CO_{Z\times Z}/\CI_{\Delta_Z}^2
\otimes\CO_Z(1)\boxtimes\CO_Z(1)
\right)
\rightarrow
H^{0}(Z,\CO_Z(2))
\rightarrow0.
$$
Indeed, the composition
$$
H^{0}\left(
Z\times Z,\CO_Z(1)\boxtimes\CO_Z(1)
\right)
\longrightarrow
H^{0}\left(
Z\times Z,
\CO_{Z\times Z}/\CI_{\Delta_Z}^2
\otimes\CO_Z(1)\boxtimes\CO_Z(1)
\right)
\longrightarrow
H^{0}(Z,\CO_Z(2))
$$
coincides with the natural multiplication map
$$
H^{0}(Z,\CO_Z(1))\otimes H^{0}(Z,\CO_Z(1))
\longrightarrow
H^{0}(Z,\CO_Z(2)),
$$
which is surjective for flag varieties. Hence the rightmost map is surjective.
By the preceding theorem, the middle term has the stated irreducible decomposition. Since
$
H^{0}(Z,\CO_Z(2))=U_{2\lambda}^{\vee}
$
is precisely its summand corresponding to $\beta=0$, the kernel of the rightmost map is the direct sum of the remaining summands. This gives~\eqref{irr_decom}.

\end{proof}

Our analysis naturally splits into three cases:
\begin{lem}\label{three_cases}
	Let $U=U_{\lambda_{i}}$
	be an irreducible $\fg$-module. Then one of the following holds:
	\begin{itemize}
		\item[(1)]
		$U_{\lambda}$ is self-dual, i.e.\
		$U_{\lambda}\cong U_{\lambda}^{\vee};$
		
		\item[(2)]
		$U_{\lambda}$ is (co)minuscule (Definition~\ref{cominuscule});
		
		\item[(3)]
		$(\fg,\lambda)=(E_{6},\lambda_{3})$ or $(E_6,\lambda_{5}).$
	\end{itemize}
\end{lem}

\begin{proof}
	Recall that $U_{\lambda_{i}}$
	is self-dual if and only if $w_{0}(\lambda)=-\lambda.$
	It is well known that $w_{0}\neq -\id$
	only for types $A,$ $D,$ and $E_{6}$ (see e.g.~\cite[p.~66, Table~1]{Hum72}), where
	$w_{0}$ is induced by the unique involution of the Dynkin diagram. Hence, up to duality, the non-self-dual cases are
	$(A_{n},\lambda_{k})$
	with $k<n+1-k,$
	$(D_{n},\lambda_{n-1}),$ and $(E_{6},\lambda_{3}).$
	The first two are cominuscule, and the claim follows.
\end{proof}
\subsubsection{Self-dual case} 
Recall that a self-dual module is either orthogonal or symplectic; namely there is a one-dimensional space of $\fg$-invariants:
\[
(S^{2}U)^{\fg} \simeq \mathbb{C} \quad \text{or} \quad (\wedge^{2}U)^{\fg} \simeq \mathbb{C}.
\]
The following unpublished result of E.~B. Vinberg is proved in
\cite[Theorem 3.3]{PY08}.

\begin{prop}\label{small}
	Let \(U=U_{\lambda_i}\) be a self-dual \(\fg\)-module. If \(U\) is
	orthogonal \textup{(}resp.\ symplectic\textup{)}, then
	\[
	\dim\Hom_{\fg}\!\left(\fg,\wedge^2U^\vee\right)=1
	\qquad
	\left(
	\textup{resp. }
	\dim\Hom_{\fg}\!\left(\fg,\Sym^{2}U^\vee\right)=1
	\right).
	\]
	Equivalently, \(\wedge^2U^\vee\) \textup{(}resp.\ \(\Sym^{2}U^\vee\)\textup{)}
	contains the adjoint representation with multiplicity one.
\end{prop}
Combining with Corollary \ref{im_Gaussian}, we deduce:
\begin{cor}\label{case1_cor}
	Let $Z \subset \BP U_{\lambda_{i}}$ be a flag variety. If $U_{\lambda_{i}}$ is self-dual and $Z$ is tangentially non-degenerate, then:
	\begin{itemize}
		\item[(a)] If $U_{\lambda_i}$ is orthogonal, then there exists a root $\beta \in \Phi^{+}$ such that
		\[
		2\lambda_i = \beta_{l} + \beta,
		\]
		where $\beta_{l}$ denotes the highest long root, and $\beta - \alpha_j \notin \Phi^{+}$ for all $j \neq i$.
		\item[(b)] If $U_{\lambda_i}$ is symplectic, then 
		\[
		2\lambda_i = \beta,
		\]
		where $\beta$ is either the highest long root or the highest short root, and satisfies $\beta - \alpha_j \notin \Phi^{+}$ for all $j \neq i$.
	\end{itemize}
\end{cor}
\begin{proof}
	If $U_{\lambda_i}$ is orthogonal (resp. symplectic), it follows from Proposition~\ref{small} and our assumption that $\operatorname{Im}(\psi)$ contains a copy of the adjoint representation (resp. the trivial representation). 
	The corollary then follows from the irreducible decomposition of $\operatorname{Im}(\psi)$ described in Corollary~\ref{im_Gaussian}.
\end{proof}
Thus, the classification reduces to that of highest weights satisfying one of the above conditions.
\begin{cor}
	The only highest weight module satisfying Corollary~\ref{case1_cor}(b) is
	$
	(\fs \fp_{2n},\lambda_{1}).
	$
\end{cor}

\begin{proof}
	By assumption, $\lambda_{i}$
	is either adjoint or quasi-minuscule (Definition~\ref{ad_coad}). From Table~\ref{Table1}, the condition $2\lambda_{i}=\beta$
	holds only for $\fg=\fs \fp_{2n}$
	with $\beta=2\lambda_{1}.$ The claim follows.
\end{proof}

\begin{lem}\label{cor_case1_table}
In the setting of Corollary~3.24(a), the representations
\((\fg,\lambda_i)\) for which
$
\beta=2\lambda_i-\beta_l
$
is a positive root are precisely those listed in Table~2.
\end{lem}

\begin{proof}
This is obtained by a case-by-case calculation. Indeed, if
\(\beta=2\lambda_i-\beta_l\in\Phi^+\), then necessarily
\(\beta\prec\beta_l\). Thus, for each pair \((\lambda_i,\beta_l)\), we use
their expressions in the basis of simple roots to compute
\(2\lambda_i-\beta_l\) explicitly, and then check whether
$
2\lambda_i-\beta_l\prec\beta_l
$
and whether it is a positive root. The required expressions are taken from the
expression of the highest root in terms of simple roots~\cite[p.~66,
Table~2]{Hum72} and the expression of the fundamental weights in terms of
simple roots~\cite[p.~69, Table~1]{Hum72}. This gives exactly the list in
Table~2.
\end{proof}
\begin{table}[H]
	\centering
	\caption{$(\fg,\lambda_{i})$ with $\beta = 2\lambda_i - \beta_{l} \in \Phi^{+}$} \label{Table2}
	\vspace{5pt}
	\begin{tabular}{|c|c|c|c|c|c|c|c|c|c|}
		\hline 
		$\fg$& $A_{3}$ & $B_{n}$& $C_{n}$ &$D_{n}$ &$E_{6}$ & $E_{7}$ &$E_{8}$  &$F_{4}$& $G_{2}$  \\ \hline
		$\lambda_{i}$	&	$\lambda_{2}$ & $\lambda_{2}, \,\lambda_{4}\,(n=4)$& $\lambda_{2}$ & $\lambda_{2}$ & $\lambda_{2}$  & $\lambda_{1}$ & $ \lambda_{8}$& $ \lambda_{1},\, \lambda_{4}$ & $\lambda_{2}, \, \lambda_{1}$  \\ \hline 
		
		$\beta=2\lambda_{i}-\beta_{l}$ &	$\beta_{l}$ &$\beta_{l}$, $\alpha_{3}+2\alpha_{4}$& $\beta_{l}-2\alpha_{1}$ & $\beta_{l}$ & $\beta_{l}$  & $\beta_{l}$ & $ \beta_{l}$& $ \beta_{l}$,\, $\alpha_{2}+2\alpha_{3}+2\alpha_{4}$& $\beta_{l}$, $\alpha_{1}$ \\ \hline 
	\end{tabular}
\end{table}
Comparing with Table~\ref{Table1}, we immediately deduce the following.
\begin{cor}
	A flag variety satisfies Corollary~\ref{case1_cor}(a) if and only if it is isomorphic to
	$
	Gr(2,4)\subset \BP(\wedge^{2}\BC^{4}),
	$ the orthogonal Grassmannian $
	OG(4,9)\subset \BP\Delta,
	$
	an adjoint variety not of type $A$ or $C$, or a quasi-minuscule variety not of type $A$.
\end{cor}
\subsubsection{(Co)minuscule case}
The (co)minuscule case follows from the irreducibility of $\Im(\psi)$.
\begin{lem}\label{irr_Gaussian}
	Let $Z \subset \BP U_{\lambda_{i}}$ be an irreducible Hermitian symmetric space. Then for each $k  \in \BZ$, $H^{0}(Z,\Omega_{Z}(k))$ is either zero or an irreducible $\fg$-module.	
\end{lem}
\begin{proof}
	Since $Z$ is an irreducible Hermitian symmetric space, $\Omega_{Z}(k)$ is an irreducible homogeneous vector bundle by Definition \ref{cominuscule}. Thus the claim follows by Theorem \ref{PRV0}.
\end{proof}
Then our classification follows from the classification of irreducible exterior squares given by Dynkin; see also~\cite[Theorem~1(b)]{KR09} for a streamlined proof.
\begin{thm}[{\cite{Dyn52}}]
	Let $U=U_{\lambda}$ be an irreducible representation of semisimple Lie algebra $\fg$. If $\wedge^{2}U$	is irreducible, then up to duality one of the following cases occurs:
	\begin{table}[H]
		\begin{tabular}{|c|c|c|c|c|}
			\hline 
			$\fg$ & $A_{n}$ & $B_{n}$ & $D_{n}$ & $E_{6}$ \\\hline
			$\lambda$ & $\lambda_{1}$,\, $2 \lambda_{1}$,\, $\lambda_{2}\, (n \geqslant 2)$& 
			$\lambda_{1}$ & $\lambda_{1}$,\, $\lambda_{3},\lambda_{4}\, (n=4,\,5)$ & $\lambda_{1}$
			\\\hline 
		\end{tabular}
	\end{table}
\end{thm}
As a consequence, we obtain the following characterization.
\begin{cor}
	Let $Z\subset \BP U_{\lambda_{i}}$
	be an irreducible Hermitian symmetric space. Then
	$Z$ is tangentially non-degenerate if and only if it has rank at most two.
\end{cor}

\begin{proof}
	If the Gaussian map $\psi$ is injective, then by Corollary~\ref{im_Gaussian},
	$
	\wedge^{2}U_{\lambda}^{\vee}
	\cong
	H^{0}(Z,\Omega_{Z}(2))
	$
	is irreducible by Lemma~\ref{irr_Gaussian}. Comparing with the classification table above, one checks that the corresponding (co)minuscule varieties are precisely those listed in Proposition~\ref{prim_class}(1), namely the IHSS's of rank at most two. The converse follows from Lemma~3.11.
\end{proof}

\subsubsection{$E_6/P_{3}$}
Let $\fg = \fe_{6}$ be the simple Lie algebra of type $E_{6}$, let $U = U_{\lambda_{3}}$ be the irreducible representation of highest weight $\lambda_{3}$, and let $Z \subset \BP U$ be the corresponding flag variety. We exclude this case by a direct computation.

\begin{lem}
	Let
	$
	U=U_{\lambda_{3}}
	$
	be the irreducible
	$
	\fe_{6}
	$
	-module. Then
	\[
	\wedge^{2}U^{\vee}
	\cong
	U_{\lambda_{1}+\lambda_{4}}^{\vee}
	\oplus
	U_{\lambda_{3}+\lambda_{6}}^{\vee}
	\oplus
	U_{\lambda_{1}+\lambda_{2}}^{\vee}
	\oplus
	U_{\lambda_{5}}^{\vee},
	\]
	while
	\[
	H^{0}(Z,\Omega_{Z}(2))
	\cong
	U_{\lambda_{1}+\lambda_{4}}^{\vee}
	\oplus
	U_{\lambda_{3}+\lambda_{6}}^{\vee}.
	\]
	In particular, the flag varieties
	$
	E_{6}/P_{3}\subset \BP U_{\lambda_{3}}
	$
	and
	$
	E_{6}/P_{5}\subset \BP U_{\lambda_{5}}
	$
	are not tangentially non-degenerate.
\end{lem}

\begin{proof}
	The first decomposition was computed by the Weyl character algorithm implemented in SageMath; see the Appendix. The second follows directly from the decomposition~\eqref{irr_decom}. Consequently, for the flag variety
	$
	E_6/P_3\subset \BP U_{\lambda_3},
	$
	the Gaussian map~\eqref{Gaussian_sequence} has nontrivial kernel. By the diagram involution of $E_6$, the same conclusion holds for
	$
	E_6/P_5\subset \BP U_{\lambda_5}.
	$
\end{proof}
We now give the proof of Proposition~\ref{prim_class} and Theorem~\ref{class_tang}.

\begin{proof}[Proof of Proposition~\ref{prim_class}]
	The listed varieties were shown to be tangentially non-degenerate in Section~3.1.3. Conversely, by Lemma~\ref{three_cases}, the highest weight module falls into one of the three cases considered there. The claim then follows from the corresponding classifications: Corollaries~3.24, 3.25 and 3.27 for the self-dual case, Corollary~3.30 for the (co)minuscule case, and Lemma~3.31 for the remaining $E_{6}$ case.
\end{proof}
\begin{proof}[Proof of Theorem~\ref{class_tang}]
The non-fundamental and fundamental cases follow from
Corollary~\ref{nprim} and Proposition~\ref{prim_class}, respectively.
The result then follows from the identifications of the adjoint varieties of
types \(A_n\) and \(C_n\) with \(Fl(1,n;n+1)\) and
\(v_2(\BP^{2n-1})\), respectively, together with
\(Gr(2,4)\simeq \BQ^4\) and \(OG(4,9)\simeq \BS_5\).
\end{proof}
\section{Preliminary results on the Spencer homomorphism}\label{prem_spen}
In this section, we study the sufficient conditions given in Theorem~\ref{tors_prin}. For flag varieties, the condition~\emph{(i)} is standard, and~\emph{(ii)} will be verified in Proposition \ref{surj_Psi_}. Thus the main point of this section is to analyze condition~\emph{(iii)}. For a smooth projective variety $Z \subset \BP U$, recall that
\[
\Xi_Z=
	\left\{
	\sigma\in \Hom(\wedge^2U,U)
	\ \middle|\
	\sigma(\hat z,\widehat T_zZ)\subset \widehat T_zZ
	\text{ for all }z\in Z
	\right\},
	\]
	and the Spencer homomorphism is defined by
	\begin{equation}\label{Spencer}
	\partial:
	\Hom(U,\fa \fu \ft(\widehat Z))
	\longrightarrow
	\Hom(\wedge^2U,U),
	\qquad
	\partial h(u_1,u_2)
	=
	h(u_1)(u_2)-h(u_2)(u_1).
	\end{equation}
\begin{thm}\label{reduction}
	Let \(Z\subset \BP U\) be a flag variety. If the Spencer condition $
	\Xi_Z \subset \operatorname{Im}(\partial)
	$ holds, then one of the following occurs:
	\begin{itemize}
		\item[(i)] \(Z\subset \BP U\) is tangentially non-degenerate;
		\item[(ii)] $Z$ is isomorphic to a Hermitian symmetric space and
$\wedge_Z^2\simeq \BC$.

	\end{itemize}
\end{thm}

Our strategy is to pull back the natural filtration on $\Xi_Z$ via the Spencer homomorphism and study the induced maps on its successive graded pieces. The lowest piece, closely related to $\wedge_Z^2$, already yields strong necessary conditions for the surjectivity of $\partial$. We introduce this filtration in Section~4.1, derive general obstructions to the surjectivity of $\partial^0$ in Section~4.2, and verify them for flag varieties in Section~4.3, thereby proving Theorem~\ref{reduction}.

\subsection{The filtration on $\partial^{-1}(\Xi_Z)$}
The filtration on $\Xi_{Z}$ has been defined and discussed in \cite{FH18a, LH24}:
\begin{defn}[{\cite[Definition 3.1]{FH18a}}]\label{filt_01}
	Let $Z\subset\BP U$ be a smooth projective variety.
	Denote by $\Omega_Z$ the cotangent bundle of $Z$ and set
	$\CO(1)=\CO_{\BP U}(1)|_Z$.
	For $u\in\widehat Z$, identify $T_{[u]}Z=\Hom(\BC u, \widehat{T}_{[u]} Z/\BC u)$ as before.
	
	\begin{itemize}
		\item[(i)]
		For $\sigma\in\Xi_Z$, define
		$\zeta(\sigma)\in H^0(Z,T_Z\otimes\Omega_Z(1))$
		by
		\[
		\zeta(\sigma)_{[u]}(u\otimes[u'])=[\sigma(u,u')],
		\qquad u'\in \widehat T_{[u]}Z.
		\]
		This defines a homomorphism
		\[
		\zeta\colon\Xi_Z\to H^0(Z,T_Z\otimes\Omega_ Z(1)).
		\]
		Set $\Xi_Z'=\Ker(\zeta)$.
		
		\item[(ii)]
		For $\sigma\in\Xi_Z'$, define
		$\eta_\sigma\in H^0(Z,\Omega Z(1))$ by
		\[
		\eta_{\sigma,[u]}([u'])\,u=\sigma(u,u'),
		\qquad u'\in \widehat{T}_{[u]} Z.
		\]
		This yields a homomorphism
		\[
		\eta\colon\Xi_Z'\to H^0(Z,\Omega_Z(1)).
		\]
		Set $\Xi_Z^0=\Ker(\eta)$.
	\end{itemize}
\end{defn}
Therefore, $\Xi_Z$ carries a natural filtration
$
\Xi_Z\supset\Xi'_Z\supset\Xi_Z^0.
$
More precisely, a direct calculation gives the following.

\begin{lem}
Keep the notation as above. Then
\begin{align*}
\Xi'_Z
&=
\{
\sigma\in\Hom(\wedge^2U,U)
\mid
\sigma(\hat z,\widehat T_zZ)\subset\langle\hat z\rangle
\text{ for all }z\in Z
\},\\
\Xi_Z^0
&=
\{
\sigma\in\Hom(\wedge^2U,U)
\mid
\sigma(\hat z,\widehat T_zZ)=0
\text{ for all }z\in Z
\}.
\end{align*}
In particular,
$
\Xi_Z^0=\wedge_Z^2\otimes U.
$
\end{lem}

Pulling this filtration back by the Spencer homomorphism \eqref{Spencer} gives the 
following diagram.
\begin{equation}\label{diag_Spencer}
	\begin{tikzcd}
		\partial^{-1}(\Xi_{Z}) \arrow[r,"\partial"] & \Xi_{Z} \arrow[r,"\zeta"] &  H^0(Z,T_Z\otimes\Omega_Z(1)) \\
		\partial^{-1}(\Xi_{Z}^{'}) \arrow[r,"\partial^{'}"] \arrow[u, hook]& \Xi_{Z}' \arrow[r,"\eta"]  \arrow[u, hook]& H^0(Z,\Omega_Z(1))\\
		\partial^{-1}(\Xi_{Z}^{0}) \arrow[r,"\partial^{0}"] \arrow[u, hook]& \Xi_{Z}^0 \arrow[u, hook]
	\end{tikzcd}
\end{equation}
Here $\partial'$ and $\partial^{0}$ denote the restrictions of $\partial$.
Theorem~\ref{reduction} can be reformulated more explicitly as follows.
\begin{thm}\label{refor_reduction}
	Let $Z \subset \BP U$ be a flag variety. Then the map $\partial^{0}$ is surjective if and only if $Z$ satisfies one of the conditions in Theorem~\ref{reduction}.
\end{thm}

We first identify the preimage \(\partial^{-1}(\Xi_Z)\).
\begin{prop}\label{pre_partial}
	Let $Z \subset \BP U$ be a smooth irreducible projective non-degenerate variety with $Z \not=\BP U$. Consider the natural homomorphism: 
	$$
	\Psi: Hom(U,\hfg) \rightarrow H^{0}(Z,T_Z(1)),
	$$	
	defined in Theorem \ref{tors_prin} (ii).
	Then $Ker(\Psi)=\partial^{-1}(\Xi_{Z})=\{h \in Hom(U,\hfg): h(u,u) \in \BC u, \forall u \in \widehat{Z}\}$.
\end{prop}
\begin{proof}
	By definition,
	\[
	\Ker(\Psi)
	=
	\{\, h \in \Hom(U,\hfg) : h(u,u) \in \BC u \text{ for all } u \in \widehat{Z} \,\}.
	\]
	Moreover,
	\begin{align*}
		\partial^{-1}(\Xi_{Z})&=\{h \in Hom(U,\hfg): h(u,w)-h(w,u) \in \widehat{T}_{[u]} Z, \, \forall u \in \widehat{Z}, \, \forall w \in \widehat{T}_{[u]} Z\}\\
		&=\{h \in Hom(U,\hfg): h(u,\widehat{T}_{[u]}Z) \subset \widehat{T}_{[u]} Z, \, \forall u \in \widehat{Z}\},
	\end{align*}
	where the second equality follows from the fact that
	$h(w,u) \in \widehat{T}_{[u]}Z$ since $h(w) \in \hfg$.
	
	For any $h \in \Ker(\Psi)$ and any nonzero $u \in \widehat{Z}$, the endomorphism $h(u)$ preserves the line $\BC u$, hence also preserves the affine tangent space $\widehat{T}_{[u]}Z$. This shows
	$
	\Ker(\Psi) \subset \partial^{-1}(\Xi_{Z}).
	$
	
	Conversely, let $h \in \partial^{-1}(\Xi_{Z})$. We claim that for a general point $z = [u] \in Z$, one has $h(u,u) \in \BC u$, equivalently, the one-parameter subgroup
	$
	H_{u}=\exp(\BC h(u))
	$
	fixes $z$. This implies $\Psi(h)=0$, since $T_Z(1)$ is a vector bundle. To prove the claim, consider the Gauss map
	\[
	g_{1} : Z \longrightarrow Gr(\dim Z+1,U), \qquad z \longmapsto \widehat{T}_{z}Z.
	\]
	 By \cite[Corollary 2.8]{Zak93}, since $Z \not= \BP U$, $g_1$ is birational onto its image. Hence, for a general point $z=[u]$, the subgroup $H_{u}$ fixes $z$ because it fixes $g_{1}(z)$ by definition. This proves the claim.
\end{proof}

Combining the above Proposition with Lemma 4.3, we can directly write $\partial^{-1}(\Xi_{Z}')$ and  $\partial^{-1}(\Xi_{Z}^0)$ as follows.
\begin{cor}\label{description_preimages}
	With the notation above,
	\[
	\partial^{-1}(\Xi_{Z}^{'})
	=
	\left\{
	h\in\Hom(U,\hfg)
	\ \middle|\
	\begin{array}{l}
		h(u,u)\in\BC u,\\[2pt]
		h(u,u')-h(u',u)\in \BC u,
	\end{array}
	\ \forall\,u\in\widehat Z,\ u'\in \widehat{T}_{[u]}Z
	\right\}.
	\]
	\[
	\partial^{-1}(\Xi_{Z}^{0})
	=
	\left\{
	h\in\Hom(U,\hfg)
	\ \middle|\
	\begin{array}{l}
		h(u,u)\in\BC u,\\[2pt]
		h(u,u')=h(u',u),
	\end{array}
	\ \forall\,u\in\widehat Z,\ u'\in \widehat{T}_{[u]} Z
	\right\}.
	\]
\end{cor}

Next we describe the composition maps $\zeta \circ \partial$ and $\eta \circ \partial'$ in \eqref{diag_Spencer} which will play a central role in verifying the surjectivities of $\partial$ and $\partial'$. Using the above descriptions, we realize these maps as taking global sections of morphisms between sheaves on $Z$ (see Lemmas \ref{CPadjoint} and \ref{etaparital'}).

Denote by
$
E(\hfg)=\hfg\otimes\CO_Z
$
the trivial vector bundle on $Z$. Then the map $\Psi$ is obtained by taking global sections of the evaluation morphism
$$
\widehat{\Psi}: E(\hfg)(1)\rightarrow T_Z(1).
$$
Accordingly, we follow the descriptions in Corollary \ref{description_preimages} to define subsheaves
\[
\widehat{\mathcal P}'_0
\subset
\widehat{\mathcal P}'
\subset
\widehat{\mathcal P}
\subset
E(\hfg),
\]
whose fibers record distinguished subspaces of the corresponding isotropy subalgebras.
\begin{defn}\label{hatP}
	Keep the notation above and set
	$
	\widehat{\CP}:=\ker(\widehat{\Psi})(-1)\subset E(\hfg).
	$
	\begin{itemize}
	\item[(i)] \(\widehat{\CP}\) is the subsheaf whose fiber at \([u]\in Z\) is the
	isotropy subalgebra
	\[
	\widehat{\CP}_{[u]}
	=
	\{A\in\hfg\mid A(u)\in\BC u\}.
	\]
	By Proposition~\ref{pre_partial}, we have
	$
	H^0(Z,\widehat{\CP}(1))
	=
	\partial^{-1}(\Xi_Z).
	$
	\item[(ii)]
	Define the evaluation morphism
	$
	\Theta:\widehat{\CP}(1)\longrightarrow\CO_Z(1)
	$
	fiberwise by
	\[
	\Theta_{[u]}:
	\Hom(\BC u,\widehat{\CP}_{[u]})
	\longrightarrow
	(\BC u)^\vee,
	\qquad
	\Theta_{[u]}(h)(u)\,u=h(u)(u).
	\]
	Taking global sections gives
	\[
	\theta:
	\partial^{-1}(\Xi_Z)
	\longrightarrow
	H^0(Z,\CO_Z(1))
	=
	U^\vee.
	\]
	\end{itemize}
\end{defn}
The isotropy subalgebra \(\widehat{\CP}_{[u]}\) acts naturally on
\(T_{[u]}Z\); denote this action by
\[
\ad_{[u]}:
\widehat{\CP}_{[u]}
\longrightarrow
\End(T_{[u]}Z).
\]
Then the composition \(\zeta\circ\partial\) admits the following description.

\begin{lem}\label{CPadjoint}
	With the notation above:
	\begin{itemize}
		\item[(1)]
		For any \(h\in\partial^{-1}(\Xi_Z)\), \(u\in\widehat Z\), and
		\(u'\in\widehat T_{[u]}Z\), one has
		\begin{equation}\label{eq:sym}
			h(u,u')+h(u',u)
			=
			\theta(h)(u)u'
			+
			\theta(h)(u')u
			\in\langle u,u'\rangle.
		\end{equation}
		
		\item[(2)]
The map \(\zeta\circ\partial\) in~\eqref{diag_Spencer} is induced on global
sections by
\[
\widehat{\zeta\circ\partial}:
\widehat{\CP}(1)
\longrightarrow
T_Z\otimes\Omega_Z(1).
\]
Its fiber at \([u]\in Z\) sends
\[
h\in\Hom(\BC u,\widehat{\CP}_{[u]})
\]
to the element of
\[
\langle u \rangle^{\vee}\otimes
End(\widehat{T}_{[u]}Z/\BC u)
\]
given by
\[
(u,\overline{u'})
\longmapsto
2h(u,u')
-
\Theta_{[u]}(h)(u)\,u'
\mod \BC u,
\qquad
\forall\,u'\in\widehat{T}_{[u]}Z.
\]
Equivalently, after twisting by \(\BC u\), it is defined by
\begin{equation}\label{adjoint_formula}
\widehat{\CP}_{[u]}
\longrightarrow
End(T_{[u]}Z),
\qquad
X\longmapsto
2\,\ad_{[u]}(X)
+
Tr(X|_{\BC u})\,\id_{T_{[u]}Z}.
\end{equation}
	\end{itemize}
\end{lem}
\begin{proof}
	For
	$
	u \in \widehat{Z},
	$
	we identify
	$
	T_{[u]}Z
	\cong
	\Hom(\BC u,\widehat{T}_{u}Z/\BC u)
	$
	as before. Let
	$
	u' \in \widehat{T}_{u}Z,
	$
	and choose a local arc
	$
	\gamma:\BC\rightarrow \hat{Z}
	$
	such that
	$
	\gamma(0)=u
	$
	and
	$
	\gamma'(0)=u'.
	$
	By Definition~\ref{hatP}, we have
	\[
	h(\gamma(t),\gamma(t))
	=
	\theta(h)(\gamma(t)) \cdot \gamma(t)
	\]
	for all $t$. Expanding at $t=0$, we obtain
	\[
	h(u,u)
	+
	t\big(h(u,u')+h(u',u)\big)
	+\cdots
	=
	\theta(h)(u)u
	+
	t\big(\theta(h)(u)u'
	+
	\theta(h)(u')u\big)
	+\cdots.
	\]
	Comparing the coefficients of $t$ gives~\eqref{eq:sym}.
	
	For the second claim, we compute from Definition \ref{filt_01} that
	\[
	(\zeta\circ \partial)(h)_{[u]}(u,\overline{u'})
	=
	h(u,u')-h(u',u)
	\mod \BC u.
	\]
	Using~\eqref{eq:sym}, this becomes
	\[
	(\zeta\circ \partial)(h)_{[u]}(u,\overline{u'})
	=
	2h(u,u')
	-
	\theta(h)(u)\cdot u'
	\mod \BC u.
	\]
	The conclusion follows directly from the definitions of
	$\Theta$ and $\theta.$
	The last claim follows from the induced action on
	$T_{[u]}Z;$
	note that the sign in the expressions changes when passing from the affine tangent space to the tangent space.
\end{proof}

Next we partially describe the composition map
$
\eta \circ \partial'
$
in~\eqref{diag_Spencer}.

\begin{defn}\label{hatP'} Keep the notation above. 
	Let
	$
	\widehat{\CP'}
	=
	\ker(\widehat{\zeta \circ \partial})(-1),
	$
	which is a subsheaf of $\widehat{\CP}$. We further define
	$
	\widehat{\CP'_0}
	=
	\ker(\Theta|_{\widehat{\CP'}(1)})(-1),
	$
	which is a subsheaf of $\widehat{\CP'}$. 
\end{defn}
\begin{lem}\label{lem_2.9}
Keep the notation above. Then $\widehat{\CP'_{0}}$ is the subsheaf of $\widehat{\CP'}$, whose local sections take values, at each point
$
[u]\in Z,
$
in the subspace
\[
(\widehat{\CP'_0})_{[u]}
=
\big\{
\CA \in \hfg
\;\big|\;
\CA(u)=0,
\quad
\CA(\widehat{T}_{[u]}Z)\subset \BC u
\big\}.
\]	
\end{lem}
\begin{proof}
	Take any local section $\CA$ of $\widehat{\CP'_0}$, then $\Theta(-1)(\CA)=0$ gives $\CA_{[u]}(u)=0$; and from Lemma \ref{CPadjoint}(2), $\widehat{\zeta \circ \partial}(-1)(\CA)=0$ gives $\CA_{[u]}(\widehat{T}_{[u]}Z) \subset \BC u$. 
\end{proof}
\begin{lem}\label{etaparital'} Keep the notation above.
	We have
	$
	H^{0}(Z,\widehat{\CP^{'}}(1))
	=
	\ker(\zeta \circ \partial)
	=
\partial^{-1}(\Xi_{Z}^{'}).
	$
	Moreover, the restriction of the composition map
	$
	\eta \circ \partial'
	$
	to the subspace
	$
	H^{0}(Z,\widehat{\CP^{'}_0}(1))
	$
	is obtained by taking global sections of the morphism
	\[
	\widehat{\eta\circ\partial'}:
	\widehat{\CP^{'}_0}(1)\rightarrow \Omega_Z(1),
	\]
	defined fiberwise by
	\[
	(\widehat{\eta\circ\partial'})_{[u]}:
	\Hom(\BC u,\widehat{\CP^{'}_0}_{[u]})
	\longrightarrow
	\Hom\big(
	\BC u \otimes
	\hat{T}_{[u]}Z/\langle u\rangle,
	\BC u
	\big), \quad
	(\widehat{\eta\circ\partial'})_{[u]}(h)(u,\overline{u'})
	=
	2h(u,u').
	\]
\end{lem}
\begin{proof}
The first claim follows from the above Lemmas and that  $\Xi_{Z}'=Ker(\zeta)$ by Definition \ref{filt_01}(i). The second claim also follows directly from applying Lemmas \ref{CPadjoint} and \ref{lem_2.9}.
\end{proof}
We summarize the above constructions in the following sequences:
\begin{align}
	0 & \longrightarrow \widehat{P}(1) \longrightarrow E(\hfg)(1) \xrightarrow{
	\widehat{\Psi}} T_Z(1) \label{eq:2.30}\\	
	0 &\longrightarrow \widehat{\CP'}(1)
	\longrightarrow \widehat{\CP}(1)
	\xrightarrow{\widehat{\zeta \circ \partial}}
	T_Z \otimes \Omega_Z(1), \label{eq:2.3}\\
	0 &\longrightarrow \widehat{\CP'_0}(1)
	\longrightarrow \widehat{\CP'}(1)
	\xrightarrow{\Theta|_{\widehat{\CP'}(1)}}
	\CO(1),\label{eq:2.4}\\
	& \phantom{0\longrightarrow}
	\widehat{\CP'_0}(1)
	\xrightarrow{\widehat{\eta\circ\partial'}}
	\Omega_Z(1). \label{eq:2.5}
\end{align}

\subsection{Obstructions}
In this subsection, we study the diagram~\eqref{diag_Spencer} in greater detail and derive conditions for the surjectivity of the bottom map $\partial^{0}$ (see Corollary~\ref{cor_classify}). These conditions will later be shown to be both necessary and sufficient for flag varieties in Proposition~\ref{HX}. 

The analysis is based on the \emph{two-step} decomposition of
$
\partial^{-1}(\Xi_{Z}^{'})
$
induced by~\eqref{eq:2.4}. Taking global sections of~\eqref{eq:2.4}, we obtain a left exact sequence
\begin{equation}\label{exac_seq2}
	0 \longrightarrow \widetilde{U}
	\longrightarrow
	\partial^{-1}(\Xi_{Z}')
	\xrightarrow{~\theta~}
	H^{0}(Z,\CO_{Z}(1))
	\cong U^{\vee},
\end{equation}
where we denote
$
\widetilde{U}
=
H^{0}(Z,\widehat{\CP'_0}(1))
$ and
by Lemma~\ref{lem_2.9},
\[
\widetilde{U}
=
\left\{
h\in \Hom(U,\hfg)
\ \middle|\
h(u)(u)=0,\;
h(u)(\widehat{T}_{[u]}Z)\subset \BC u
\text{ for all }
u\in \widehat{Z}
\right\}.
\]

For the map $\theta$, it naturally leads to the notion of Euler-symmetric varieties, which we now recall.

\begin{defn}[{\cite[Definition 2.1]{euler_sym}}]\label{euler_symdefn}
	Let \(Z \subseteq \BP U\) be a projective subvariety. A \(\BC^*\)-action on \(Z\) is called of Euler type at a nonsingular fixed point \(x\) if the induced isotropy action on the tangent space \(T_xZ\) is by scalar multiplication.
		We say that \(Z \subseteq \BP U\) is an Euler-symmetric variety if, for a general point \(x\in Z\), there exists a \(\BC^*\)-action of Euler type at \(x\), induced by a multiplicative subgroup of \(\operatorname{GL}(U)\).
\end{defn}

The following proposition is the main general observation of this subsection. 
\begin{prop}\label{ext_seq}
	Let $Z \subset \BP U$ be a smooth non-degenerate linearly normal projective variety with $Z \not= \BP U$.
	Then the following hold:
	\begin{itemize}
		\item[(1)]
		The map $\theta$ in \eqref{exac_seq2} is nonzero if and only if  $Z$ is Euler-symmetric. 
		\item[(2)]
		Restricting~\eqref{exac_seq2} to
		$
		\partial^{-1}(\Xi_{Z}^{0})
		$
		yields
		\begin{equation}\label{exac_seq3}
			0
			\longrightarrow
			\widetilde{U}\cap \partial^{-1}(\Xi_{Z}^{0})
			\longrightarrow
			\partial^{-1}(\Xi_{Z}^{0})
			\xrightarrow{~\theta~}
			H^{0}(Z,\CO_{Z}(1))
			\cong U^{\vee},
		\end{equation}
		where $\theta$ is nonzero whenever it is nonzero in \eqref{exac_seq2}.
		
		\item[(3)]
		We have
		\[
		\widetilde{U}\cap \partial^{-1}(\Xi_{Z}^{0})
		=
		\bigl\{
		h\in \Hom(U,\hfg)
		\bigm|
		h(u,u')=0
		\ \forall\,u\in \widehat{Z},\,
		u'\in \widehat{T}_{[u]}Z
		\bigr\}.
		\]
		Moreover, the restriction map
		\[
		\partial^{0}:
		\widetilde{U}\cap \partial^{-1}(\Xi_{Z}^{0})
		\longrightarrow
		\Xi_{Z}^{0}
		\]
		is injective.
	\end{itemize}
\end{prop}
Consequently, a key necessary condition for the surjectivity of $\partial^{0}$ can be deduced:
\begin{cor}\label{cor_classify}
	Let $Z \subset \mathbb{P}U$ be as in Proposition~\ref{ext_seq}, and assume that 
	$\partial^{0}$ in~\eqref{diag_Spencer} is surjective. Then either $Z \subset \mathbb{P}U$ is Euler-symmetric, or the restriction map
	\[
	\widetilde{U} \cap \partial^{-1}(\Xi_{Z}^{0}) \xrightarrow{\ \partial^{0}\ } \Xi_{Z}^{0}
	\]
	is an isomorphism.
	
	In particular, if $\widetilde{U} \cap \partial^{-1}(\Xi_{Z}^{0}) = 0$, then $\partial^{0}$ is surjective if and only if one of the following occurs:
	\begin{itemize}
		\item[(i)] $Z \subset \mathbb{P}U$ is tangentially non-degenerate;
		\item[(ii)] $Z \subset \mathbb{P}U$ is Euler-symmetric and tangentially non-degenerate, and the first prolongation space 
		\[
		\hfg^{(1)} = (\Sym^{2}U^{\vee} \otimes U) \cap \Hom(U,\hfg)
		\]
		is nonzero;
		\item[(iii)] $Z \subset \mathbb{P}U$ is Euler-symmetric with $\hfg^{(1)}=0$, moreover
		the map $\theta$ in~\eqref{exac_seq3} is an isomorphism, and $\wedge_{Z}^{2} \cong \BC$ is one-dimensional.
	\end{itemize}
\end{cor}
\begin{proof}
	Assume that the map 
	$
	\partial^{0}: \partial^{-1}(\Xi_{Z}^{0}) \rightarrow \Xi_{Z}^{0}
	$
	in \eqref{diag_Spencer} is surjective. Then, from the exact sequence~\eqref{exac_seq3}, either $\theta$ is nonzero, which is equivalent to $Z \subset \BP U$ being Euler-symmetric by Proposition~\ref{ext_seq}(1), or $\theta=0$, in which case 
	$
	\widetilde{U} \cap \partial^{-1}(\Xi_{Z}^{0}) = \partial^{-1}(\Xi_{Z}^{0}).
	$
	In the latter case, $\partial^{0}$ must be an isomorphism, since it is injective by Proposition~\ref{ext_seq}(3). This proves the first claim.
	
	Now assume further that 
	$
	\widetilde{U} \cap \partial^{-1}(\Xi_{Z}^{0})=0.
	$
	Then from~\eqref{exac_seq3}, we obtain 
	\[
	\partial^{-1}(\Xi_{Z}^{0}) \cong \operatorname{Im}(\theta).
	\]
	
	If $\theta=0$, then $\partial^{-1}(\Xi_{Z}^{0})=0$, and hence $\Xi_{Z}^{0}=0$, that is, $Z$ is tangentially non-degenerate.
	
	If $\theta \neq 0$, then $\partial^{-1}(\Xi_{Z}^{0}) \neq 0$, and $Z$ is Euler-symmetric. We are led to two cases. If $\Xi_{Z}^{0}=0$, then $\ker(\partial)$ is nonzero, which is equivalent to the existence of a nonzero element in
	\[
	\hfg^{(1)}= \Sym^{2}U^{\vee} \otimes U \cap \Hom(U,\hfg).
	\]
	Otherwise, assume that $\Xi_{Z}^{0} \neq 0$. In this case, we note that
	\[
	\dim(\Xi_{Z}^{0}) = \dim(\wedge_{Z}^{2} \otimes U) \geqslant \dim(U),
	\]
	while
	\[
	\dim(\partial^{-1}(\Xi_{Z}^{0})) = \dim(\operatorname{Im}(\theta)) \leqslant \dim(U^{\vee}) = \dim(U).
	\]
	It follows that both $\theta$ and $\partial^{0}$ are isomorphisms, and that $\wedge_{Z}^{2} = \BC$ is one-dimensional.
	
Conversely, assume that one of the three conditions holds. Then the above argument can be directly reversed to deduce the surjectivity of $\partial^{0}$.
\end{proof}

We now establish Proposition \ref{ext_seq}. We first treat the parts involving Euler-symmetric varieties.

\begin{lem}\label{iml}
	Keep notation as in Proposition \ref{ext_seq}. The map $\theta$ is nonzero only if $Z$ is Euler-symmetric.
\end{lem}

\begin{proof}
	Suppose $\theta(h) \neq 0$ for some $h \in \partial^{-1}(\Xi_{Z}')$. Let $[u] \in Z$ be a general point. Recall from \eqref{eq:sym} that 
$h(u,u') + h(u',u)
=\theta(h)(u)u' + \theta(h)(u')u,
$
for any $u' \in \widehat{T}_{[u]}Z$.
	Combining with Corollary \ref{description_preimages}, we obtain
	\[
	2h(u,u') - \theta(h)(u)u' \in \BC u.
	\]
Thus $h(u)$ acts on $T_{[u]}Z$ by the scalar
$-\theta(h)(u)/2$ which is nonzero by assumption.
 Consequently, its semisimple part $h(u)_s$ is also nonzero and acts on $T_{[u]}Z$ by the same nonzero scalar.
	Now consider the one-parameter subgroup $\exp(\BC h(u)_s) \subset \Aut(\widehat{Z})$. Its Zariski closure is a subtorus that fixes $[u]$ and acts on $T_{[u]}Z$ via a nontrivial character. Choosing a general $\BG_m$-subgroup of this torus yields a one-parameter subgroup that fixes $[u]$ and acts on $T_{[u]}Z$ by scalar multiplication. Hence $Z \subset \BP U$ is Euler-symmetric in the sense of Definition \ref{euler_symdefn}.
\end{proof}
For the converse, we recall some key properties of Euler-symmetric varieties.

\begin{thm}[{\cite{euler_sym}}]\label{euler_sym}
	Let $Z \subset \BP U$ be an Euler-symmetric variety of dimension $n$. For a general point $z = [u] \in Z$, there exists a subgroup $\BC^{*} \subset \Aut(\widehat{Z})$ whose action on $U$ yields a weight decomposition
	$$
	U = U_{-2} \oplus U_{-1} \oplus U_{0} \oplus \cdots \oplus U_{r-2},
	$$
	where $U_{k}$ denotes the weight subspace of weight $k$, with $U_{-2} = \BC u$ and $U_{-2} \oplus U_{-1} = \widehat{T}_{[u]}Z$. Moreover, there exists an abelian subalgebra $\fn_{1} \subset \hfg$ satisfying:
	\begin{itemize}
		\item[(1)] $\fn_{1}$ is $\BC^{*}$-invariant and has weight $1$;
		\item[(2)] $\fn_{1}.U_{k} \subset U_{k+1}$ for every $k$, so that one obtains an iteration map
		$$
		f_{k} \colon \Sym^{k}(\fn_{1}) \longrightarrow \Hom(U_{-2}, U_{k-2}), \qquad f_{k}(X^{k})(u) = X^{k}(u).
		$$
		The map $f_{1}$ is an isomorphism, and $f_{k}$ is surjective for every $k \geq 1$.
	\end{itemize}
\end{thm}
\begin{proof}
	The existence of the $\BC^{*}$-subgroup follows from the construction in~\cite[Theorem~3.7]{euler_sym}. Moreover, loc.~cit.~constructs a vector group action on $Z$
	arising from a linear representation $W=\BG_{a}^{n}\subset \Aut(\hat{Z}).$
	Setting $\fn_{1}=\Lie(W),$ the stated properties follow directly from the construction.
\end{proof}

\begin{lem}
	Let $Z \subset \BP U$ be as in Theorem~\ref{euler_sym}, and consider the adjoint action of $\BC^{*}$ on $\hfg$ with weight space decomposition $\hfg = \bigoplus_{k} \hfg_{k}$. Then $\hfg_{k} = 0$ for every $k \geq 2$.
\end{lem}
\begin{proof}
	Suppose otherwise, and let $d\ge2$ be maximal such that
$\widehat\fg_d\neq0$. Then $d \geqslant 2$. We have $\hfg_{d}.U_{k} \subset U_{k+d}$ for any $k$. It follows that
	$$
	\hfg_{d}.U_{-2} \subset U_{d-2} \cap \widehat{T}_{[u]}Z=0.
	$$
	On the other hand we have $[\hfg_{d},\hfg_{1}] \subset \hfg_{d+1}=0$ by assumption.  Since $\fn_{1} \subset \hfg_{1}$ by Theorem \ref{euler_sym}(1), the action of $\hfg_{d}$ commutes with $\fn_{1}$, we deduce that
	$$
	\hfg_{d}.U_{k-2}=\hfg_{d}.(f_{k}(\Sym^{k}\fn_{1}). U_{-2})=\Sym^{k}\fn_{1}.(\hfg_{d}.U_{-2})=0,
	$$
	where the first equality follows from Theorem \ref{euler_sym}(2). Thus $\widehat\fg_d$ acts trivially on $U$, so
$\widehat\fg_d=0$, a contradiction.
\end{proof}
Next we prove the converse of Lemma~\ref{iml} by constructing explicitly an element
$h\in \partial^{-1}(\Xi_{Z}^{0})$ such that $\theta(h)\neq 0.$

\begin{lem}\label{const_euler_inverse}
	Let $Z\subset\BP U$ be as in Theorem~\ref{euler_sym}, and assume that $Z$ is smooth. Let $H_u\in\hfg$ be a nonzero infinitesimal generator of the $\BC^*$-subgroup, so that
$
H_u|_{U_k}=k\cdot\id_{U_k}
$
for all $k$. Define
$$
h:U\to\hfg,\qquad
\lambda u+X(u)+\sum_{k=0}^{r-2}v_k
\longmapsto
\lambda H_u+[X,H_u],
$$
where $X\in\fn_1$, $u\in U_{-2}$, and $v_k\in U_k$. Then $h$ is $\fn_1$-invariant, lies in $\partial^{-1}(\Xi_Z^0)$, and satisfies $\theta(h)\neq0$.

\end{lem}

\begin{proof}
	For any
	$
	Y\in \fn_{1}
	$
	and
	$
	v=\lambda u+X(u)+\sum_{k=0}^{r-2}v_{k}\in U,
	$
	we compute
	\[
	[Y,h(v)]
	=
	[Y,\lambda H_{u}+[X,H_{u}]]
	=
	\lambda [Y,H_{u}],
	\]
	where 
	$
	[Y,[X,H_{u}]]=0
	$
	since
	$
	H_{u} \in \fg_{0}
	$
	and
	$
	[Y,[X,H_{u}]] \in \fg_{2}=0
	$
	from the above Lemma.
	On the other hand,
	\[
	Y(v)
	=
	\lambda\,Y(u)
	+
	Y(X(u))
	+
	\sum_{k=1}^{r-1}Y(v_{k-1}),
	\]
	hence
	$
	h(Y(v))
	=
	[\lambda Y,H_{u}]
	=
	[Y,h(v)].
	$
	Thus
	$
	h
	$
	is
	$
	\fn_{1}
	$
	-invariant.
	
	Since
	$
	H_{u}|_{U_{k}}=k \cdot \id,
	$
	we also have
	$
	ad(H_{u})|_{\fg_{k}}=k \cdot \id
	$
	for all
	$
	k.
	$
	Then for any
	$
	v=\lambda u+X(u)\in \hat{T}_{[u]}Z,
	$
	\[
	\partial(h)(u,v)
	=
	\partial(h)(u,X(u))
	=
	H_u(X(u))-[X,H_{u}](u)
	=
	-X(u)-(-X(u))
	=
	0.
	\]
	As $h$
is
$\fn_{1}$-invariant, it is invariant under
$W=\exp(\fn_{1}),$
which acts on $Z$
with an open orbit by~\cite[Theorem~3.7]{euler_sym}. Hence
$
\partial(h)(u',\widehat T_{[u']}Z)=0
$
along the open $W$-orbit. Since
$
\zeta\circ\partial(h)
$
in \eqref{Spencer} is a global section of vector bundle on the smooth variety $Z$, it vanishes identically. Then similarly $
\eta\circ\partial(h)=0
$ and thus
$
h\in\partial^{-1}(\Xi_Z^0).
$
Finally, $$
h(u)(u)
=H_u(u)=
-2u
\neq0,
$$
and hence
$
\theta(h)\neq0.
$
\end{proof}
We now complete the proof of Proposition \ref{ext_seq}.
\begin{proof}[Proof of Proposition \ref{ext_seq}]
	(1) follows from the preceding lemmas. For (2), we consider the restriction of \eqref{exac_seq2} to $\partial^{-1}(\Xi_Z^0)$. Suppose that $\theta$ is nonzero in~\eqref{exac_seq2}. Then by part~(1), $Z$ is Euler-symmetric. Hence, by Lemma~\ref{const_euler_inverse}, there exists $h \in \partial^{-1}(\Xi_Z^0)$ such that $\theta(h) \neq 0$. This proves the claim. 
	
	For (3), assume that $h \in \widetilde{U} \cap \partial^{-1}(\Xi_Z^0)= \ker(\theta) \cap \partial^{-1}(\Xi_Z^0).$
	By definition, we have $h(u,u') = h(u',u)$ and $\theta(h)(u)=0$ for any $u \in \hat{Z}$ and $u' \in \hat{T}_{[u]}Z$. Combining this with \eqref{eq:sym}, we obtain $h(u,u')=0$. This proves the first claim.
	Finally, take any $h \in \widetilde{U} \cap \partial^{-1}(\Xi_Z^0)$ such that $\partial(h)=0$. Then
	$
	h \in \Sym^2 U^{\vee} \otimes U \;\cap\; \Hom(U,\hfg),
	$
	hence $h \in \hfg^{(1)}$, the first prolongation of the Lie subalgebra $\hfg \subset \fg \fl(U)$. By \cite[Proposition 2.1.3]{HM05}, we must have $h=0$, since $h(u,u)=\theta(h)(u)=0$ for any $u \in \widehat{Z}$.
\end{proof}

\subsection{Flag varieties}
We now specialize the preceding discussion to flag varieties. In this setting, the bottom kernel
$
\widetilde{U}\cap \partial^{-1}(\Xi_Z^0)
$
vanishes and we have a complete description of the decomposition $\partial^{-1}(\Xi_{Z}')$, see Proposition \ref{HX} and Corollary \ref{cor_end}.

We fix the following notation. Set $G = \Aut^{0}(Z)$, $\fg = \Lie(G)$ and
$
\hfg = \fg \oplus \BC\, \mathrm{id}_{U}.
$
Accordingly, for any $h \in \Hom(U,\hfg)$, we write
$
h = h^{*} + h^{\circ},
$
where $h^{*} \in \Hom(U,\fg)$ and $h^{\circ} \in \Hom(U,\BC id_{U})$.

We decompose $G = \prod_{\nu=1}^{r} G_{[\nu]}$ into simple factors, and accordingly
\[
\fg = \bigoplus_{\nu=1}^{r} \fg_{[\nu]}, 
\qquad 
\lambda = \sum_{\nu=1}^{r} \lambda_{[\nu]}, 
\qquad 
U_{\lambda} = \bigotimes_{\nu} U_{\lambda_{[\nu]}},
\qquad
Z = \prod_{\nu=1}^{r} Z_{\nu}
= \prod_{\nu=1}^{r} G_{[\nu]}/P_{[\nu]},
\]
where $P_{[\nu]}$ is the parabolic subgroup corresponding to $\lambda_{[\nu]}$. We set $\fp_{[\nu]} = \Lie(P_{[\nu]})$ and write $u_{\lambda} = \bigotimes_{\nu} u_{\lambda_{[\nu]}}$ for a highest weight vector. Finally, for each pair $(\fg_{[\nu]}, \fp_{[\nu]})$, we consider the associated parabolic grading in Definition \ref{para_grad}:
\[
\fg_{[\nu]} = \bigoplus_{k=-d_{\nu}}^{d_{\nu}} \fg_{[\nu,k]}.
\]

We now identify the subsheaves of
$
E(\hfg)=\hfg\otimes\CO
$
introduced in Section~4.1 as homogeneous subbundles. We first treat $\widehat{\CP}$ and $\widehat{\CP'_0}$:
\begin{lem}\label{bundle_CP} Assume that $Z \not= \BP U$.
	Then $\widehat{\CP}$ and $\widehat{\CP_{0}'}$ are homogeneous subbundles represented by
	\[
	\widehat{\CP}
	\cong
	E(\fp \oplus \BC\,\mathrm{id}_{U}),
	\qquad
	\widehat{\CP'_0}
	\cong
	E(\ker(\phi_{\fp})),
	\]
	where
	$
	\phi_{\fp}:\fp \rightarrow \End(\fg/\fp)
	$
	is the isotropy representation defined in~\eqref{phip_defn}. Moreover,
	\[
	\ker(\phi_{\fp})
	=
	\bigoplus_{\nu}\fg_{[\nu,d_{\nu}]}.
	\]
\end{lem}
\begin{proof}
	By Definitions~\ref{hatP} and~\ref{hatP'}, both
	$
	\widehat{\CP}
	$
	and
	$
	\widehat{\CP'_0}
	$
	are homogeneous bundles. Hence
	\[
	\widehat{\CP}
	\cong
	E(\widehat{\CP}_{[u_{\lambda}]}),
	\qquad
	\widehat{\CP'_0}
	\cong
	E(\widehat{\CP'_0}_{[u_{\lambda}]}).
	\]
	
	From Definition~\ref{hatP},
	$
	\widehat{\CP}_{[u_{\lambda}]}
	$
	is precisely the stabilizer of the line
	$
	\BC u_{\lambda},
	$
	hence
	$
	\widehat{\CP}_{[u_{\lambda}]}
	=
	\fp \oplus \BC\,\mathrm{id}_{U}.
	$
	
	Similarly, Lemma~\ref{lem_2.9} gives
	\[
	\widehat{\CP'_0}_{[u_{\lambda}]}
	=
	\{
	X \in \hfg
	\mid
	X(u_{\lambda})=0,
	\quad
	X(\widehat{T}_{[u_{\lambda}]}Z) \subset \BC u_{\lambda}
	\}.
	\]
	We claim that
	$
	\widehat{\CP'_0}_{[u_{\lambda}]}
	=
	\ker(\phi_{\fp}).
	$
	Indeed, let
	$
	X \in \widehat{\CP'_0}_{[u_{\lambda}]},
	$
	and write
	\[
	X=X^{0}\cdot \mathrm{id}_{U}+X^{*},
	\qquad
	X^{*}=\sum_{\nu}X^{*}_{\nu},
	\]
	where
	$
	X^{*}_{\nu}\in \fp_{[\nu]}
	$
	for each $\nu$. Since
	$
	X(u_{\lambda})=0,
	$
	it induces an infinitesimal action on
	$
	T_{[u_{\lambda}]}Z.
	$
	
	Recall that
	\[
	T_{[u_{\lambda}]}Z
	\cong
	\fg/\fp
	\cong
	\bigoplus_{\nu}\fg_{[\nu]}/\fp_{[\nu]}.
	\]
	Since the identity
	$
	\mathrm{id}_{U}
	$
	acts trivially on
	$
	\fg/\fp,
	$
	the isotropy action of $X$ is induced by $X^{*}$ and is given by
	\[
	\phi_{\fp}(X^{*})
	=
	\sum_{\nu}
	\phi_{\fp_{[\nu]}}(X^{*}_{\nu})
	\in
	\bigoplus_{\nu}
	\End(\fg_{[\nu]}/\fp_{[\nu]}).
	\]
	By assumption,
	$
	X
	$
	annihilates
	$
	T_{[u_{\lambda}]}Z,
	$
	hence
	$
	\phi_{\fp}(X^{*})=0.
	$
	Therefore
	$
	\phi_{\fp_{[\nu]}}(X^{*}_{\nu})=0
	$
	for every $\nu$. By Lemma~\ref{ker_isotropy}, we obtain
	$
	X^{*}_{\nu}\in \fg_{[\nu,d_{\nu}]}
	$
	for all $\nu$. In particular,
	$
	X^{*}
	$
	is nilpotent, hence annihilates the highest weight line
	$
	\BC u_{\lambda}.
	$
	But since
	$
	X(u_{\lambda})=0,
	$
	we have also
	$
	X^{0}=0.
	$
	Therefore
	$
	X=X^{*}\in \ker(\phi_{\fp}),
	$
	proving the claim. The final assertion also follows from the above discussion.
\end{proof}
Next we describe the bundle
$
\widehat{\CP'},
$
together with the sequence \eqref{eq:2.4}. By Definition~\ref{hatP'}, $\widehat{\CP'}$ is a homogeneous bundle and we can write
$
\widehat{\CP'}=E((\widehat{\CP'})_{[u_{\lambda}]}).
$
Moreover the formula \eqref{adjoint_formula} gives
\begin{equation}\label{eq:2.9}
	(\widehat{\CP'})_{[u_{\lambda}]}
	=
	\left\{
	X \in \fp \oplus \BC\,\mathrm{id}_{U}
	\ \middle|\
	ad(X)|_{T_{[u_{\lambda}]}Z}
	=
	-\frac{1}{2}Tr(X|_{\BC u_{\lambda}})
	\cdot
	\mathrm{id}_{T_{[u_{\lambda}]}Z}
	\right\},
\end{equation}
where $X$ acts on
$
T_{[u_{\lambda}]}Z
$
through the isotropy representation.

By Definition~\ref{hatP}, after twisting by
\(\langle u_{\lambda}\rangle^{\vee}\), the sequence~\eqref{eq:2.4}
is given by the following sequence of \(P\)-modules:
\begin{equation}\label{local_P'}
	0
	\longrightarrow
	\bigl(\widehat{\CP'_0}\bigr)_{[u_\lambda]}
	\longrightarrow
	\bigl(\widehat{\CP'}\bigr)_{[u_\lambda]}
	\xrightarrow{\ \operatorname{Tr|}_{\BC u_\lambda}\ }
	\BC.
\end{equation}
In the next lemma we describe a $\fg_{0}$-equivariant section of $\operatorname{Tr}|_{\BC u_{\lambda}}$ whenever it is nonzero.

\begin{lem}\label{lem_section_tr}
	Keep the notation above.
	\begin{enumerate}
		\item
		The space
		$$
		(\ft\oplus \BC\,\mathrm{id}_{U})
		\cap
		\bigl(\widehat{\CP'}\bigr)_{[u_{\lambda}]}
		$$
		has dimension at most one. If it is nonzero, then it is a trivial
		$\fg_{0}$-module, and each of its nonzero elements acts on
		\(T_{[u_{\lambda}]}Z\) by a nonzero scalar.
		
		\item
		The map
		\[
		\operatorname{Tr}|_{\BC u_{\lambda}}
		:
		\bigl(\widehat{\CP'}\bigr)_{[u_{\lambda}]}
		\longrightarrow \BC
		\]
		is nonzero if and only if its restriction to
		$
		(\ft\oplus \BC\,\mathrm{id}_{U})
		\cap
		\bigl(\widehat{\CP'}\bigr)_{[u_{\lambda}]}
		$
		is nonzero.
	\end{enumerate}
\end{lem}

\begin{proof}
	By Lemma~\ref{bundle_CP},
	$
	\bigl(\widehat{\CP'_0}\bigr)_{[u_{\lambda}]}
	=
	\ker(\phi_{\fp})
	\subset \fn_{\fp}.
	$
	Consequently,
	\[
	(\ft\oplus \BC\,\mathrm{id}_{U})
	\cap
	\bigl(\widehat{\CP'_0}\bigr)_{[u_{\lambda}]}
	=
	0.
	\]
	It follows from~\eqref{local_P'} that the restriction of
	\(\operatorname{Tr}|_{\BC u_{\lambda}}\) to
	$
	(\ft\oplus \BC\,\mathrm{id}_{U})
	\cap
	\bigl(\widehat{\CP'}\bigr)_{[u_{\lambda}]}
	$
	is injective. In particular, this space has dimension at most one.
	Suppose that it is nonzero, and choose a generator
	\[
	H=H^{0}\,\mathrm{id}_{U}+H^{*},
	\qquad
	H^{0}\in \BC,\quad H^{*}\in \ft.
	\]
	By injectivity,
	$
	\operatorname{Tr}(H|_{\BC u_{\lambda}})\neq 0.
	$
	Equation~\eqref{eq:2.9} then gives
	\[
	\phi_{\fp}(H^{*})
	=
	-\frac{1}{2}
	\operatorname{Tr}(H|_{\BC u_{\lambda}})
	\,\mathrm{id}_{\fg/\fp},
	\]
	so \(H\) acts on \(T_{[u_{\lambda}]}Z\simeq \fg/\fp\) by a
	nonzero scalar.	Moreover, for every \(Y\in \fg_{0}\),
	\[
	\phi_{\fp}([H^{*},Y])
	=
	[\phi_{\fp}(H^{*}),\phi_{\fp}(Y)]
	=
	0.
	\]
	Since
	\(\ker(\phi_{\fp})\cap \fg_{0}=0\), it follows that
	\([H^{*},Y]=0\) for all \(Y\in \fg_{0}\). Thus \(H^{*}\) lies in
	the center of \(\fg_{0}\).
	This proves~(1).
	
	For~(2), suppose that
	$
	\operatorname{Tr}(X|_{\BC u_{\lambda}})\neq 0
	$
	for some
	\(X\in \bigl(\widehat{\CP'}\bigr)_{[u_{\lambda}]}\).
	Write
	\[
	X
	=
	X^{0}\,\mathrm{id}_{U}
	+
	X_{\ft}
	+
	X_{\fn_{\fp}},
	\]
	where
	$
	X^{0}\in \BC,
	X_{\ft}\in \ft$ and $
	X_{\fn_{\fp}}\in \fn_{\fp}.
	$
	By~\eqref{eq:2.9}, the induced action on
	\(T_{[u_{\lambda}]}Z\simeq \fg/\fp\) satisfies
	\[
	\phi_{\fp}(X_{\ft}+X_{\fn_{\fp}})
	=
	-\frac{1}{2}
	\operatorname{Tr}(X|_{\BC u_{\lambda}})
	\,\mathrm{id}_{\fg/\fp}.
	\]
	
	Consider the grading
	$
	\fg/\fp\simeq \bigoplus_{k<0}\fg_{k}.
	$
	The endomorphism \(\phi_{\fp}(X_{\ft})\) preserves each graded
	piece, whereas \(\phi_{\fp}(X_{\fn_{\fp}})\) strictly raises the
	grading. Since their sum is a scalar endomorphism, the
	positive-degree component must vanish. Therefore
	$
	\phi_{\fp}(X_{\fn_{\fp}})=0
	$ and
	hence
	$
	X_{\fn_{\fp}}
	\in
	\ker(\phi_{\fp})
	=
	\bigl(\widehat{\CP'_0}\bigr)_{[u_{\lambda}]}.
	$
	Subtracting this element from \(X\), we obtain
	\[
	X^{0}\,\mathrm{id}_{U}+X_{\ft}
	\in
	(\ft\oplus \BC\,\mathrm{id}_{U})
	\cap
	\bigl(\widehat{\CP'}\bigr)_{[u_{\lambda}]}.
	\]
	Its restriction to \(\BC u_{\lambda}\) has the same nonzero trace
	as that of \(X\). Thus the restriction of
	\(\operatorname{Tr}|_{\BC u_{\lambda}}\) to the 
	one-dimensional space is nonzero.
\end{proof}

The following is the main result of this subsection.

\begin{prop}\label{HX}
	Let \(Z\subset \mathbb P U\) be a flag variety with
	\(Z\neq \mathbb P U\). Keep the notation above. Then:
	\begin{enumerate}
		\item
		In~\eqref{exac_seq2}, the map \(\theta\) is nonzero if and only if
		\(Z\) is a Hermitian symmetric space, equivalently if and only if $
		\Theta|_{\widehat{\CP'}(1)}\neq 0
		$
		in~\eqref{eq:2.4}. In this case, both
		\(\theta\) and \(\Theta|_{\widehat{\CP'}(1)}\) are surjective.
		
		\item
		In~\eqref{eq:2.5}, the morphism
		\[
		\widehat{\eta\circ\partial'}
		:
		\widehat{\CP'_0}(1)
		\longrightarrow
		\Omega_Z(1)
		\]
		is injective. More precisely, it is induced, after twisting by
		\(\langle u_{\lambda}\rangle^{\vee}\), by the inclusion of
		\(P\)-modules
		\[
		\bigoplus_{\nu}\fg_{[\nu,d_{\nu}]}
		\hookrightarrow
		\fn_{\fp}
		\cong
		(\fg/\fp)^{\vee},
		\]
		where the last identification is induced by the Killing form.
		
		\item
		The restricted map
		\[
		\eta\circ\partial'|_{\widetilde U}
		:
		\widetilde U
		\longrightarrow
		H^{0}(Z,\Omega_Z(1))
		\]
		is injective. Equivalently, in~\eqref{exac_seq3},
		$
		\widetilde U\cap \partial^{-1}(\Xi_Z^{0})=0.
		$
	\end{enumerate}
\end{prop}

\begin{proof}
	We first prove~(1). By Proposition~\ref{ext_seq},
	\(\theta\neq 0\) if and only if \(Z\) is Euler-symmetric. For flag varieties, being Euler-symmetric is equivalent to being Hermitian symmetric by~\cite[Example~3.13]{euler_sym}, proving the first equivalence.
	
	 For the second, first recall that $\theta$ in \eqref{exac_seq2} is obtained by taking global sections of $\Theta|_{\widehat{P'}(1)}$ in \eqref{eq:2.4}. Hence  $\theta$ is nonzero implies that $\Theta|_{\widehat{\CP'}(1)}$ is also nonzero. Conversely assume that
	\(\Theta|_{\widehat{\CP'}(1)}\neq 0\). By homogeneity, the fiber map at $[u_{\lambda}]$, that is 
	\(\operatorname{Tr}|_{\BC u_{\lambda}}\) in~\eqref{local_P'} is nonzero. Lemma~\ref{lem_section_tr} then yields an element
	\(H\in \hfg\) preserving the line \(\BC u_{\lambda}\) and acting on
	\(T_{[u_{\lambda}]}Z\) by a nonzero scalar. As in Lemma~\ref{iml}, the corresponding one-parameter subgroup defines an Euler-type \(\BC^{*}\)-action at \([u_{\lambda}]\). Since \(Z\) is homogeneous, it is Euler-symmetric, and hence \(\theta\neq 0\), proving the second equivalence.
	
	Finally, since the map \(\theta\) is \(\fg\)-equivariant and its target \(U^{\vee}\) is irreducible, it is surjective whenever it is nonzero. In this case,
	the map \(\Theta|_{\widehat{\CP'}(1)}\) is also surjective. This proves~(1).

(3) follows from (2) by taking global sections of $\widehat{\eta \circ \partial'}$. For~(2), by Lemma~\ref{etaparital'}, the morphism
$
\widehat{\eta\circ\partial'}
$
is induced by the following $P$-equivariant map:
\[
\Hom(\BC u_{\lambda},(\widehat{\CP'_0})_{[u_{\lambda}]})
\longrightarrow
\Hom(\BC u_{\lambda}\otimes \widehat{T}_{[u_{\lambda}]}Z/\BC u_{\lambda},\BC u_{\lambda}),
\qquad
h \longmapsto \big((u_{\lambda},\overline{u'})\mapsto 2h(u_{\lambda},u')\big).
\]

For the first claim, it suffices to prove that this map is injective. Let
$
h \in \Hom(\BC u_{\lambda},(\widehat{\CP'_0})_{[u_{\lambda}]})
$
have zero image. Equivalently,
$
h(u_{\lambda})
$
annihilates
$
\widehat{T}_{[u_{\lambda}]}Z.
$
Write
\[
h(u_{\lambda})
=
\sum_{\nu}h(u_{\lambda})_{\nu},
\qquad
h(u_{\lambda})_{\nu}\in \fg_{[\nu,d_{\nu}]},
\]
as in the Lemma \ref{bundle_CP}.
Recall from \eqref{iden} that
\[
\widehat{T}_{[u_{\lambda}]}Z
=
\BC u_{\lambda}
\oplus
\bigoplus_{\nu}
\fn_{\fp_{[\nu]}}^{-}\cdot u_{\lambda},
\]
on which each component
$
h(u_{\lambda})_{\nu}
$ of $h(u_{\lambda})$
acts trivially on
$
\bigoplus_{\nu'\neq \nu}
\fn_{\fp_{[\nu']}}^{-}\cdot u_{\lambda},
$
while its action on
$
\BC u_{\lambda}
\oplus
\fn_{\fp_{[\nu]}}^{-}\cdot u_{\lambda}
$
coincides with the isotropy action on the affine tangent space of
$
Z_{[\nu]}
$
at
$
[u_{\lambda_{[\nu]}}].
$
Since
$
h(u_{\lambda})
$
annihilates
$
\widehat{T}_{[u_{\lambda}]}Z,
$
it follows that each
$
h(u_{\lambda})_{\nu}\in \fg_{[\nu,d_{\nu}]}
$
annihilates the affine tangent space of
$
Z_{[\nu]}
$
at
$
[u_{\lambda_{[\nu]}}].
$
Hence, by Lemma~\ref{faithful_isotropy},
$
h(u_{\lambda})_{\nu}=0
$
for all $\nu$. 
Therefore
$
h(u_{\lambda})=0.
$
The last claim follows from $h(u_{\lambda})_{\nu} \in \fg_{[\nu,d_{\nu}]}$ for each $\nu$.
\end{proof}
As a corollary, we obtain a complete description of $\partial^{-1}(\Xi_{Z}')$ by comparing $\eta \circ \partial'$ and \eqref{exac_seq2}.
\begin{cor}\label{cor_end}
	Let $Z \subset \BP U$ be a flag variety with $Z \not= \BP U$. Then the morphism
	$
	\eta \circ \partial'
	:
	\partial^{-1}(\Xi_{Z}')
	\longrightarrow
	H^{0}(Z,\Omega_{Z}(1))
	$
	induces a decomposition
	\begin{equation}
		\partial^{-1}(\Xi_{Z}')
		=
		\partial^{-1}(\Xi_{Z}^{0})
		\oplus
		\widetilde{U},
	\end{equation}
	where:
	\begin{itemize}
		\item[(i)]
		$
		\ker(\eta \circ \partial')
		=
		\partial^{-1}(\Xi_{Z}^{0})
		\cong
		\operatorname{Im}(\theta),
		$
		which is either zero or isomorphic to $U^{\vee}$, and the latter case happens exactly when $Z$ is a Hermitian symmetric space;
		
		\item[(ii)]
			There is 
			an isomorphism of \(\mathfrak g\)-modules
			\[
			\widetilde{U} \cong 
			\bigoplus_{\nu =1}^{r} 
			H^{0}\!\left(Z_{[\nu]},\, 
			E\big(\langle u_{\lambda_{[\nu]}} \rangle^{\vee} \otimes \mathfrak g_{[\nu,d_{\nu}]}\big)
			\right)
			\otimes \!\!\bigotimes_{\nu' \neq \nu} U_{\lambda_{[\nu']}}^{\vee},
			\]
			where each nonzero summand is an irreducible \(\fg\)-module.
	\end{itemize}
\end{cor}
\begin{proof}
The fact that	$
	\ker(\eta \circ \partial')
	=
	\partial^{-1}(\Xi_{Z}^{0})$ follows from Definition \ref{filt_01}(2).
	We consider the decomposition of
	$
	\partial^{-1}(\Xi_{Z}')
	$
	given by~\eqref{exac_seq2}:
	\[
	0 \longrightarrow \widetilde{U}
	\longrightarrow
	\partial^{-1}(\Xi_{Z}')
	\xrightarrow{~\theta~}
	H^{0}(Z,\CO_{Z}(1))
	\cong U^{\vee}.
	\]
	
	By Proposition~\ref{HX}(2),
	$
	\widetilde{U}
	$
	intersects trivially with
	$
	\ker(\eta\circ \partial').
	$
	Hence, if
	$
	\theta=0,
	$
	the decomposition follows immediately.
	
	Assume now that
	$
	\theta\neq 0.
	$
	By Proposition~\ref{HX}(1),
	$
	\theta
	$
	is surjective. Moreover, by Lemma~\ref{const_euler_inverse}, the restriction
	$
	\theta|_{\partial^{-1}(\Xi_{Z}^{0})}
	$
	is also nonzero. Since
	$
	U^{\vee}
	$
	is an irreducible
	$
	\fg
	$-module, the decomposition follows. (ii) follows directly from Proposition \ref{HX}.
\end{proof}
We now prove Theorems~\ref{reduction} and \ref{refor_reduction}.
\begin{proof}[Proof of Theorem~\ref{reduction} and Theorem~\ref{refor_reduction}]
We first deduce Theorem~\ref{reduction} from Theorem~\ref{refor_reduction}. Under condition~\emph{(iii)},
$
\Xi_Z\subset\operatorname{Im}(\partial),
$
the bottom map
$
\partial^0
$
in~\eqref{diag_Spencer} is surjective, so the conclusion follows from Theorem~\ref{refor_reduction}.

It remains to prove Theorem~\ref{refor_reduction}. If
$
Z=\BP U,
$
then
$
\wedge_Z^2=0,
$
and the claim is immediate. Assume therefore that
$
Z\neq\BP U.
$
By Proposition~\ref{HX}\emph{(3)},
$$
\widetilde U\cap\partial^{-1}(\Xi_Z^0)=0.
$$
Hence, by Corollary~\ref{cor_classify}, the map
$
\partial^0
$
is surjective if and only if either $Z$ is tangentially non-degenerate, or $Z$ is Euler-symmetric with
$
\hfg^{(1)}=0
$
and
$
\wedge_Z^2\simeq\BC.
$
By Proposition~\ref{HX}\emph{(1)}, the latter case is Hermitian symmetric.
Finally, the Hermitian symmetric spaces with
$
\hfg^{(1)}\neq0
$
were classified by Cartan and Kobayashi–Nagano; equivalently, they are the VMRTs of irreducible Hermitian symmetric spaces which are exactly the varieties listed in Theorem \ref{class_tang}\emph{(1)}
(see~\cite[Table~5]{ML99}). Hence they are tangentially
non-degenerate. This proves the theorem.
\end{proof}

\section{The Hwang--Li criterion for flag varieties}
In this section, we classify the flag varieties satisfying the conditions in Theorem~\ref{tors_prin}. In Section~5.1, we verify condition~(ii) by reducing it to a root-theoretic statement via Theorem~\ref{PRV}. In Section~5.2, using Theorems~\ref{reduction} and~\ref{class_tang}, we classify the flag varieties for which the restricted map
$\partial^0$
in \eqref{diag_Spencer} is surjective. Finally, in Section~5.3, we refine the Spencer diagram, prove the surjectivity of $\partial'$, and reduce the surjectivity of $\partial$ to the fiberwise equality
$
\mathrm{Im}(\widetilde{\phi_{\fp}})=\fl_z,
$
whenever the VMRT is defined.

Throughout this section, we keep the notation as in Section 2. Let $Z \subset \BP U$ be a flag variety, and let $G=Aut^{0}(Z)$ with $\fg=Lie(G)$. Then we can identify $\fa \fu \ft(\hat{Z})$ with $\hfg=\fg \oplus \BC id_{U} \subset End(U)$. We also fix a highest weight vector $u_{\lambda} \in U_{\lambda}$.
\subsection{The surjectivity of $\Psi$}
\begin{prop}\label{surj_Psi_}
	Let $Z=G/P \subset \BP U$ be a flag variety. Then it is linearly normal, and the map 
	\[
	\Psi:\hat{\fg}\otimes U^{\vee}
	\longrightarrow
	H^{0}(Z,T_{Z}(1))
	\]
	defined in Theorem \ref{tors_prin} is surjective.
\end{prop}
\begin{cor}
    Let $Z=G/P \subset \BP U$ be a flag variety.  Suppose $\CC \subset \BP TM$ is a $Z$-isotrivial cone structure and let $\CP$ be the associated $G$-structure. Then any conic connection on $\CC$ is locally induced by a principal connection on $\CP$.
\end{cor}
\begin{lem}\label{Psi-local}
	The map $\Psi$
	is induced, under the identification
	$T_{Z}\simeq E(\fg/\fp),$
	by taking global sections of the homogeneous bundle morphism
	\[
	\widehat{\Psi}:
	E(\langle u_{\lambda}\rangle^{\vee}\otimes \hfg)
	\longrightarrow
	E(\langle u_{\lambda}\rangle^{\vee}\otimes \fg/\fp),
	\]
	associated with the
	$
	P
	$
	-module morphism
	\[
	\widetilde{\Psi}:
	\langle u_{\lambda}\rangle^{\vee}\otimes \hat{\fg}
	\longrightarrow
	\langle u_{\lambda}\rangle^{\vee}\otimes \fg/\fp,
	\qquad
	u_{\lambda}^{\vee}\otimes (X+c\,\id_{U})
	\longmapsto
	u_{\lambda}^{\vee}\otimes \overline{X},
	\]
	where
	$
	X\in \fg,
	$
	$
	c\in \BC,
	$
	and
	$
	\overline{X}
	$
	denotes the image of
	$
	X
	$
	in
	$
	\fg/\fp.
	$
\end{lem}
\begin{proof}
	This follows directly from the definition of $\Psi$ and the decomposition~\eqref{peter_wey}.
\end{proof}
Note that $\widetilde{\Psi}\big|_{\langle u_{\lambda}\rangle^{\vee}\otimes \BC\,\id_{U}}
=0.$  Using Theorem~\ref{PRV}, we further reduce the problem to the following statement.

\begin{lem}
With the notation above,
	$
	\Psi
	$
	is surjective if and only if, for every dominant weight
	$
	\mu\in \Lambda^{+},
	$
	the canonical map
	\begin{equation}\label{surj_Psi}
	\Hom_{P}(U_{\mu}^{\vee},\langle u_{\lambda}\rangle^{\vee}\otimes \fg)
	\longrightarrow
	\Hom_{P}(U_{\mu}^{\vee},\langle u_{\lambda}\rangle^{\vee}\otimes \fg/\fp)
	\end{equation}
	is surjective. 
\end{lem}

\begin{proof}
	By Theorem \ref{PRV}(1), we have the canonical isomorphism~\eqref{peter_wey}. This isomorphism is functorial. More precisely, given a homogeneous bundle morphism
	$
	E(M)\to E(N)
	$
	induced by
	$
	\phi:M\to N,
	$
	the induced map on global sections is given by the canonical restriction maps on the direct summands:
	\[
	\Hom_{P}(U_{\mu}^{\vee},M)
	\longrightarrow
	\Hom_{P}(U_{\mu}^{\vee},N).
	\]
	The claim therefore follows from Lemma~5.3 and the above discussion.
\end{proof}
To verify the restriction map~\eqref{surj_Psi} is surjective, we apply the PRV formula
(Theorem~\ref{PRV}(2)) to the two
$P$-modules
\[
M:=\langle u_\lambda\rangle^\vee\otimes\fg,
\qquad
\overline{M}:=\langle u_\lambda\rangle^\vee\otimes\fg/\fp.
\]
Via evaluation at a lowest weight vector
$
u_{-\mu}\in U_{\mu}^{\vee},
$
the space
$
\Hom_{P}(U_{\mu}^{\vee},M)
$
is identified with the space of elements
$
u_{\lambda}^{\vee}\otimes X\in M_{-\mu}
$
satisfying the PRV conditions. Since
$
u_{\lambda}^{\vee}
$
has weight
$
-\lambda,
$
the element
$
X
$
has weight
$
\alpha=\lambda-\mu.
$
Hence Theorem~\ref{PRV}(2) yields the identification
\[
\Hom_{P}(U_{\mu}^{\vee},\langle u_{\lambda}\rangle^{\vee}\otimes \fg)
\simeq
\left\{
X\in \fg_{\alpha}
\ \middle|\
\begin{array}{l}
	\alpha=\lambda-\mu,\\[1mm]
	ad(e_{i})^{\langle \mu,\alpha_{i}^{\vee}\rangle+1}(X)=0
	\quad \text{for all } i,\\[1mm]
	ad(f_{k})(X)=0
	\quad \text{for all } \alpha_{k}\notin \Delta_{Z}
\end{array}
\right\},
\]
given by the evaluation map
$
f\mapsto f(u_{-\mu})(u_{\lambda}).
$
Similarly, using the representative of a class in
\(\mathfrak g/\mathfrak p\) lying in \(\mathfrak n_p^-\), we obtain
\[
\Hom_{P}(U_{\mu}^{\vee}, \langle u_{\lambda} \rangle^{\vee} \otimes \mathfrak g/\mathfrak p)
\simeq
\left\{
X \in \mathfrak g_{\alpha}\subset\mathfrak n_p^- \ \middle|\
\begin{array}{l}
	\alpha=\lambda-\mu,\\[2mm]
	ad(e_i)^{\langle\mu,\alpha_i^\vee\rangle+1}( X)\in\mathfrak p
	\quad \text{for all } i,\\[1mm]
	ad(f_k)(X)\in\mathfrak p
	\quad \text{for all } \alpha_k\notin\Delta_Z
\end{array}
\right\},
\]
also given by the evaluation map
$
f'\mapsto f'(u_{-\mu})(u_{\lambda}).
$

Under these identifications, since the restriction map in \eqref{surj_Psi}
is induced by the natural projection
$
\mathfrak g\longrightarrow \mathfrak g/\mathfrak p,
$
it remains to show that every element satisfying the quotient
conditions already satisfies the corresponding conditions in \(\mathfrak g\).
This is the content of the following lemma.
\begin{lem}
	Let  $\fg$ be simple, and let $U=U_{\lambda}$ be an irreducible module. Let
	 $X \in \fg_{\alpha} \subset \fn_{\fp}^{-}$ and assume that
	$
	\mu=\lambda-\alpha \in \Lambda^{+}.
	$
	Suppose that
	\[
	ad(e_{i})^{\langle \mu,\alpha_{i}^{\vee}\rangle+1}(X)\in \fp
	\quad (1\le i\le n),
	\qquad
	ad(f_{k})(X) \in \fp
	\quad \text{for all } \alpha_{k} \notin \Delta_{Z}.
	\]
	Then
	\[
	ad(e_{i})^{\langle \mu,\alpha_{i}^{\vee}\rangle+1}(X)=0
	\quad (1\le i\le n),
	\qquad
	ad(f_{k})(X) =0
	\quad \text{for all } \alpha_{k} \notin \Delta_{Z}.
	\]
\end{lem}
\begin{proof}
Let $X \in \fg_{\alpha} \subset \fn_{\fp}^{-}$ satisfying the above conditions. We first treat the simple roots outside $\Delta_{Z}$. If $\alpha_{k} \notin \Delta_{Z}$,  then by the definition of the parabolic grading, $ad(f_{k}).\fn_{\fp}^{-} \subset \fn_{\fp}^{-}$. Therefore if $ad(f_{k})(X) \in \fp$, then $ad(f_{k})(X) \in \fn_{\fp}^{-} \cap \fp =0.$ Similarly, if $\alpha_{i} \notin \Delta_{Z}$, we have 
$
ad(e_{i}).\fn_{\fp}^{-} \subset \fn_{\fp}^{-},
$
hence
$
ad(e_{i})^{\langle \mu,\alpha_{i}^{\vee}\rangle+1}(X)
\in \fn_{\fp}^{-}\cap \fp =0.
$

It remains to treat the case \(\alpha_i\in\Delta_Z\). 
Suppose, to the contrary, that for some \(\alpha_i\in\Delta_Z\),
\[
\operatorname{ad}(e_i)^{\langle\mu,\alpha_i^\vee\rangle+1}(X)
\]
is nonzero and lies in \(\mathfrak p\). Then it follows from $X \in \fg_{\alpha}$ that
\[
\alpha+\bigl(\langle\mu,\alpha_i^\vee\rangle+1\bigr)\alpha_i
\]
is a nonzero root. We will derive a contradiction from the \(\alpha_i\)-string through \(\alpha\). Write
$$
\alpha=-k\alpha_{i}-\beta,
$$
where $\beta$ is a nonnegative sum of simple roots satisfying
$
m_{\alpha_{i}}(\beta)=0,
$
and $k>0$. We distinguish two cases.

Assume first that $\beta=0$. Then necessarily $k=1$ and
$
\alpha=-\alpha_{i}.
$
Since $\mu=\lambda-\alpha$, we have
\[
\langle \mu,\alpha_{i}^{\vee}\rangle+1
=
\langle \lambda-\alpha,\alpha_{i}^{\vee}\rangle+1
=
3+\langle \lambda,\alpha_{i}^{\vee}\rangle
\ge 3.
\]
Hence
\[
0\neq
\operatorname{ad} (e_i)^{\langle\mu,\alpha_i^\vee\rangle+1}(X)
\in \fg_{m\alpha_{i}}
\]
for some $m\ge 2$, which is impossible.

Assume now that $\beta \neq 0$. Then
$
\alpha+(\langle \mu,\alpha_{i}^{\vee}\rangle+1)\alpha_{i}
$
is a root that can be written as
$$
\alpha+(\langle \mu,\alpha_{i}^{\vee}\rangle+1)\alpha_{i}
=
(\langle \mu,\alpha_{i}^{\vee}\rangle+1-k)\alpha_{i}-\beta,
$$
whose coefficient of $\alpha_{i}$ is
\[
m_{\alpha_{i}}
\big(
\alpha+(\langle \mu,\alpha_{i}^{\vee}\rangle+1)\alpha_{i}
\big)
=
\langle \mu,\alpha_{i}^{\vee}\rangle+1-k
\ge 0,
\]
since
$
(\operatorname{ad} e_i)^{\langle\mu,\alpha_i^\vee\rangle+1}(X) \in \fp
$.
But since $\beta$ is nonzero and positive, this implies that the coefficient of $\alpha_{i}$ has to be zero, that is
\begin{equation}\label{eq:4.1}
	k= \langle \mu,\alpha_{i}^{\vee}\rangle+1 \ge 1
\end{equation}
More precisely we know that 
\begin{equation}\label{eq:4.105}
-\beta \in \Phi_{Z}^{0}, \quad  -\beta+ \alpha_{i} \notin \Phi,
\end{equation}
 since $\beta$ is added from $\alpha \in \Phi_{Z}^{-}$ with multiple of the simple root $\alpha_{i} \in \Delta_{Z}$.
We may therefore write the $\alpha_i$-string through $\alpha$ as
\[
\alpha-\gamma\alpha_{i},
\dots,
\alpha,
\alpha+\alpha_{i},
\dots,
\alpha+k\alpha_{i}=-\beta,
\]
satisfying the relation
$
\gamma-k=\langle \alpha,\alpha_{i}^{\vee}\rangle
$
by the $s_{\alpha_{i}}$-action on the string. Combining this with~\eqref{eq:4.1}, we obtain
\begin{equation}\label{eq:4.2}
	\gamma
	=
	k+\langle \alpha,\alpha_{i}^{\vee}\rangle
	=
	\langle \mu+\alpha,\alpha_{i}^{\vee}\rangle+1
	=
	\langle \lambda,\alpha_{i}^{\vee}\rangle+1 \ge 2,
\end{equation}
where the last inequality follows from $\alpha_{i}\in \Delta_{Z}$.

On the other hand, since
$
\alpha=-k\alpha_{i}-\beta,
$
we have
\begin{equation}\label{eq:4.3}
	-\langle \beta,\alpha_{i}^{\vee}\rangle
	=
	\langle \alpha,\alpha_{i}^{\vee}\rangle+2k
	=
	k+\gamma \ge 3,
\end{equation}
from \eqref{eq:4.1} and \eqref{eq:4.2}.
We now derive a contradiction.

First, we claim that $\alpha_{i}$ lies at an end of $\CD_{\alpha}$, which is the subdiagram generated by the support of $\alpha$. Indeed, as $\alpha$ and $-\beta$ are roots, $\CD_{\alpha}$ and $\CD_{-\beta}$ are both connected subgraphs of the Dynkin diagram. Moreover since
\[
-\beta=\alpha +k\alpha_{i},
\qquad \text{with} \,\,
m_{\alpha_i}(\beta)=0,
\]
the diagram $\CD_{-\beta}$ is obtained from $\CD_{\alpha}$ by deleting the node $\alpha_i$. This is only possible if $\alpha_i$ is an end node of $\CD_{\alpha}$.

Next let $\Phi'$ be the root subsystem generated by $\operatorname{Supp}(\alpha)$. Then $\Phi'$ is the root system of a simple Lie subalgebra $\fg'\subset \fg$, whose Dynkin diagram is precisely $\CD_{\alpha}$, and where $\alpha_i$ corresponds to an end node. Note that
$
\operatorname{Supp}(\alpha-\gamma\alpha_i)
=
\operatorname{Supp}(\alpha),
$
since $m_{\alpha_{i}}(\alpha)=-k <0$.
Therefore 
$
\alpha-\gamma\alpha_i \in \Phi'
$
is a negative root whose coefficient of $\alpha_i$ equals
\begin{equation}\label{eq:4.4}
	m_{\alpha_{i}}(\alpha-\gamma\alpha_{i})=-k-\gamma \le -3
\end{equation}
by~\eqref{eq:4.3}. 

Summarizing above we find a positive root $-\alpha+\gamma \alpha_{i}$ in the root system $\Phi'$, and a simple root $\alpha_{i}$ on the end of the Dynkin diagram, such that the coefficient of $\alpha_{i}$ in $-\alpha+\gamma \alpha_{i}$ is at least three. In particular for the highest positive root $\beta_{l}'$ of $\Phi'$, we have $m_{\alpha_{i}}(\beta_{l}') \ge 3$. From the table~\cite[p.~66, Table~2]{Hum72},  this can occur only when
\[
(\fg',\alpha_i)= (E_8,\alpha_2')
\qquad\text{or}\qquad
(\fg',\alpha_i)= (G_2,\alpha_1'),
\]
where we follow the numbering convention of~\cite{Hum72}.

In both cases, we have necessarily $\fg=\fg'$, and moreover the coefficient of $\alpha_i$ in the highest root of $\Phi'$ equals $3$. It then follows from \eqref{eq:4.4} that $k+\gamma=3$. Combining this with~\eqref{eq:4.1}, \eqref{eq:4.2}, and~\eqref{eq:4.3}, we obtain
$$
k=1,\qquad \gamma=2,\qquad
\langle \beta,\alpha_{i}^{\vee}\rangle=-3.
$$

If $(\fg',\alpha_{i})=(G_2,\alpha_1'),$ then $\beta$ is a positive root in $\Phi'$ with $\CD_{\beta}=\CD_{\alpha}\backslash \{\alpha_{i}\}=A_1,$ that is $\beta=\alpha_{2}'$. By \eqref{eq:4.105}, this just means that $\lambda$ is proportional to $\lambda_{1}$. But then  $Z=G_{2}/P_{1} \cong \BQ^5$, and by our assumption $\fg$ is taken as $Lie(Aut(Z)) \cong \fs \fo_{7}$, which is strictly larger than $\fg_2$, a contradiction.

If $(\fg',\alpha_i)= (E_8,\alpha_2'),$ then  by definition
$
\CD_{\beta}=\CD_{\alpha}\backslash \{\alpha_{i}\}=A_7.
$
Hence
$
m_{\alpha_{j}}(\beta)=1
$
for every
$
\alpha_{j} \in \operatorname{Supp}(\beta).
$
This contradicts
$
\langle \beta,\alpha_{i}^{\vee} \rangle=-3,
$
since $\alpha_{i}$ has only one adjacent node in $\CD_{\alpha}$. 
\end{proof}
\begin{proof}[Proof of Proposition \ref{surj_Psi_} and Corollary 5.2]
    The linear normality is automatic. For the surjectivity of $\Psi$, by Lemma~5.4, it suffices to prove that the canonical restriction map~\eqref{surj_Psi} is surjective for every
	$
	\mu\in \Lambda^{+}.
	$
	When
	$
	\fg
	$
	is simple, this follows from Lemma~5.5 as explained above; assume that
	$
	\fg
	$
	is reducible. Decomposing into products as in Section~4.3, the map
	$
	\Psi
	$
	becomes
	\[
	\BC\id_{U}\otimes U^{\vee}
	\oplus
	\bigoplus_{\nu}
	\left(
	\fg_{[\nu]}\otimes U_{\nu}^{\vee}
	\otimes
	\bigotimes_{\nu'\neq \nu}U_{\nu'}^{\vee}
	\right)
	\longrightarrow
	\bigoplus_{\nu}
	H^{0}(Z_{\nu},T_{Z_{\nu}}(1))
	\otimes
	\bigotimes_{\nu'\neq \nu}U_{\nu'}^{\vee},
	\]
	where the summand
	$
	\BC\id_{U}\otimes U^{\vee}
	$
	is mapped to zero, and the remaining summands are mapped by
	$
	\Psi_{\nu}
	$
	on each simple factor, tensored with the identity on the other factors. Hence
	$
	\Psi
	$
	is surjective if and only if each
	$
	\Psi_{\nu}
	$
	is surjective. Finally Corollary 5.2 follows from Proposition 5.1 and Theorem \ref{tors_prin}.
\end{proof}
\subsection{On the surjectivity of $\partial^{0}$}
In this subsection, we complete the classification of flag varieties for which the restricted map
$
\partial^{0}
$
in~\eqref{diag_Spencer} is surjective.

\begin{prop}\label{partial0_surj_classify}
	Let $Z \subset \BP U$ be a flag variety. Then the restricted map
	$
	\partial^{0}
	$
	in~\eqref{diag_Spencer} is surjective if and only if $Z$ is isomorphic to one of the following varieties:
	\begin{itemize}
		\item[(1)] the tangentially non-degenerate varieties classified in Theorem~\ref{class_tang};
		
		\item[(2)] Hermitian symmetric spaces satisfying
		$
		\wedge_{Z}^{2} \cong \BC
		$, which are classified by subadjoint varieties, namely:
		\begin{itemize}
			\item
			$
			\BP^1 \times \BQ^{m}
			\subset
			\BP(\BC^2 \otimes \BC^{m+2})
			$
			($m \ge 1$), and
			$
			v_{3}(\BP^1) \subset \BP(S^3 \BC^2);
			$
			
			\item
			$
			E_{7}/P_{7},
			$
			$
			\BS_{6},
			$
			$
			Gr(3,6),
			$
			and
			$
			Lag(3,6)
			$
			under their minimal embeddings.
		\end{itemize}
	\end{itemize}
\end{prop}
%Recall from Theorem~\ref{refor_reduction} that $\partial^{0}$ is surjective if and only if either $Z$ is tangentially non-degenerate, or $Z$ is a Hermitian symmetric space satisfying $\wedge_{Z}^{2}\cong \BC.$

\begin{proof}[Proof of Proposition \ref{partial0_surj_classify}]
We divide the classification in (2) into two cases.
	
\medskip
\noindent\textbf{Case 1.}  Assume that
$
Z=Z' \times Z'' \subset \BP(U_{\lambda'}\otimes U_{\lambda''})
$
can be written as a nontrivial product of flag varieties. We claim that
$
Z \cong \BP^1 \times \BQ^m \subset \BP(\BC^2 \otimes \BC^{m+2}).
$
Indeed, by~\eqref{prod_wedge_Z}, up to interchanging $Z'$ and $Z''$, the condition
$
\wedge_{Z}^2\cong \BC
$
is equivalent to
\[
\wedge_{Z'}^{2}=I_{2}(Z')=0,
\qquad
\wedge^{2}U_{\lambda'}^{\vee} \cong \BC,
\qquad
I_{2}(Z'')\cong \BC,
\qquad
\wedge_{Z''}^{2}=0.
\]

Since flag varieties are cut out by quadrics, as proved by Kostant (see \cite{Gar82}), the first two conditions imply that
$
Z' \cong \BP^1.
$
%Indeed, the condition$I_{2}(Z')=0$forces $Z'$ to be a projective space, while$\wedge^{2}U_{\lambda'}^{\vee}\cong \BC$ implies$\dim U_{\lambda'}=2.$
Similarly, the condition
$
I_{2}(Z'')\cong \BC
$
implies that $Z''$ is a hyperquadric. Hence
$
Z \cong \BP^1 \times \BQ^m
$
as claimed.

\medskip
\noindent\textbf{Case 2.} Assume that $G=Aut^{0}(Z)$ is a simple group, that is $Z=G/P$ is isomorphic to an IHSS. Then from Lemma \ref{irr_Gaussian} the image of the Gaussian map $H^{0}(Z,\Omega_{Z}(2))$ is irreducible. Hence from the exact sequence of the Gaussian map \eqref{Gaussian_sequence}, we obtain the decomposition
$$
\wedge^{2}U^{\vee} \cong H^{0}(Z,\Omega_{Z}(2)) \oplus \BC,
$$
which splits as a trivial plus a nontrivial irreducible representation. This condition was classified in \cite[Theorem 1.b]{KR09} for arbitrary flag varieties. 
\begin{thm}[{\cite[Theorem~1(b)]{KR09}}]
	Let $Z \subset \BP U_{\lambda}$ be a flag variety such that
	$
	G=\Aut^{0}(Z)
	$
	is simple. Then
	$
	\wedge^{2}U_{\lambda}^{\vee}
	$
	decomposes as the direct sum of a trivial representation and a nontrivial irreducible representation if and only if $Z$ is isomorphic to one of the following varieties:
	\[
	v_{3}(\BP^1) \subset \BP(S^3 \BC^2), 
	\qquad
	E_{7}/P_{7},
	\qquad
	\BS_{6},
	\qquad
	Gr(3,6),
	\qquad
	Lag(3,6),
	\]
	under their minimal embeddings.
\end{thm}
It remains to show that, for each of the varieties listed above, there exists a nonzero form
$
w\in\wedge^2U^\vee
$
such that \(\psi(w)=0\). Indeed, these varieties are known to be Legendrian (cf. \cite[Proposition 4.5]{LH24}), i.e.,  there exists a \(\fg\)-invariant non-degenerate skew-symmetric form
$
w\in\wedge^2U^\vee
$
such that, for every \(z\in Z\), the affine tangent space
\(\widehat T_zZ\subset U\) is a maximal isotropic subspace with respect to \(w\). In particular,
$
w(\hat z,\widehat T_zZ)=0
$
for all $z\in Z$,
and hence \(\psi(w)=0\) by definition. This completes the proof.
\end{proof}
\subsection{The surjectivity of $\partial'$ and $\partial$}
In this subsection, we complete the proof of Theorem~\ref{main_thm}. To make the picture clearer, we first refine the diagram~\eqref{diag_Spencer} via the Gaussian map as follows. 

Tensoring the Gaussian sequence \eqref{Gaussian_sequence} with $U$ we obtain an exact sequence:
$$	
0 \rightarrow \Xi_{Z}^{0} \rightarrow \wedge^{2}U^{\vee}\otimes U \xrightarrow{\psi \otimes U} H^{0}(Z,\Omega_{Z}(2)\otimes U) \rightarrow 0.
$$
Recall the inclusions $\CO_{Z}(-1) \subset \widehat{T} \subset U \otimes \CO_{Z}$ and the sequence \eqref{eulerseq}, tensoring the latter with $\Omega_{Z}(2)$ induces the natural exact sequence:
\begin{equation}\label{iq}
0 \rightarrow \Omega_{Z}(1)  \xrightarrow{\widehat{i_T}}     \Omega_{Z}(2) \otimes \hat{T}  \xrightarrow{\widehat{q_T}} T_Z\otimes\Omega_Z(1) \rightarrow 0
\end{equation}

The following result follows immediately from Definition \ref{filt_01}.
\begin{lem}
	Let $Z \subset \BP U$ be a flag variety. 
	 Regarding $ \Omega_{Z}(1)$ and $\Omega_{Z}(2) \otimes \hat{T}$ as subbundles of $\Omega_{Z}(2) \otimes U$, we have:
	$$
	\Xi_{Z}=(\psi \otimes U)^{-1}(	H^{0}(Z,\Omega_{Z}(2)\otimes \widehat{T})), \quad  \Xi_{Z}'=(\psi \otimes U)^{-1}(H^{0}(Z,\Omega_{Z}(1))).
	$$
	In particular the map $\eta: \Xi_{Z}' \rightarrow H^{0}(Z,\Omega_Z(1))$ in \eqref{diag_Spencer} coincides with $(\psi \otimes U)|_{\Xi_{Z}'}$, and therefore it is surjective.
\end{lem}
Summing the above up, we actually obtain the following diagram refining \eqref{diag_Spencer}:
\begin{equation}\label{refine_diag_Spencer}
	\begin{tikzcd}
		U^{\vee}\otimes \hfg \arrow[r, "\partial"]	&\wedge^{2}U^{\vee}\otimes U
		\arrow[r, two heads, "\psi \otimes U"]
		&
		H^{0}(Z,\Omega_{Z}(2)\otimes U)
		\\
		\partial^{-1}(\Xi_{Z}) \arrow[r, "\partial"]\arrow[u, hook]	& \Xi_{Z}
		\arrow[u, hook]
		\arrow[r, two heads, "\psi \otimes U|_{\Xi_{Z}}"]
		&
		H^{0}(Z,\Omega_{Z}(2)\otimes \hat{T})
		\arrow[r, "q_{T}"] \arrow[u,hook]
		&
		H^{0}(Z,T_Z\otimes\Omega_Z(1))
		\\
		\partial^{-1}(\Xi_{Z}') \arrow[r, "\partial'"]	\arrow[u, hook]& \Xi_{Z}'
		\arrow[u, hook]
		\arrow[r, two heads, "\eta"]
		&
		H^{0}(Z,\Omega_Z(1))
		\arrow[u, hook, "i_{T}"] \\
		\partial^{-1}(\Xi_{Z}^{0}) \arrow[r, "\partial^{0}"]  \arrow[u, hook]&
		\Xi_{Z}^{0} \arrow[u, hook] 
	\end{tikzcd}
\end{equation}
where $q_{T}$ and $ i_{T}$ are obtained by taking global sections of the sequence \eqref{iq}, and $q_{T} \circ(\psi \otimes U)|_{\Xi_{Z}}$ coincides with $\zeta$ in \eqref{diag_Spencer}.

\subsubsection{On the surjectivity of $\partial'$}
\begin{prop}\label{partial0_imply_partial'}
	Let
	$
	Z \subset \BP U
	$
	be a flag variety with $Z \not= \BP U$. If the restricted map
	$
	\partial^{0}
	$
	is surjective, then the composition $\eta \circ \partial': \partial^{-1}(\Xi_{Z}') \rightarrow H^{0}(Z,\Omega_Z(1))$ is surjective. Consequently 
	$
	\partial' 
	$
	is surjective.
\end{prop} 
\begin{proof}
	
By Corollary~\ref{cor_end}, under the condition $Z \not= \BP U$, to prove the surjectivity of
$
\eta \circ \partial',
$
it suffices to show that the restriction
$
(\eta \circ \partial')|_{\widetilde{U}}
$
is surjective. This map is obtained by taking global sections of the inclusion of homogeneous bundles
\[
\widehat{\eta \circ \partial'}
:
\widehat{\CP'_0}(1)
\longrightarrow
\Omega_{Z}(1)
\]
in~\eqref{eq:2.5}, which by Proposition~\ref{HX}(2) is induced by the inclusion of $P$-modules
\[
\langle u_{\lambda} \rangle^{\vee}
\otimes
\bigoplus_{\nu}\fg_{[\nu,d_{\nu}]}
\hookrightarrow
\langle u_{\lambda} \rangle^{\vee}
\otimes
\fn_{\fp}
\cong
\langle u_{\lambda} \rangle^{\vee}
\otimes
(\fg/\fp)^{\vee}.
\]

Assume first that $Z$ is a Hermitian symmetric space. Then for each $\nu$, $\fg_{[\nu]}/\fp_{[\nu]}$ is $\fp_{\nu}$-irreducible from Definition \ref{cominuscule}, that is the above inclusion is an isomorphism. Hence
$
\widehat{\eta \circ \partial'}
$
is an isomorphism, and the claim follows.

Assume now that $Z$ is not a Hermitian symmetric space. By the classification in Proposition~\ref{partial0_surj_classify} and Theorem~\ref{class_tang}, $Z$ is either an adjoint variety or a quasi-minuscule variety. The corresponding highest weights are listed in Table~\ref{Table1}. In particular, $\fg$ is simple and the parabolic grading has length
$
d=2.
$
Thus the above inclusion becomes
\[
\langle u_{\lambda} \rangle^{\vee}
\otimes
\fg_{2}
\hookrightarrow
\langle u_{\lambda} \rangle^{\vee}
\otimes
(\fg_{1}\oplus \fg_{2})
\cong
\langle u_{\lambda} \rangle^{\vee}
\otimes
(\fg/\fp)^{\vee}.
\]

Recall that
$
\fg/\fp
\cong
\fg_{-1}\oplus \fg_{-2}
$
as $\fg_{0}$-modules. The above inclusion identifies
\[
\fg_{2}
\cong
\{
\Lambda\in (\fg/\fp)^{\vee}
\mid
\Lambda(\fg_{-1})=0
\}
=
\fg_{-1}^{\perp}.
\]

The claim now follows from the following two lemmas,
adapted from~\cite{FH18a,LH24}.
\begin{lem}\label{esti_parital'}
	Let
	$
	Z \subset \BP U
	$
	be a smooth projective variety covered by lines. Choose an irreducible covering family
	$
	\CK
	$
	of lines on $Z$, and denote by
	$
	\CC\subset \BP TZ
	$
	the associated VMRT. Then for a general point
	$
	z=[u]\in Z,
	$
	and any
	$
	\sigma\in H^{0}(Z,\Omega_{Z}(1)),
	$
	the linear form
	\[
	\sigma_{z}
	\in
	(\BC u\otimes T_{z}Z)^{\vee}
	\]
	vanishes on the linear span
	$
	\langle \widehat{\CC}_{z}Z\rangle
	$
	of the affine cone of the VMRT at $z$.
\end{lem}
\begin{proof}
	Let
	$
	z
	$
	be a general point of $Z$, and let
	$
	\sigma \in H^{0}(Z,\Omega_{Z}(1)).
	$
	Choose a line
	$
	l\in \CK
	$
	through $z$. 
	Restricting $\sigma$ to $l$, we obtain from \eqref{unbendable} that 
	\[
	\sigma|_{l}
	\in
	\Hom\bigl(
	\CO_{l}(2)
	\oplus
	\CO_{l}(1)^{\oplus r}
	\oplus
	\CO_{l}^{\,n-1-r},
	\CO_{l}(1)
	\bigr).
	\]
	Since
	$
	\Hom(\CO_{l}(2),\CO_{l}(1))=0,
	$
	the section
	$
	\sigma|_{l}
	$
	vanishes on the summand
	$
	\CO_{l}(2)\subset T_{Z}|_{l}.
	$
	Hence
	$
	\sigma_{z}
	$
	vanishes on
	$
	T_{z}l=\CO_{l}(2)_{z}.
	$
	When the line $l$ varies, the tangent directions
	$
	T_{z}l
	$
	span
	$
	\langle \widehat{\CC}_{z}Z\rangle.
	$
	Therefore
	$
	\sigma_{z}
	$
	vanishes on
	$
	\langle \widehat{\CC}_{z}Z\rangle.
	$
\end{proof}
\begin{lem}\label{inf_vmrt_adco}
	Let
	$
	Z \subset \BP U
	$
	be an adjoint or a quasi-minuscule variety which is not a Hermitian symmetric space. Then $Z$ is covered by lines. Moreover:
	\begin{itemize}
		\item[(a)]
		If
		$
		Z \cong Gr_{\omega}(2,2n)
		$
		with $n \ge 3$
		or
		$
		F_{4}/P_{4},
		$
		then $Z$ admits a unique family of lines whose VMRT at a general point is non-degenerate.
		
		\item[(b)]
		Assume that
		$
		Z
		$
		is not isomorphic to
		$
		Gr_{\omega}(2,2n),
		$
		$
		F_{4}/P_{4},
		$
		or
		$
		Fl(1,n;n+1).
		$
		Then $Z$ admits a unique family of lines whose VMRT spans a distribution
		$
		\CF\subset T_{Z}
		$
		given by
		$
		\fg_{-1}\subset \fg/\fp,
		$
		namely
		\[
		\CF
		\cong
		E(\fg_{-1})
		\subset
		T_{Z}
		\cong
		E(\fg/\fp).
		\]
		
		If
		$
		Z \cong Fl(1,n;n+1),
		$
		then $Z$ admits two families of lines
		$
		\CK
		$
		and
		$
		\CK',
		$
		whose VMRTs together span a distribution with the same property.
	\end{itemize}
\end{lem}
\begin{proof}
	From Proposition 2.5, we can identify a general line through a point with its tangent direction. For a rational homogeneous space
	$
	Z=G/P
	$
	with \(G\) simple, the spaces of lines on \(Z\) in its minimal embedding, as well as the spaces of lines through a fixed point, were described explicitly in~\cite{LM03}. In particular, statement~\emph{(a)} follows from~\cite[Theorem~4.8(2) and Corollary~5.5]{LM03} for
	$
	Gr_{\omega}(2,2n),
	$
	and from~\cite[Proposition~6.5]{LM03} for
	$
	F_{4}/P_{4}.
	$
	The cases in~\emph{(b)} follow from~\cite[Theorem~4.8(1)]{LM03}.
\end{proof}
Now suppose that $Z$ is one of the varieties appearing in Lemma~5.11 (a). Then the above lemmas imply that
$
H^{0}(Z,\Omega_{Z}(1))=0,
$
hence
$
\eta \circ \partial'
$
is surjective; suppose that $Z$ is one of the varieties in Lemma~5.11(b). Then
\[
H^{0}(Z,\Omega_{Z}(1))
=
\{
\sigma\in H^{0}(Z,\Omega_{Z}(1))
\mid
\sigma|_{\CF}=0
\}
=
H^{0}\bigl(
Z,
E(\langle u_{\lambda}\rangle^{\vee}\otimes \fg_{-1}^{\perp})
\bigr).
\]
By the above discussion, this space coincides with the image of $\eta \circ \partial'$ and the claim follows.
\end{proof}
\subsubsection{On the surjectivity of $\partial$}
In this subsection, we complete the proof of Theorem~\ref{main_thm} by analyzing the surjectivity of
$
\partial
$
under the assumption that
$
\partial^{0}
$
and
$
\partial'
$
are surjective. The main result of this subsection is the following proposition.

\begin{prop}\label{main_432}
	Let
	$
	Z \subset \BP U
	$
	be a flag variety with $Z \not= \BP U$, and assume that the maps
	$
	\partial^{0}
	$
	and
	$
	\partial'
	$
	are surjective. Then the map
	\[
	\partial:
	\partial^{-1}(\Xi_{Z})
	\longrightarrow
	\Xi_{Z}
	\]
	is surjective if and only if the composition map
	$
	\zeta \circ \partial
	$
	in~\eqref{diag_Spencer} is surjective. Moreover, this happens if and only if
	$
	Z
	$
	is isomorphic to one of the varieties listed in Theorem~\ref{main_thm}.
\end{prop}
The main task in this subsection is to study the surjectivity of
$
\zeta \circ \partial,
$
for which we study the sequence~\eqref{eq:2.3}:
\[
0
\longrightarrow
\widehat{\CP'}(1)
\longrightarrow
\widehat{\CP}(1)
\xrightarrow{\widehat{\zeta \circ \partial}}
T_Z\otimes\Omega_Z(1).
\]
By Lemma~\ref{CPadjoint}, the map
$
\zeta \circ \partial
$
is obtained by taking global sections of
$
\widehat{\zeta \circ \partial}.
$
Based on this, we shall derive a sufficient condition in Corollary \ref{suff_condition_surj} for the surjectivity of $\zeta \circ \partial$. 
\begin{defn}
	Let $Z=G/P \subset \BP U$ be a flag variety.
	\begin{itemize}
		\item[(1)] Denote the extended isotropy representation at
	$z=[u_{\lambda}]$ by
	\begin{equation}\label{ex_phip}
		\widetilde{\phi_{\fp}}:
		\fp \oplus \BC
		\xrightarrow{\phi_{\fp}\oplus \mathrm{id}}
		\End(\fg/\fp),
	\end{equation}
	where $\phi_{\fp}$ is the isotropy action \eqref{phip_defn} and $\BC$ is mapped into scalar multiplications. 
	
 \item[(2)]	Assume that $Z \subset \BP U$ is covered by lines. Let
	$
	\{\CK^{i}\}
	$
	be all irreducible families of lines on $Z$, and let
	$
	\{\CC^{i}\subset \BP TZ\}
	$
	be the associated VMRT-structures. For any point
	$
	z\in Z,
	$
	define the total isotropy subalgebra as
	\[
	\fl_{z}
	=
	\bigcap_{i}
	\fa \fu \ft(\widehat{\CC}^{i}_{z}Z) \subset End(T_{z}Z).
	\]
	Since
	$
	G
	$
	acts on each family
	$
	\CK^{i},
	$
	the isotropy representation at
	$
	z=[u_{\lambda}]
	$
	factors as
	\[
	\widetilde{\phi_{\fp}}:
	\fp \oplus \BC 
	\longrightarrow
	\fl_{z}
	\subset
	\End(T_{z}Z).
	\]
	\end{itemize}
\end{defn}

\begin{cor}\label{suff_condition_surj}
	Let
	$
	Z \subset \BP U
	$
	be a flag variety with $Z \not= \BP U$, and keep the above notation. Assume that
	$
	H^{1}(Z,\widehat{\CP'}(1))=0,
	$
	and that one of the following conditions holds for $\widetilde{\phi_{\fp}}$ defined in \eqref{ex_phip}:
	\begin{itemize}
		\item[(1)]
		$
		\widetilde{\phi_{\fp}}
		$ 
		is surjective;
		
		\item[(2)]
		$
		Z
		$
		is covered by lines, and
		$
		\widetilde{\phi_{\fp}}
		$
		maps
		$
		\fp\oplus \BC
		$
		surjectively onto
		$
		\fl_{z}.
		$
	\end{itemize}
	
	Then the composition map
	$
	\zeta \circ \partial
	$
	in~\eqref{diag_Spencer} is surjective. 
\end{cor}
For the proof, we first identify the image of
$
\widehat{\zeta \circ \partial}.
$
Let $z=[u_{\lambda}] \in Z$. As in~\eqref{iden}, write \(T_zZ=\fg/\fp\). Moreover,
\[
T_Z \otimes \Omega_Z(1)=\End(T_Z)(1)
=
E(\End(\fg/\fp)\otimes \langle u_{\lambda}\rangle^{\vee}), \qquad
\widehat{\CP}(1)
=
E((\fp\oplus \BC)\otimes \langle u_{\lambda}\rangle^{\vee})
\]
as in Lemma~\ref{bundle_CP}. 
\begin{lem}
	Keep the notation above. Then the image of
	$
	\widehat{\zeta \circ \partial}
	$
	coincides with the homogeneous subbundle
	\[
	E(\operatorname{Im}(\widetilde{\phi_{\fp}})\otimes \langle u_{\lambda}\rangle^{\vee}).
	\]
\end{lem}

\begin{proof}
	By the last statement of Lemma~\ref{CPadjoint},
	$
	\widehat{\zeta \circ \partial}
	$
	is induced by the following map of $P$-modules, twisted by
	$
	\langle u_{\lambda}\rangle^{\vee}:
	$
	\[
	\fp\oplus \BC
	\longrightarrow
	\End(\fg/\fp),
	\]
	\[
	(X,c)\cdot Y
	=
	2\phi_{\fp}(X)(Y)
	+
	\bigl(
	Tr(X|_{\langle u_{\lambda}\rangle})+c
	\bigr)Y,
	\qquad
	X\in \fp,\ c\in \BC, Y \in \fg/\fp.
	\]
	 The claim follows immediately.
\end{proof}

Next we  recall the following lemma, which was proved in \cite{FH18a} by an argument similar to that used in Lemma \ref{esti_parital'} combined with \eqref{unbendable_}.

\begin{lem}[{\cite[Proof of Lemma~7.2]{FH18a}}]\label{esti_partial}
Let
$
Z \subset \BP U
$
be a smooth projective variety covered by lines. Choose an irreducible covering family
$
\CK
$
of lines on $Z$, and denote by
$
\CC\subset \BP TZ
$
the associated VMRT. Then for a general point
$
z=[u]\in Z,
$
and any
$
A\in H^{0}(Z,\End(T_{Z})(1)),
$
the endomorphism
\[
A_{z}
\in
\End(T_{z}Z)\otimes \langle u\rangle^{\vee}
\]
lies in
$
\fa \fu \ft(\widehat{\CC}_{z}Z)
\otimes
\langle u\rangle^{\vee}.
$
\end{lem}

\begin{proof}[Proof of Corollary \ref{suff_condition_surj}]
Taking global sections of the sequence~\eqref{eq:2.3}, together with the description of the image of
$
\widehat{\zeta\circ\partial}
$
from Lemma~5.15, yields the exact sequence
\[
0
\longrightarrow
H^{0}(Z,\widehat{\CP'}(1))
\longrightarrow
H^{0}(Z,\widehat{\CP}(1))
\xrightarrow{\zeta \circ \partial}
H^{0}\bigl(
Z,
E(\operatorname{Im}(\widetilde{\phi_{\fp}})\otimes \langle u_{\lambda}\rangle^{\vee})
\bigr)
\longrightarrow
0,
\]
where the right exactness follows from the vanishing
$
H^{1}(Z,\widehat{\CP'}(1))=0.
$
Hence, if
$
\widetilde{\phi_{\fp}}
$
is surjective, then
$
\zeta \circ \partial
$
is surjective. Similarly in (2), let
$
A\in H^{0}(Z,\End(T_Z)(1)).
$
By the preceding lemma, for every point
$
z'=[u]\in Z,
$
one has
\[
A_{z'}
\in
\fl_{z'}\otimes \langle u\rangle^{\vee}.
\]
Equivalently,
$
A
\in
H^{0}\bigl(
Z,
E(\fl_{z}\otimes \langle u_{\lambda}\rangle^{\vee})
\bigr).
$
Therefore
\[
H^{0}(Z,\End(T_Z)(1))=H^{0}(Z,E(Im(\widetilde{\phi_{\fp}}) \otimes \langle u_{\lambda} \rangle^{\vee}))=H^{0}\bigl(
Z,
E(\fl_{z}\otimes \langle u_{\lambda}\rangle^{\vee})
\bigr),
\]
and the surjectivity of
$
\zeta \circ \partial
$
follows.
\end{proof}

To apply,  we then examine the conditions for the varieties appearing in Proposition~\ref{partial0_surj_classify}.

\begin{prop}\label{prop_last_verify}
	Let $Z=G/P \subset \BP U$ be a flag variety with $Z \not= \BP U$, and assume that the maps $\partial^{0}$ and $\partial'$ are surjective. Then $H^{1}(Z,\widehat{P'}(1))=0$, and the following hold:
	\begin{itemize}
		\item[(1)]
		If $Z$ is not covered by lines, namely
		\[
		Z\cong v_{2}(\BP^{n}) \subset \BP(S^{2}\BC^{n+1})
		\quad \text{or} \quad
		v_{3}(\BP^{1}) \subset \BP(S^{3}\BC^{2}),
		\]
		then $\widetilde{\phi_{\fp}}$ is surjective.
		
		\item[(2)]
		If $Z$ is covered by lines, then $\widetilde{\phi_{\fp}}$ 
		maps onto $\fl_{z}$ if and only if $Z$ is isomorphic to one of the varieties listed in Theorem~\ref{main_thm}.
	\end{itemize}
\end{prop}
 The first condition follows from the following more precise calculations:
\begin{lem}\label{van_ome}
		Let
	$
	Z \subset \BP U
	$
	be a flag variety with $Z \not= \BP U$, and assume that the maps
	$
	\partial^{0}
	$
	and
	$
	\partial'
	$
	are surjective.  Then $H^{i}(Z,\Omega_{Z}(1))=H^{i}(Z,\widehat{\CP^{'}}(1))=H^{i}(Z,\widehat{\CP'_0}(1))=0$, for any $i \geqslant 1$.
\end{lem}
\begin{proof}
Recall from Proposition \ref{HX}(1) that we have the following left exact sequence \eqref{eq:2.4}:
\begin{equation*}
	0 \longrightarrow \widehat{\CP'_0}(1)
	\longrightarrow \widehat{\CP'}(1)
	\xrightarrow{\Theta|_{\widehat{\CP'}(1)}}
	\CO(1),
\end{equation*}
where the rightmost map is surjective whenever nonzero. Therefore the vanishing of $H^{i}(Z,\widehat{\CP^{'}}(1))$ is equivalent to that of $H^{i}(Z,\widehat{\CP'_0}(1))$ since $H^{i}(Z,\CO_{Z}(1))=0$ for $i \ge 1$ by Theorem \ref{PRV0}.
On the other hand, from Proposition \ref{HX}(2), there is an exact sequence of homogeneous bundles:
\[
0  \rightarrow \widehat{\CP'_0}(1)   \xrightarrow{\widehat{\eta \circ \partial'}} \Omega_Z(1) \rightarrow Coker(\widehat{\eta \circ \partial'}) \rightarrow 0
\]
Thus the claim would follow from the vanishing of $H^{i}$ for both sides. We verify it in two cases.

Assume that
$
Z
$
is a Hermitian symmetric space. Then the map
$
\widehat{\eta \circ \partial'}
$
is an isomorphism as seen in the proof of Proposition \ref{partial0_imply_partial'}.  Moreover when $Z$ is reducible as written in Section 4.3 the Künneth formula gives:
\[
H^{i}(Z,\Omega_{Z}(1))
\cong
\bigoplus_{\nu}
H^{i}(Z_{[\nu]},\Omega_{Z_{[\nu]}}(1))
\otimes
\bigotimes_{\nu'\neq \nu}
U_{\lambda_{\nu'}}^{\vee}.
\]
Hence it suffices to consider the case where $Z$ is irreducible.
In this case,
$
\lambda=m\lambda_{i}
$
for some
$
m\geqslant 1,
$
and
\[
\Omega_{Z}(1)
=
E(\langle u_{\lambda}\rangle^{\vee}\otimes \fg_{1}),
\]
whose associated $P$-module has lowest weight
\[
\mu_{low}=-m\lambda_{i}+\alpha_{i}
=
(2-m)\lambda_{i}
+
\sum_{j\neq i} n_{ij}\lambda_{j},
\]
where
$
n_{ij}=\langle \alpha_{i},\alpha_{j}^{\vee}\rangle 
$
are the Cartan integers.
If
$
m=1,
$
then
$
\rho-\mu_{low}
$
is singular; and if
$
m\geqslant 2,
$
then
$
\rho-\mu_{low}
$
is regular and dominant.  Hence in both cases all higher cohomology groups vanish by Theorem~\ref{PRV0}. 

Now assume that
$
Z
$
is not a Hermitian symmetric space. By Proposition~\ref{partial0_surj_classify},
$
Z
$
is either an adjoint variety or a quasi-minuscule variety. In both cases, as we already discussed in Section 5.3.1, the length of the parabolic grading is equal to two, and we have
\[
\widehat{\CP'_0}(1)
=
E(\langle u_{\lambda}\rangle^{\vee}\otimes \fg_{2}),
\qquad
\Coker(\widehat{\eta\circ\partial'})
=
E(\langle u_{\lambda}\rangle^{\vee}\otimes (\fg_{1} \oplus \fg_{2})/\fg_{2}).
\]
We claim that the second bundle is an irreducible homogeneous bundle whose associated $P$-module has lowest weight $-\lambda+\alpha_i$. The required vanishing of $H^i$ then follows exactly as in the Hermitian symmetric case.
Indeed, $R_u(P)$ acts trivially on
$
(\fg_1\oplus\fg_2)/\fg_2,
$
while the induced action of $L$ identifies with its action on $\fg_1$. This representation is dual to the action of $L$ on $\fg_{-1}$, which is irreducible by~\cite[Proposition~2.4]{LM03}.

For the vanishing of
$
\widehat{\CP'_0}(1),
$
note first that in the adjoint case it is known that
$
\widehat{\CP'_0}(1)
\cong
\CO_{Z},
$
since
$
E(\fg_{2})
$
is isomorphic to the dual of the contact line bundle
$
L\cong \CO_{Z}(1)
$
(see for example~\cite{LH24}), and the vanishing follows; in the quasi-minuscule case,
$
Z
$
is isomorphic to either
$
Gr_{\omega}(2,2n)
$
($n \ge 3$)
or
$
F_{4}/P_{4}.
$
The lowest weights of the corresponding \(P\)-modules
$
\langle u_{\lambda}\rangle^{\vee}\otimes \fg_{2}
$
are of the form \(-\lambda+\alpha\), where \(\alpha\) is the lowest root in \(\fg_{2}\) with respect to the Levi subgroup. Explicitly, these roots are respectively
\[
\alpha=2\alpha_{2}+\cdots+2\alpha_{n-1}+\alpha_n,
\qquad
\alpha=\alpha_2+2\alpha_3+2\alpha_4.
\]
Hence the corresponding lowest weights are
\[
\mu_{\mathrm{low}}
=
-2\lambda_{1}+\lambda_{2},
\qquad
-\lambda_{1}+\lambda_{4}.
\]
In both cases, \(\rho-\mu_{\mathrm{low}}\) is singular. Therefore all cohomology groups vanish by Theorem~\ref{PRV0}.
\end{proof}
We now prove Proposition \ref{prop_last_verify}:
\begin{proof}[Proof of Proposition \ref{prop_last_verify}]
For (1), $Z\cong v_{r}(\BP^{n})$, in which the Levi subalgebra $\fg_{0} \cong \fg \fl_{n}$, and there is an isomorphism:
$
\phi_{\fp}|_{\fg_{0}}: \fg_{0} \simeq End(\fg/\fp).
$
Therefore the claim follows. 

For~(2), we divide the proof into the following cases.

\medskip
\noindent\textbf{Case 1.} $Z$ is isomorphic to a Hermitian symmetric space. 

\medskip
\noindent\textbf{Case 1.1.} If $Z$ is irreducible,
we need to show the condition (2) holds. 

By~\cite[Theorem~4.8]{LM03} there is a unique family of lines on $Z$, and the VMRT at $z=[u_{\lambda}]$
is given by
\[
\CC_{z}Z
\cong
L/Q
\subset
\BP T_{z}Z,
\]
where $L \subset P$ is the Levi subgroup, and $L/Q$ is the unique closed orbit of the action of $L$ on $\BP T_{z}Z.$
From the explicit description of $L/Q$ in~loc.\ cit., one checks that the action of $L$ induces a surjection
\[
\fg_{0}=\Lie(L)
\longrightarrow
H^{0}(\CC_{z}Z,T_{\CC_{z}Z}).
\]
Moreover, $Z$ admits an Euler-type $\BC^{*}$-action at $z=[u_{\lambda}],$ coming from a multiplicative subgroup of $L$. Hence $\operatorname{Im}(\phi_{\fp}) \subset \End(T_{z}Z)$
contains the scalars. Therefore from the decomposition of $\fa \fu \ft(\widehat{\CC}_{z}Z)$ as in \eqref{eulerseq_02},
$
\phi_{\fp}
$
maps
$
\fp
$
surjectively onto
$
\fl_{z},
$
and the claim follows.

\medskip
\noindent\textbf{Case 1.2.} Assume now that $Z$
is reducible. By Proposition~\ref{partial0_surj_classify}, $Z$
is isomorphic to either $\BP^{k}\times \BP^{l}$ or $\BP^{1}\times \BQ^{n-1}$
for
$
n\geqslant 2,
$
under their minimal embeddings. We need to show the condition (2) holds if and only if $Z\not=\BP^{1}\times \BQ^{1}$.

If $Z$ is neither $\BP^{1}\times \BQ^{1}$ nor $\BP^{1}\times \BQ^{2}$, then $Z$ admits two families of lines, one induced from each factor. Denote the corresponding VMRT-structures by $\CC^{i}$, $i=1,2$. Writing
$Z=Z_{1}\times Z_{2}$ and $z=(z_{1},z_{2}),$
we have
$$
\widehat{\CC}_{z}^{1}=\widehat{\CC}{z_{1}}Z_{1},
\qquad
\widehat{\CC}_{z}^{2}=
\widehat{\CC}_{z_{2}}Z_{2},
$$
under the decomposition
$
T_{z}Z
=
T_{z_{1}}Z_{1}
\oplus
T_{z_{2}}Z_{2}.
$
Since both
$\widehat{\CC}_{z_1}^{1}$
and
$\widehat{\CC}_{z_2}^{2}$
are non-degenerate, one computes directly from \eqref{fautdef} that
$$
\fl_{z}=
\{
A\in \End(T_{z}Z)
\ \mid
A|_{T_{z_{i}}Z_{i}}
\in
\fa\fu\ft(\widehat{\CC}_{z_i}Z_i),
\ i=1,2
\}
=
\fl_{z_{1}}
\oplus
\fl_{z_{2}}.
$$
On the other hand, the isotropy representation decomposes as
$
\phi_{\fp}
=\phi_{\fp_1}\oplus\phi_{\fp_2},
$
so the claim follows from the irreducible case. If
$
Z=\BP^1\times\BQ^2\simeq\BP^1\times\BP^1\times\BP^1,
$
the claim follows similarly by applying the same argument to the three factors.

If
$
Z=\BP^{1}\times\BQ^{1}=
\BP^1\times v_{2}(\BP^1),$
then there is a unique family of lines $\CK_1$, induced from the first factor. Hence
$$
\fl_{z}=
\{
A\in\End(T_{z}Z)
\ \mid
A(T_{z_{1}}Z_{1})\subset T_{z_{1}}Z_{1}
\},
$$
which also contains the off-diagonal endomorphisms
$
\Hom(T_{z_2}Z_2,T_{z_1}Z_1).
$
On the other hand, the isotropy representation
$
\phi_{\fp}
=
\phi_{\fp_1}\oplus\phi_{\fp_2}
$
acts diagonally. Therefore
$
\operatorname{Im}(\widetilde{\phi_{\fp}})
\subsetneq
\fl_{z}.
$

\medskip
\noindent\textbf{Case 2.} Now assume $Z$ is not isomorphic to a Hermitian symmetric space. Then by Proposition \ref{partial0_surj_classify}, $Z$ is isomorphic to either an adjoint variety, or $Gr_{\omega}(2,2n)$ for $n \ge 3$, or $F_4/P_4$ in their minimal embeddings. In this case $\fg$ is simple and the parabolic grading has length $d=2$. We can identify the tangent space as $\fg_{0}$-modules:
$$
T_{z}Z=\fg/\fp \cong \fg_{-1} \oplus \fg_{-2}.
$$

\medskip
\noindent\textbf{Case 2.1} Assume that $Z$ is an adjoint variety
$
Z=X_{\mathrm{ad}}(\fg)\subset \BP\fg.
$
We can further assume that $\fg$
is not of type $C_{n},$
since in this case it is projectively isomorphic to $v_{2}(\BP^{2n-1})$ which already occurs in Case 1. We need to show the condition (2) holds if and only if $Z$ is not of type $A_{n \not=2}$. The following well known facts can be found in \cite[Section 6.3]{Man13}.
\begin{lem} Keep the notation above.
	\begin{itemize}
	\item[(1)]  $\fg_{-1} \subset \fg/\fp$ is a hyperplane and the adjoint action induces an isomorphism:
\begin{equation}\label{contact}
\phi_{\fp}|_{\fg_{1}}: \fg_{1} \simeq Hom(\fg_{-2},\fg_{-1}).
\end{equation}
\item[(2)]
There exists $H$ in the Cartan subalgebra $ \ft \subset \fg_{0}$ such that $ad(H)|_{\fg_{k}}=k \cdot id_{\fg_{k}}$. In particular the span of $\widetilde{\phi_{\fp}}(\BC)$ and $ad(H)$ gives
$
\langle \widetilde{\phi_{\fp}}(\BC),  ad(H) \rangle=\BC id_{\fg_{-1}} \oplus \BC id_{\fg_{-2}}.
$ 
\end{itemize}
\end{lem}

\medskip
\noindent\textbf{Case 2.1 (a)} If $\fg$ is not of type $A_{n}$, namely $Z \not=Fl(1,n;n+1)$. Then 
from Lemma \ref{inf_vmrt_adco}(b), $Z$ admits a unique family of lines, whose VMRT at $z$ spans the hyperplane $\fg_{-1}$. Moreover from \cite[Theorem 4.8]{LM03}, the VMRT
$$\CC_{z}Z=L/Q \subset \BP \fg_{-1}$$
is the unique closed orbit of the action of $L$ on $\BP \fg_{-1}$, where $L \subset P$ is the Levi subgroup. Therefore by a direct calculation using \eqref{eulerseq_02}, we can decompose $\fl_{z}$ as $\fg_{0}$-modules:
\begin{align*}
\fl_{z}&=
\{
A\in \End(\fg_{-1} \oplus \fg_{-2})
\mid
A|_{\fg_{-1}} \in \fa \fu \ft(\widehat{L/Q}) \subset End(\fg_{-1})\}\\
&=\fa \fu \ft(\widehat{L/Q}) \oplus \Hom(\fg_{-2},\fg_{-1}) \oplus \Hom(\fg_{-2},\fg_{-2})
\end{align*}
We now show that each factor lies in the image of $\widetilde{\phi_{\fp}}$:
\begin{itemize}
\item $\fa \fu \ft(\widehat{L/Q}) \subset \phi_{\fp}(\fg_{0})|_{\fg_{-1}}$. From the explicit description of $L/Q$ in~loc.\ cit., one checks that the action of $L$ induces a surjection
$
\phi_{\fp}|_{\fg_{0}}: \fg_{0}=\Lie(L)
\longrightarrow
H^{0}(L/Q,T_{L/Q}).
$
Together with (2) in the above Lemma, the claim follows;
\item $\Hom(\fg_{-2},\fg_{-1}) \subset \phi_{\fp}(\fg_{1})$. This is \eqref{contact}.
\item $\Hom(\fg_{-2},\fg_{-2})
\subset
\widetilde{\phi}_{\fp}(\BC)\oplus\phi_{\fp}(\fg_0)$. This follows from (2) in the above Lemma.
\end{itemize}

\medskip
\noindent\textbf{Case 2.1~(b)} If $Z=Fl(1,n;n+1)$. Then from Lemma \ref{inf_vmrt_adco}(b) it admits two families of lines whose VMRTs at $z$ together span the hyperplane $\fg_{-1}$. More precisely:
$$
\fg_{-1} \cong \BC^{n-1} \oplus (\BC^{n-1})^{\vee}, \quad \CC_{z}Z=\BP(\BC^{n-1}),\quad  \CC'_{z}Z=\BP((\BC^{n-1})^{\vee}),
$$
and $\fg_{0}\cong \fg \fl_{n-1} \oplus \BC$ for $n \ge 3$, $\fg_{0}\cong \ft$ for $n=2$. 

Then a similar discussion as in (a) yields that $\widetilde{\phi_{\fp}}$ is surjective onto $\fl_{z}$ if and only if 
$$
\fl_{z}|_{\fg_{-1}}=\End(\BC^{n-1}) \oplus \End((\BC^{n-1})^{\vee}) \subset \phi_{\fp}(\fg_{0})|_{\fg_{-1}}.
$$
From the above description, this can only happen if $n=2$. In this case $\fg_{-1} \cong \BC^{2}$ and $\phi_{\fp}(\fg_{0})|_{\fg_{-1}}$ equals the diagonal matrices, hence coinciding with $\fl_{z}|_{\fg_{-1}}$.

\medskip
\noindent\textbf{Case 2.2}
Assume $Z$ is isomorphic to either $Gr_{\omega}(2,2n)$ for $n \ge 3$ or $F_4/P_4$.
Then from~\cite{LM03}, $\CC_{z}Z \subset \BP T_{z}Z$
is a non-degenerate smooth projective subvariety on which $P$
acts with an open dense orbit. Our analysis proceeds with the help of the following fact:
\begin{lem}
	If $Z$ is isomorphic to one of the varieties above, then
$$ker(\widetilde{\phi_{\fp}})=ker(\phi_{\fp})=\fg_{2} \subset \fp.$$
	In particular $\widetilde{\phi_{\fp}}$ induces an injection:
	$$
	\widetilde{\phi_{\fp}}|_{\fg_{0} \oplus \fg_{1} \oplus \BC}: \fg_{0} \oplus \fg_{1} \oplus \BC \rightarrow \fa \fu \ft(\widehat{\CC}_{z}Z) \subset \End(\fg/\fp).
	$$
\end{lem}
\begin{proof}
Assume that $(X,c) \in ker(\widetilde{\phi_{\fp}})$ in \eqref{ex_phip}. Then it follows $\phi_{\fp}(X)=-c \cdot id_{\fg/\fp}$ is a scalar multiplication on $\fg/\fp$. If \(c=0\), then $X \in \fg_{2}$ from Lemma \ref{ker_isotropy}; for $c \not=0$, then as in the proof of Lemma \ref{iml}, the semisimple part of $X$ induces an Euler-type $\BC^{*}$-action on $Z$. Since $Z$ is homogeneous, this implies $Z$ is Euler-symmetric in the sense of Definition \ref{euler_symdefn}, and therefore $Z$ is isomorphic to a Hermitian symmetric space, contradicting our assumption. 
\end{proof}
\medskip
\noindent\textbf{Case 2.2 (a)}  $Z=Gr_{\omega}(2,2n)$.  We need to show the condition (2) holds if and only if $n=3$. 

In this case, a precise description of the isotropy representation
$
\phi_{\fp}
$
together with
$
\fa \fu \ft(\widehat{\CC}_{z}Z)
$
can be found in~\cite{LH21}. The claim follows from the description that we summarize below.

\begin{lem}
Let $U$ be a two-dimensional vector space, $(Q,w)$ be a vector space equipped with a non-degenerate skew form, and denote by $\fs\fp(Q)$ the symplectic Lie algebra. Set
$$W=U\otimes Q\oplus\Sym^{2}U.
$$
	\begin{itemize}
		\item[(a)](\cite[Corollary~5.5]{LM03}, \cite[Proposition~4.14]{LH21}) There exists an isomorphism
		$
		\varphi:T_{z}Z\cong W
		$
		such that
        the isotropy action of
		$
		\fg_{0}
		$
		on
		$
		T_{z}Z
		$
		coincides with the natural action of
		$
		\fg \fl(U)\oplus \fs \fp(Q)
		$
		on
		$
		W.
		$
		Moreover,
		\[
		\fg_{1}\cong \Hom(U,Q),
		\qquad
		\fg_{2}\cong \Sym^{2}U^{\vee}
		\]
		as
		$\fg_{0}$-modules. We still denote by
		$
		\phi_{\fp}:\fp\to \End(W)
		$
		the induced action of
		$
		\fp
		$
		on
		$
		W.
		$
		
		\item[(b)](\cite[Lemma~6.2]{LH21})
		The Lie algebra
		$
		\fa \fu \ft(\widehat{\CC}_{z}Z)
		$
		is equal to the direct sum of
		$$
		\BC\,\mathrm{id}_{W} \oplus \fg \fl(U) \oplus \fs\fl(Q)  \oplus \Hom(U,Q),
		$$
		where the action of $\fg \fl(U) \oplus \fs\fl(Q)$ is by their representation on $W$, and the action of $Hom(U,Q)$ is the same as the isotropy action of $\phi_{\fp}(\fg_{1})$.
	\end{itemize}
\end{lem}
Together with Lemma 5.20 one sees that $Im(\widetilde{\phi_{\fp}})$ coincides with $\fl_{z}=\fa \fu \ft(\widehat{\CC}_{z}Z)$ if and only if $\fs \fp(Q)=\fs \fl(Q)$. This happens exactly when $n=3$, where $Q$ is two-dimensional.

\medskip
\noindent\textbf{Case 2.2 (b)}
Assume that $Z=F_{4}/P_{4}$. We need to show that condition~(2) holds.

In this case, $\CC_{z}Z$ is a smooth hyperplane section of $\BS_{5}$,
which is a horospherical variety of Picard number one whose automorphism
group was determined in~\cite[Theorem~1.11]{Pas09}.

\begin{lem}
Keep the notation above.
\begin{itemize}
\item[(a)] (\cite[Proposition~6.5]{LM03})
The Levi subalgebra and the positive graded pieces are given by
$$
\fg_{0}\cong\fs\fo_{7}\oplus\BC,
\qquad
\fg_{1}\cong U_{8},
\qquad
\fg_{2}\cong V_{7},
$$
as $\fs\fo_{7}$-modules, where $V_{7}$ is the standard vector
representation of $\fs\fo_{7}$ and $U_{8}$ is its spin representation.
	\item[(b)] (\cite[Theorem~1.11]{Pas09})
	The Lie algebra of the automorphism group of $\CC_{z}Z$ decomposes,
	as an $\fs\fo_{7}$-module, as
	\[
	H^{0}(\CC_{z}Z,T_{\CC_{z}Z})
	=
	\BC\oplus\fs\fo_{7}\oplus U_{8}.
	\]
\end{itemize}
\end{lem}
Together with Lemma~5.20, it follows that
$Im(\widetilde{\phi_{\fp}})$ and
$\fl_{z}=\fa\fu\ft(\widehat{\CC}_{z}Z)$ have the same dimension.
Therefore they are equal.
\end{proof}
Next we verify that $\zeta \circ \partial$ is non-surjective provided that $\widetilde{\phi_{\fp}}$ is not onto $\fl_{z}$ as classified above. For the case $Gr_{\omega}(2,2n)$ see also Example \ref{nonhomospin} for an alternative proof.
\begin{lem}
	Assume that
	$
	Z
	$
	is isomorphic to either
	$
	Gr_{\omega}(2,2n)
	$
	for
	$
	n\geqslant 4
	$
	or
	$
	\BP^{1}\times v_{2}(\BP^{1}).
	$
	Then the map
	$
	\zeta \circ \partial
	$
	in~\eqref{diag_Spencer} is not surjective.
\end{lem}

\begin{proof}
	In both cases,
	$
	Z
	$
	is covered by a unique family of lines. By Lemmas~5.15 and~5.16, it suffices to show that the inclusion
	\[
	H^{0}\bigl(
	Z,
	E(\operatorname{Im}(\widetilde{\phi_{\fp}})\otimes \langle u_{\lambda}\rangle^{\vee})
	\bigr)
	\subset
	H^{0}\bigl(
	Z,
	E(\fl_{z}\otimes \langle u_{\lambda}\rangle^{\vee})
	\bigr)
	\]
	is strict, where
	$
	\operatorname{Im}(\widetilde{\phi_{\fp}})
	\subset
	\fl_{z}=\fa \fu \ft(\widehat{\CC}_{z}Z)
	$
	are subalgebras of
	$
	\End(\fg/\fp).
	$
	
	If
	$
	Z=\BP^{1}\times v_{2}(\BP^{1}).
	$
	Writing
	$
	z=(z_{1},z_{2}),
	$
	we have
	$
	\widehat{\CC}_{z}Z
	=
	\widehat{\CC}_{z_{1}}(\BP^{1})
	\subset
	T_{z_{1}}Z_{1},
	$
	under the decomposition
	$
	T_{z}Z
	=
	T_{z_{1}}Z_{1}
	\oplus
	T_{z_{2}}Z_{2}.
	$
	Hence, as
	$
	\fp
	$
	-modules, there is a direct sum of
	\[
	\fl_{z}
	=
	\{
	A\in \End(T_{z}Z)
	\mid
	A(T_{z_{1}}Z_{1})
	\subset
	T_{z_{1}}Z_{1}
	\}
	=
	\End(T_{z_{1}}Z_{1})
	\oplus
	\Hom(T_{z_{2}}Z_{2},T_{z}Z),
	\]
	while 
	$
	\operatorname{Im}(\widetilde{\phi_{\fp}})
	=
	\End(T_{z_{1}}Z_{1})
	\oplus
	\End(T_{z_{2}}Z_{2}).
	$
	Therefore we have as $\fp$-modules: 
	\[
	\fl_{z}= \Im(\widetilde{\phi_{\fp}})  \oplus 
	\Hom(T_{z_{2}}Z_{2},T_{z_{1}}Z_{1}).
	\]
	For the latter factor, the corresponding homogeneous bundle is
	$$
	E(\Hom(T_{z_{2}}Z_{2},T_{z_{1}}Z_{1}) \otimes  \langle u_{\lambda} \rangle^{\vee} )=T_{\BP^{1}}(1)\boxtimes \Omega_{v_{2}(\BP^{1})}(1)
	\cong
	T_{\BP^{1}}(1)\boxtimes \CO_{\BP^{1}},
	$$
	whose space of global sections is
	$
	H^{0}(Z,T_{\BP^{1}}(1)\boxtimes \CO_{\BP^{1}})
	\cong
	H^{0}(\BP^{1},T_{\BP^{1}}(1))
	\neq 0.
	$ The claim for this case then follows.
    
 If $Z=Gr_{\omega}(2,2n)$ for $n\geq4$, then by Lemmas~5.20 and~5.21, there is an exact sequence of $P$-modules
$$
0\rightarrow \Im(\widetilde{\phi_{\fp}})
\rightarrow \fl_z
\rightarrow \fs\fl(Q)/\fs\fp(Q)\cong\wedge^2_0Q\rightarrow0,
$$
where $\wedge^2_0Q$ is the nontrivial irreducible $\fs\fp(Q)$-submodule of $\wedge^2Q$. We claim that
$$
H^1\left(
Z,
E\left(
\Im(\widetilde{\phi_{\fp}})
\otimes
\langle u_\lambda\rangle^\vee
\right)
\right)=0,
\qquad
H^0\left(
Z,
E\left(
\wedge^2_0Q
\otimes
\langle u_\lambda\rangle^\vee
\right)
\right)
\not=
0.
$$
The non-surjectivity then follows from the associated long exact sequence.

For the first vanishing, Lemma~5.20 gives a $P$-equivariant filtration of $\Im(\widetilde{\phi_{\fp}})$ whose successive quotients are
$\BC,\fg_0,\fg_1,$
respectively, on which the unipotent radical of $P$ acts trivially. After tensoring with $\langle u_\lambda\rangle^\vee,$
the first quotient corresponds to $\CO_Z(1)$, while the vanishing for the third quotient was proved in Lemma~5.18. It remains to consider $\fg_0$. Write
$$
\fg_0=\BC \oplus \fg_0^1 \oplus \fg_0^2,\qquad \fg_0^1 \cong \fs\fl_2, \qquad \fg_0^2\cong \fs\fp_{2n-4}.
$$
The central summand again corresponds to $\CO_Z(1)$. For $i=1,2$, the lowest weight of
$
\fg_0^i\otimes\langle u_\lambda\rangle^\vee
$
is
$
-\lambda-\beta^i,
$
where $\beta^i$ is the highest root of $\fg_0^i$. More precisely,
$$
\beta^1
=\alpha_1=
2\lambda_1-\lambda_2,
\qquad
\beta^2
=2\alpha_3+2\alpha_4+\cdots+2\alpha_{n-1}+\alpha_n=-2\lambda_2+2\lambda_3.
$$
Since $\lambda=\lambda_2$, the corresponding lowest weights are
$
\mu_{\mathrm{low}}
=
-2\lambda_1$ and $
\lambda_2-2\lambda_3
$ respectively.
In the first case, $\rho-\mu_{\mathrm{low}}$ is regular and dominant, while in the second case it is singular. Therefore, in both cases, the first cohomology vanishes by Theorem~\ref{PRV0}. It follows from the filtration that the first space vanishes.

Finally,
$
\wedge^2_0Q\otimes\langle u_\lambda\rangle^\vee
$
is an irreducible $P$-module of lowest weight
$
-\lambda_4,$ hence $\rho-\mu_{\mathrm{low}}$ is regular and dominant. Then Theorem~\ref{PRV0} gives the non-vanishing of the second space.
\end{proof}
We can now complete the proof of Proposition~\ref{main_432} and Theorem \ref{main_thm}.

\begin{proof}[Proof of Proposition~\ref{main_432}]
Assume that $Z\subset \BP U$ is a flag variety such that the maps $\partial^0$ and $\partial'$
are surjective. 
By Lemma~\ref{van_ome}, we have
$
H^1(Z,\Omega_Z(1))=0.
$
Hence, from the sequence~\eqref{iq}, the map
$
q_T
$
in~\eqref{refine_diag_Spencer} is surjective. In particular, $\zeta$ is surjective, and the first claim follows directly from Definition~\ref{filt_01}.

For the second claim, assume first that $Z$ is one of the varieties listed in Theorem~\ref{main_thm}. Then Corollary~\ref{suff_condition_surj}, together with Proposition~\ref{prop_last_verify}, implies that
$
\zeta\circ\partial
$
is surjective. Conversely, suppose that $Z$ does not appear in the list of Theorem~\ref{main_thm}. By Proposition~\ref{partial0_surj_classify}, $Z$ is then isomorphic to one of the following varieties:
an adjoint variety $X_{\mathrm{ad}}(\fg)\subset \BP\fg$ with $\fg$ of type $A_l$ $(l\neq 2)$;
$
Gr_{\omega}(2,2n)
$
with
$
n\geq 4;
$
or
$
\BP^1\times v_2(\BP^1).
$
In the adjoint case, the non-surjectivity of $\partial$ was proved in~\cite[Example 8.8]{LH24}; in the remaining cases, $\zeta\circ\partial$ is not surjective by Lemma~5.23. This completes the proof.
\end{proof}

\begin{proof}[Proof of Theorem~\ref{main_thm}]
If $Z=\BP U$, then
$
\Xi_Z=\Hom(\wedge^2U,U),
$
and the Spencer map is simply the skew-symmetrization map
\[
\partial:\Hom(U,\End(U))\longrightarrow \Hom(\wedge^2U,U),
\qquad
\partial h(u,v)=h(u)(v)-h(v)(u),
\]
which is clearly surjective. For $Z\neq \BP U$, the varieties satisfying condition~\textup{(iii)} were classified in Proposition~\ref{main_432}.
\end{proof}
\section{Isotrivial VMRT-structures of highest weight type}

The main geometric applications of Theorem \ref{main_thm} concern isotrivial VMRT-structures of highest weight type. In this section, following \cite[Section~8]{LH24}, we study such structures and derive consequences.

\begin{defn}
	Let \(X\) be a uniruled projective manifold, let \(\CK\) be a family of minimal rational curves, and let
	\(\CC\subset \BP TX\)
	be the associated VMRT-structure (Definition~\ref{vmrt_defn}). Let
	\(Z\subset \BP U\)
	be a smooth projective variety with
	\(\dim U=\dim X\).	
	We say that \(\CC\) is \emph{\(Z\)-isotrivial} if, over some Zariski-open subset
	\(M\subset X\),
	the restricted cone structure
	\(\CC|_M\)
	is \(Z\)-isotrivial. If \(Z\) is a flag variety, we say that \(\CC\) is \emph{isotrivial of highest weight type}.
\end{defn}

We recall two examples that are relevant to our classification. The first example produces a locally flat structure (Definition 2.2).

\begin{example}[{\cite[Example~8.3]{LH24}}]\label{flatmodel}
	Let
	\(\BP^{n-1}\subset \BP^n\)
	be a hyperplane and let
	\(Z\subset \BP^{n-1}\)
	be a smooth projective subvariety.
	Let
	\(X_Z\)
	be the blow-up of
	\(\BP^n\)
	along
	\(Z\),
	and let
	\(\CK_Z\)
	be the family of minimal rational curves whose general members are the proper transforms of lines meeting \(Z\).
	Then the associated VMRT-structure, restricted to the open subset corresponding to
	\(\BP^n\setminus \BP^{n-1}\),
	is a locally flat \(Z\)-isotrivial cone structure.
\end{example}
The following construction provides locally homogeneous but
non-locally-flat structures arising from smooth projective
symmetric varieties of Picard number one classified by~\cite{Ruz11}.

\begin{example}\label{sym_example} With the notation of Definition~1.2, let $X\subset \BP V $ be a non-homogeneous smooth projective symmetric variety of Picard number one in its minimal embedding. By~\cite{Ruz11}, there are six such varieties. They are covered by lines, and their geometry has been studied in \cite{Ruz10,Man18,Man21,KP19,CFL23}. Their geometric descriptions and VMRTs are listed in Table~3; see Proposition~6.4 for details. \end{example}
	\begin{table}[H]
	\centering
	\renewcommand{\arraystretch}{1.25}
	\caption{Non-homogeneous smooth projective symmetric varieties of Picard number one}
	\label{Table3}
	\begin{tabular}{|c|c|c|c|}
		\hline
		$X$ & $\widetilde{G}/G$ & geometric description & VMRT $Z \subset \BP U$ \\
		\hline
		$X_1$ 
		& $E_6/F_4$ 
		& hyperplane section of $E_7/P_7$ 
		& $F_4/P_4$ \\
		\hline
		$X_2$ 
		& $SL_6/Sp_6$ 
		& hyperplane section of $\BS_6$ 
		& $Gr_{\omega}(2,6)$ \\
		\hline
		$X_3$ 
		& $(SL_2\times SL_2)/SL_2$ 
		& hyperplane section of $Gr(3,6)$ 
		& $A_2$-adjoint variety \\
		\hline
		$X_4$ 
		& $SL_3/SO_3$ 
		& hyperplane section of $Lag(3,6)$ 
		& $v_4(\BP^1)$ \\
		\hline
		$X_5$ 
		& $G_2/(SL_2\times SL_2)$ 
		& Cayley Grassmannian 
		& $\BP^1\times v_3(\BP^1)$ \\
		\hline
		$X_6$ 
		& $(G_2\times G_2)/G_2$ 
		& double Cayley Grassmannian 
		& $G_2$-adjoint variety \\
		\hline
	\end{tabular}
\end{table}
\begin{prop}
	Let \(X\subset \BP V\) be a non-homogeneous smooth projective symmetric
	variety of Picard number one. Fix a \(\widetilde G\)-equivariant open
	embedding
	$
	M=\widetilde G/G\subset X,
	$
	and let \(x=[G]\in M\) be the base point. Then
	$
	T_xX\simeq \Lie(\widetilde G)/\Lie(G)
	$
	is an irreducible \(G\)-module, and the following hold:
	\begin{itemize}
		\item[(1)]
		The variety \(X\) is covered by lines in \(\BP V\). Let \(\CK\) be the
		corresponding family of minimal rational curves and
		\(\CC\subset \BP TX\) its VMRT-structure. Then the VMRT \(\CC_x\) is
		isomorphic to the unique closed \(G\)-orbit in \(\BP T_xX\). Denote this closed orbit by \(Z\simeq\CC_x\).
		
		\item[(2)]
		The restriction \(\CC|_M\) of the VMRT-structure to the open orbit
		\(M=\widetilde G/G\) is \(Z\)-isotrivial but not locally flat.
	\end{itemize}
\end{prop}

\begin{proof}
	By~\cite[Theorem  1.1]{Ruz10}, in all six cases $G=G^\nu$ is the connected fixed subgroup, and its isotropy action on $T_xX$ is irreducible. The cases $X_3$ and $X_6$ are of adjoint type: the embedding
$
G\subset\widetilde G=G\times G
$
is diagonal, and the isotropy action is the adjoint representation of $G$. For the remaining four cases, the corresponding highest weights are given in \cite[Section 8]{Wol84}. 

The assertion (1) for the first four varieties follows from
\cite[Proposition~2.8]{CFL23}, while the case $X_6$ follows from
\cite[Proposition~19]{Man21}. For $X_5$, using its realization in~\cite{Man18} as the zero locus of a general section in
$
H^0(Gr(4,7),\wedge^3\CU_4^\vee),
$
where $\CU_4$ is the tautological rank-four subbundle, a calculation on the lines of $Gr(4,7)$ together with a dimension count shows that $X_5$ is covered by lines. Hence its VMRT at $x$ is a $G$-invariant smooth projective surface; see~\cite[Remark 4.2]{KP19}. By \cite[Theorem~1.5]{HL26}, it is also the closure of a $G$-orbit. Since it contains the unique closed $G$-orbit in $\BP T_xX$, which has the same dimension, the two coincide.

Finally, since $M$ is a $\widetilde G$-orbit, the restricted VMRT-structure
$
\CC|_M
$
is $Z$-isotrivial. If $\mathcal C|_M$ were locally flat, then
\cite[Proposition~6.13]{FH12} and
\cite[Proposition~4.5]{BFM20} would give
$$
\aut(X)\simeq
\BC^n\rtimes\aut(\widehat Z)
\oplus\aut(\widehat Z)^{(1)},
$$
where $n=\dim(X)$ and $\aut(\widehat Z)^{(1)}$ is the first prolongation of
$
\aut(\widehat Z);
$
see Corollary~\ref{cor_classify}(ii). On the other hand,
$
\aut(X)=\Lie(\widetilde G)
$
by~\cite{Ruz10}, while
$
\aut(\widehat Z)=
\BC\mathrm{id}_U\oplus\Lie(G).$
Moreover,
$
\aut(\widehat Z)^{(1)}=0
$
by the Kobayashi–Nagano classification; see~\cite[Proposition~4.3]{LH24}. Taking dimensions then contradicts
$
n=\dim(\widetilde G)-\dim(G).
$
\end{proof}
For the applications using Theorem \ref{appp_homoge}, we need the formal curvature spaces \eqref{form_curv} of the VMRTs listed in Table~\ref{Table3}, most of which were computed in~\cite{ML99,LH24}.

\begin{prop}\label{form_symm_one}
	Let $Z\subset \BP U$ be one of the varieties listed in the last column of Table~\ref{Table3}, and let
	$
	\hfg=\fg\oplus \BC\,\id_U,
	$
	where $\fg=\Lie(Aut(Z))$.
	If $Z$ is not isomorphic to $v_4(\BP^1)$ or $\BP^1\times v_3(\BP^1)$, then
	$
	\dim \BK(\hfg)=1.
	$
\end{prop}

\begin{proof}
The varieties in the third and sixth rows are adjoint varieties of types
$A_2$ and $G_2$, respectively, so the claim follows from
\cite[Theorem~1.6]{LH24}. For the remaining cases, it is known that each symmetric pair $(\widetilde G,G)$ determines
an element of
$
\BK(\fg)\subset \BK(\hfg)
$ (cf. \cite[Theorem 3.2 (1)]{KN69}). Indeed, writing the eigenspace
decomposition of
\(
\widetilde{\fg}=\Lie(\widetilde G)
\)
with respect to the involution $\nu$ as
$
\widetilde{\fg}
=
\fg\oplus\fg_{-1},
$
we can identify $\fg_{-1}$ with $U$. Since
$
[\fg_{-1},\fg_{-1}]\subset \fg,
$
the map
\[
R:\wedge^2U\longrightarrow \fg,
\qquad
R(X,Y)=[X,Y],
\]
defines an element of \(\BK(\fg)\). Moreover, in both cases \(R\neq 0\), since \(\fg_{-1}\) is not abelian. 
Hence
$
\dim\BK(\hfg)\ge1.
$
On the other hand, the inequality
$
\dim\BK(\hfg)\le1
$
is proved in
\cite[Proof of Theorem~6.10, Types $C_n$(ii) and $F_4$(iv)]{ML99}
for $Gr_{\omega}(2,6)$ and $F_4/P_4$, respectively.
\end{proof}
We obtain the following result, analogous to~\cite[Theorem~8.9]{LH24}. 

\begin{cor}\label{cor_vmrt}
	Let \(Z\subset \BP U\) be one of the VMRTs listed in Table \ref{Table3}, and assume that $Z$ is not isomorphic to either
	\(v_4(\BP^1)\subset \BP(S^4\BC^2)\) or $\BP^1 \times v_{3}(\BP^1)$. Then every $Z$-isotrivial cone structure admitting a characteristic conic connection is locally symmetric, and hence locally homogeneous.
	
	Moreover let \(X\) be a smooth projective variety of dimension \(\dim U\) equipped with a family \(\CK\) of minimal rational curves such that, over some open subset \(M\subset X\), the restricted VMRT-structure
	\(\CC|_M\)
	is \(Z\)-isotrivial. Then:
	\begin{itemize}
		\item[(1)]
		The VMRT-structure \(\CC|_M\) is locally symmetric. In particular, it is locally homogeneous.
		
		\item[(2)]
		A neighborhood of a general member of \(\CK\) is biholomorphic to a neighborhood of a general member of the family of minimal rational curves given in Examples 6.2 or 6.3.
		
		\item[(3)]
		If \(b_2(X)=1\), then \(X\) is quasi-homogeneous, i.e. \(\Aut(X)\) has an open orbit.
	\end{itemize}
\end{cor}

\begin{proof}
		The variety $Z\subset\BP U$ satisfies conditions~\textup{(i)--(iii)}
		of Theorem~\ref{tors_prin}: condition~\textup{(i)} and~\textup{(ii)} from
		Proposition~\ref{surj_Psi_}, and condition~\textup{(iii)} from
		Theorem~\ref{main_thm}. Moreover, $Z$ is not among the exceptional
		varieties in Theorem~\ref{appp_homoge}: these are listed in  \cite[Propositions 4.3 and 4.5]{LH24}, which are precisely the varieties listed in Theorem \ref{class_tang}(1)  and Proposition~\ref{partial0_surj_classify}\textup{(2)}. 
		Hence every $Z$-isotrivial cone structure admitting a characteristic
		conic connection is locally symmetric by
		Theorem~\ref{appp_homoge}\textup{(1)}.
		
		Now let $\CC|_M$ be the VMRT-structure associated with $\CK$.
		By Proposition~\ref{vmrtischar}, it admits a characteristic conic
		connection, so \textup{(1)} follows from the preceding
		paragraph. (2) follows from  Theorem~\ref{appp_homoge}\textup{(2)}, Proposition \ref{form_symm_one}, Theorem~\ref{HM01} and Examples 6.2 and 6.3. 
		Finally, \textup{(3)} follows from \textup{(1)}
		and Theorem~\ref{HM01}.
\end{proof}
Next we recall an example of isotrivial VMRT-structure of highest weight type which is non-locally symmetric. By an argument
similar to~\cite[Example~8.8]{LH24}, it gives an alternative proof
for the failure of the Spencer condition for
$
Z=\mathrm{Gr}_{\omega}(2,2m)$ with $m\ge 4,$
as established in Lemma~5.23.
\begin{example}[{\cite{BFM20}}]\label{nonhomospin}
	Let
	$
	\BS_n\subset \BP\Delta_+
	$
	be the spinor variety associated with $Spin_{2n}$, where \(n\ge 5\), and let
	\(X\subset \BP V\) be a general smooth hyperplane section. Then \(X\) is a Fano manifold of Picard number one covered by lines in $\BP V$, and its VMRT at a general point is isomorphic to the symplectic Grassmannian
	$
	Z=Gr_{\omega}(2,n)
	$
    embedded in its minimal embedding.
	In particular, \(Z\) is a flag variety when \(n=2m\) is even, in which case the associated VMRT-structure on $X$ is $Z$-isotrivial over a suitable open subset.
	By~\cite[Theorem~1.2]{BFM20}, \(X\) is quasi-homogeneous if and only if \(n\le 7\). Hence, for \(n=2m\ge 8\), its VMRT-structure cannot be locally symmetric by Theorem~\ref{HM01}. Since $Gr_{\omega}(2,2m)$ satisfies conditions~(i) and~(ii) of Theorem~\ref{tors_prin} and does not belong to the list excluded in Theorem~\ref{appp_homoge}, the Spencer condition must fail.
\end{example}

\appendix \label{appendix}
\section{}
We record the SageMath computation used in Lemma~3.31, carried out with
SageMath~10.5. We use the Bourbaki numbering of the Dynkin diagram of type
\(E_6\), and write \(U_{\lambda_i}\) for the irreducible \(E_6\)-module of
highest weight \(\lambda_i\). The following SageMath session computes the
irreducible decomposition of \(\wedge^2 U_{\lambda_3}\).

\begin{verbatim}
	sage: E6 = WeylCharacterRing("E6", style="coroots")
	sage: E63 = E6(0,0,1,0,0,0)
	sage: E63.exterior_power(2)
	E6(0,0,1,0,0,1) + E6(1,1,0,0,0,0)+ E6(1,0,0,1,0,0) + E6(0,0,0,0,1,0)
\end{verbatim}

Thus
\[
\wedge^2 U_{\lambda_3}
\simeq
U_{\lambda_3+\lambda_6}
\oplus
U_{\lambda_1+\lambda_2}
\oplus
U_{\lambda_1+\lambda_4}
\oplus
U_{\lambda_5}.
\]
Dualizing this decomposition gives the first decomposition in
Lemma~3.31.
\bibliographystyle{amsalpha-abbrv}
\bibliography{product}
\end{document}